\documentclass[12pt]{amsart}
\usepackage{amsrefs}

\usepackage[dvipdfmx, bookmarkstype=toc=true, linktocpage=true, bookmarks=true]{hyperref}
\usepackage{color}
\usepackage{amssymb}
\usepackage[all]{xy}
\newif\ifdebug

\debugfalse

\newcommand{\printname}[1]{\smash{\makebox[0pt]{\hspace{-1.0in}\raisebox{8pt}{\tiny #1}}}}

\newcommand{\Label}[1]{\ifdebug{\label{#1}\printname{#1}}\else{\label{#1}}\fi}

\theoremstyle{plain}
\newtheorem{theo}{Theorem}[section]
\newtheorem{lemm}[theo]{Lemma}
\newtheorem{coro}[theo]{Corollary}
\newtheorem{prop}[theo]{Proposition}

\theoremstyle{remark}
\newtheorem{rema}[theo]{Remark}

\theoremstyle{definition}
\newtheorem{defn}[theo]{Definition}
\newtheorem{exam}[theo]{Example}

\def\C{\mathbb C}
\def\Z{\mathbb Z}
\def\R{\mathbb R}

\def\co{\colon\thinspace}
\DeclareMathOperator{\Hom}{\mathrm{Hom}}

\DeclareMathOperator{\pos}{\mathrm{pos}}

\DeclareMathOperator{\id}{\mathrm{id}}
\DeclareMathOperator{\im}{\mathrm{im}}
\DeclareMathOperator{\Aut}{\mathrm{Aut}}
\DeclareMathOperator{\Ad}{\mathrm{Ad}}
\DeclareMathOperator{\forB}{for_B}

\def\g{\mathfrak{g}}
\def\h{\mathfrak{h}}
\def\F{\mathcal{F}}

\makeatletter

\@addtoreset{equation}{section}
\makeatother

\begin{document}
\title[Towards transverse toric geometry]{Towards transverse toric geometry}

\author[H.~Ishida]{Hiroaki Ishida}
\address{Department of Mathematics and Computer Science, Graduate School of Science and Engineering, Kagoshima University}
\email{ishida@sci.kagoshima-u.ac.jp}

\date{\today}
\thanks{The author is supported by JSPS Grant-in-Aid for Young Scientists (B) 16K17596}
\keywords{complex manifold, toric variety, transverse K\"ahler form, basic cohomology, basic Dolbeault cohomology}
\subjclass[2010]{Primary 14M25, Secondary 53C12, 53D20, 57S25, 57R30}

\begin{abstract}
	We study certain foliated complex manifolds that behave similarly to complete nonsingular toric varieties. We classify them by combinatorial objects that we call marked fans. We describe the basic cohomology algebras of them in terms of corresponding marked fans. We also study the basic Dolbeault cohomology algebras of them. 
\end{abstract}

\maketitle 

\section{Introduction}\Label{sec:intro}
	A toric variety is a normal algebraic variety over the complex numbers $\C$ with an effective action of an algebraic torus having an open dense orbit. On the other hand, a fan is a collection of cones in a real vector space with the origin as vertex satisfying certain conditions. A fundamental theorem in toric geometry states that the category of toric varieties is equivalent to the category of rational fans. 
	
	In this paper, we study certain foliated manifolds that behave similarly to complete nonsingular toric varieties. As a combinatorial counterpart to such a foliated manifold, we introduce the notion of \emph{marked fan}. A marked fan is a fan in a finite dimensional vector space equipped with information of generators of $1$-cones. 

	The main concerns in this paper are the followings. 
	\begin{enumerate}
		\item A compact connected manifold $M$ equipped with a maximal action of a compact torus $G$. 
		\item The canonical foliation $F$ on a compact connected complex manifold $M$. 
	\end{enumerate}
	
	We say that an effective action of a compact torus $G$ on a smooth connected manifold $M$ is \emph{maximal} if there exists a point $x\in M$ such that $\dim G+ \dim G_x =\dim M$, where $G_x$ denotes the isotropy subgroup at $x\in M$ of $G$. All compact connected complex manifolds equipped with maximal actions of comact tori are completely classified in \cite{Ishida}. There are rich examples of such manifolds. One can see that compact complex tori, complete nonsingular toric varieties, Hopf manifolds and Calabi-Eckmann manifolds admit maximal actions of compact tori. More interesting examples are LVM manifolds, LVMB manifolds and moment-angle manifolds with complex structures. LVM manifolds are non-K\"ahler compact complex manifolds, constructed in \cite{Meersseman} and the acronyms stand for L\'opez de Medrano-Verjovsky \cite{LdM-Verjovsky} and Meersseman. LVMB manifolds are generalizations of LVM manifolds, constructed by Bosio in \cite{Bosio}. LVMB manifolds are obtained as the quotients of complements of coordinate subspace arrangements in projective spaces $\C P^n$ by free, proper, and holomorphic $\C^m$-actions. They are naturally equipped with actions of compact tori which are maximal. Moment-angle manifolds are topological manifolds equipped with actions of compact tori, and constructed from simplicial complexes. In \cite{Panov-Ustinovsky} and \cite{Tambour}, it is shown that even dimensional moment-angle manifolds coming from star-shaped spheres carry invariant complex structures. 
		
	The canonical foliation $F$ on a complex manifold $M$ is a holomorphic foliation generated by a local free action of a certain subgroup $H'$ of the group $\Aut (M)$ of all biholomorphisms on $M$. For more precise, let $\g$ be the Lie algebra of a maximal compact torus $G$ of $\Aut  (M)$ and $J$ the complex structure of $\Aut (M)$. The Lie subalgebra $\h' := \g \cap J\g$ is a complex subalgebra. One can see that the corresponding Lie group $H':= \exp (\h') \subset \Aut (M)$ acts on $M$ local freely and does not depend on the choice of the maximal compact torus $G$. Each leaf of the canonical foliation $F$ is an $H'$-orbit and vice versa. The canonical foliations on LVMB manifolds have been studied in \cite{Battaglia-Zaffran} from a different viewpoint. In \cite{Battaglia-Zaffran}, under some assumptions, it has been shown that the basic betti numbers of an LVMB manifold with respect to the canonical foliation coincides with the $h$-vector of a certain simplicial sphere. It also has been shown that the basic cohomology algebras of certain LVMB manifolds are generated by degree $2$ elements, as well as the cohomology algebras of complete nonsingular toric varieties. 

	If the action of a compact torus $G$ on a compact connected complex manifold $M$ is maximal and preserves the complex structure on $M$, then the torus $G$ can be regarded as a maximal torus of the group $\Aut (M)$ of all biholomorphisms of $M$. 
	\begin{exam}[Complete nonsingular toric varieties]
		Let $M$ be a complete nonsingular toric variety of complex dimension $n$. Let $G^\C$ be the algebraic torus of complex dimension $n$ acting on $M$. Let $G$ be the maximal compact torus of $G^\C$. $G$ is a maximal compact torus of $\Aut (M)$. Since the action of $G^\C$ is effective, we have that $\h' = \g \cap J\g = \{0\}$. Therefore, every leaf of the canonical foliation $F$ on $M$ is a point. The leaf space $M/F$ is a complete nonsingular toric variety $M$ itself. 
	\end{exam}
	\begin{exam}[Compact complex tori]
		Let $\Gamma$ be a lattice of $\C^n$. Let $G$ be the compact complex torus torus $\C^n/\Gamma$ of real dimension $2n$. Let $M = G$. We define the action of $G$ on $M$ is given as the group operation. Then the action of $G$ on $M$ is maximal and preserves the complex structure. For a fundamental vector field $X_v$ on $M$ generated by an element $v \in \C^n$, we have that $JX_v$ is the fundamental vector field $X_{\sqrt{-1}v}$. Therefore we have $\h' = \g \cap J\g = \g$. Thus the canonical foliation $F$ on $M$ has exactly one leaf $M$. The leaf space $M/F$ is a point and it can be regarded as the toric variety of dimension $0$. 
	\end{exam}
	\begin{exam}[Hopf surfaces]
		Let $\alpha_1,\alpha_2 \in \C$ such that $1<|\alpha_1|\leq |\alpha_2|$. Define an action of $\Z$ on $\C^2 \setminus \{0\}$ to be $k\cdot (z_1,z_2) := (\alpha_1^kz_1,\alpha_2^kz_2)$ for $k \in \Z$, $(z_1,z_2)\in \C^2 \setminus \{0\}$. The quotient manifold $M := \C^2 \setminus \{0\}/\Z$ is a complex manifold that is diffeomorphic to $S^3 \times S^1$. We equip the action of $(\R /\Z)^3 =:G$ on $M$ as follows. Let $\widetilde{\alpha}_1, \widetilde{\alpha}_2 \in \C$ such that $e^{2\pi\sqrt{-1}\widetilde\alpha_i} = \alpha_i$ for $i=1,2$. We define an action of $\R^3$ on $M$ by
		\begin{equation*}
			(t_1,t_2,s) \cdot [z_1,z_2] := [e^{2\pi\sqrt{-1}(t_1 + \widetilde{\alpha}_1s)}z_1, e^{2\pi\sqrt{-1}(t_2+\widetilde{\alpha}_2s)}z_2]
		\end{equation*}
		for $(t_1,t_2,s) \in \R^3$ and $[z_1,z_2] \in M$, 
		where $[z_1,z_2]$ denotes the equivalence class of $(z_1,z_2) \in \C^2\setminus \{0\}$. The action of $\R^3$ above on $M$ preserves the complex structure on $M$ and it has the global stabilizers $\Z^3$. Thus the action of $\R^3$ descends to an action of the compact torus $G$ that preserves the complex structure on $M$. One can see that the orbit through $[1,0]$ has dimension $2$. Thus, the action of $G$ on $M$ is maximal. By definition of $\h'$, we have that $\h' \subset \g = \R^3$ has the basis vectors $v_1 = (\operatorname{Re}(\widetilde{\alpha}_1), \operatorname{Re}(\widetilde{\alpha}_2), -1)$ and $v_2 = (\operatorname{Im}(\widetilde{\alpha}_1), \operatorname{Im}(\widetilde{\alpha}_2), 0)$. On the other hand, $\h' = \{ v \in \R^3 \mid  v \cdot (v_1\times v_2) = 0\}$. Thus $\h'$ is a Lie algebra of a subtorus of $G$ if and only if the fractions of entries of $v_1 \times v_2$ are rational. One can see that the fractions of entries of $v_1 \times v_2$ are rational if and only if there exist positive integers $n_1,n_2$ such that $\alpha_1^{n_1} = \alpha_2^{n_2}$. In other words, if there are no positive integers $n_1,n_2$ such that $\alpha_1^{n_1} = \alpha_2^{n_2}$, then the canonical foliation $F$ on $M$ has a leaf that is not closed. The leaf space $M/F$ is not Hausdorff unless $\alpha_1^{n_1}= \alpha_2^{n_2}$ for some positive integers $n_1$ and $n_2$. 
		
	\end{exam}
	There are $3$ main results in this paper. The first main result states that complex manifolds with maximal torus actions (up to an equivalence relation which is slightly stronger than the notion of transversally equivalence on foliated manifolds) can be described by the corresponding marked fans. The second main result states that the basic cohomology with respect to the canonical foliation of a complex manifold with a maximal torus action admits the same formula as complete nonsingular toric varieties. To show this, we introduce the \emph{basic forgetful map} $\forB \co H^*_G(M) \to H^*_B(M)$, where $H^*_G(M)$ denotes the equivariant cohomology of the $G$-manifold $M$. We will see that $\forB$ is surjective and describe the kernel of $\forB$. The third main result states that the basic Dolbeault cohomology group $H^{p,q}_B(M)$ of our space concentrates to the center. Namely, we will see that $H^{p,q}_B(M) = 0$ if $p\neq q$. To see this complex geometric result, we use Minkowsky sum of polytopes and localization of equivariant coholomogies.

	This paper is organized as follows. Section \ref{sec:preliminaries1} is preliminaries. In Subsection \ref{subsec:canonical}, we review the canonical foliations. In Subsection \ref{subsec:maximal}, we briefly review the equivalence of categories between $\mathcal{C}_1$ and $\mathcal{C}_2$, where $\mathcal{C}_1$ is the category of complex manifolds with maximal torus actions, and $\mathcal{C}_2$ is the combinatorial counter part of $\mathcal{C}_1$. In Section \ref{sec:ete}, we study the notion of transverse equivalence in $\mathcal{C}_1$. In Section \ref{sec:fund}, we introduce marked fans and show the first main result. Section \ref{sec:preliminaries2} is preliminaries for the second and third main results.  In Subsection \ref{subsec:minimal}, we review the complex structure of a tubular neighborhood of a minimal orbit. In Subsection \ref{subsec:transverse}, we briefly review transverse K\"ahler forms and  moment maps. In Section \ref{sec:cohomology}, we study the additive structure of the basic cohomology and the relation between the basic cohomology and the equivariant cohomology through the basic forgetful map. In Section \ref{sec:DJ}, we show the second main result. We describe the basic cohomology algebra in terms of the corresponding marked fan. In Section \ref{sec:basicDolbeault}, we show the third main result. 
	
\section{Preliminaries 1}\Label{sec:preliminaries1}
\subsection{Canonical foliations}\Label{subsec:canonical}
	Let $M$ be a compact connected complex manifold and $G$ a maximal compact torus of the group $\Aut (M)$ of all automorphisms on $M$. $\Aut (M)$ is a complex Lie group and its Lie algebra can be identified with the vector space $\mathfrak{X}(M)$ of all holomorphic vector fields on $M$ (see \cite{Bochner-Montgomery} for detail). Denote by $\g$ the Lie algebra of $G$ and  by $J$ the complex structure on $M$.  Put $\g^\C := \g\otimes_\R \C$ and denote by $v+\sqrt{-1}v'$ instead of $v\otimes 1+v'\otimes \sqrt{-1}$ for $v,v' \in \g$ for short. We define the $\C$-subspace $\h$ of $\g^\C$ by 
	\begin{equation}\Label{eq:h}
		\h := \{ v+\sqrt{-1}v' \in \g^\C \mid X_v+JX_{v'} = 0 \},
	\end{equation}
	where $X_v$ denotes the fundamental vector field generated by $v \in \g$. 
	Let $p \co \g^\C \to \g$ be the first projection $(\g\otimes 1) \oplus (\g \otimes \sqrt{-1}) \to \g$. We think of $\g$ as a real Lie subalgebra of $\mathfrak{X}(M)$ via $v \mapsto X_v$. Then, $p(\h)$ is a commutative subalgebra of $\mathfrak{X}(M)$ invariant under $J$ because $p(\h) = \g \cap J \g$. For $v \in p(\h)$ and $x \in M$, we denote by $(X_v)_x$ the value of the fundamental vector field $X_v$ at $x \in M$. 
	\begin{prop}\Label{prop:localfree}
		For $v \in p(\h)$ and $x \in M$, $(X_v)_x = 0$ if and only if $X_v=0$. 
	\end{prop}
	\begin{proof}
		Suppose that $(X_v)_x=0$. Let $\varphi_s$ be the partial flow of $X_v$ at time $s$ and $\psi_t$ the partial flow of $JX_v$ at time $t$. Denote by $T$ the closure of the complex Lie group $H' := \{ \varphi_s \psi_t \mid s, t \in \R\}$. Since $p(\h) \subseteq \g$, we have that $T$ is a subtorus of $G$. Since the action of $H'$ fixes $x \in M$, the action of $T$ also fixes $x \in M$. We consider the isotropy representation $T_xM$ of $T$. Since $T$ is compact, we have that $T_xM$ is a unitary representation of $T$. In particular, $T_xM$ is a unitary representation of $H'$. Since $H'$ is a complex Lie group and the action of $H'$ on $M$ is holomorphic, we have that $T_xM$ is a holomorphic representation of $H'$. For a connected complex Lie group, any holomorphic unitary representation is trivial. Therefore $T_xM$ is a trivial representation of $H'$. Since $T$ is a closure of $H'$, we have that $T_xM$ is a trivial representation of $T$. This together with the connectedness of $M$ yields that $T$-fixes all $y \in M$ by slice theorem. Since $T$ is a subtorus of $\Aut (M)$, the action of $T$ on $M$ is effective. Therefore $T$ is the trivial subgroup of $\Aut (M)$. Hence $H'$ is trivial. Thus we have that if $(X_v)_x = 0$ for $v \in p(\h)$ then $X_v=0$. This shows the ``if" part. The ``only if" part is obvious. The proposition is proved. 
	\end{proof}
	\begin{prop}\Label{prop:central}
		 The followings hold.
		 \begin{enumerate}
		 	\item Elements in $p(\h)$ centralize $\mathfrak{X}(M)$. 
			\item $p(\h)$ does not depend on the choice of a maximal compact torus $G$ of $\Aut (M)$. 
		 \end{enumerate}
	\end{prop}
	\begin{proof}
		Let $\exp_G \co \g \to G$ denote the exponential map. 
		We define the subgroup $H'$ of $G$ by $H' := \exp_G(p(\h))$. Since $G$ is compact, $\mathfrak{X}(M)$ is a unitary representation of $G$. In particular, $\mathfrak{X}(M)$ is a unitary representation of $H'$. Since $H'$ is a holomorphic subgroup of $\Aut (M)$, we have that $\mathfrak{X}(M)$ is a holomorphic representation of $H'$. Therefore $\mathfrak{X}(M)$ is a trivial representation of $H'$, showing Part (1). 
		
		Let $G'$ be another maximal compact torus of $\Aut (M)$ and $\g'$ the Lie algebra of $\g'$. As before, we think of $\g$ and $\g'$ as real Lie subalgebras of $\mathfrak{X}(M)$. Since $\Aut (M)$ is a Lie group, there exists $g \in \Aut (M)$ such that $gGg^{-1} = G'$ (see \cite[Chapter XV, Section 3]{Hochschild}). In particular we have $\Ad _g(\g) = \g'$. Since $\Ad_g$ is $\C$-linear, we have that $J(\Ad _g(\g)) = \Ad _g(J\g)$. Since $p(\h) = \g\cap J\g$, we have that $\Ad _g (p(\h)) = p'(\h')$, where $p' \co {\g'} ^\C \to \g'$ denotes the projection. On the other hand, by (1), we have that $\Ad _g$ is the identity on $p(\h)$. Therefore $p(\h)$ does not depend on the choice of $G$, proving Part (2). 
	\end{proof}
	By Proposition \ref{prop:localfree}, we have a holomorphic foliation $F$ on $M$ whose leaves are generated by $p(\h)$. The definition of $F$ is intrinsic by Proposition \ref{prop:central}. We call $F$ \emph{the canonical foliation} on $M$. 

\subsection{Complex manifolds with maximal torus actions}\Label{subsec:maximal}
	In this subsection, we briefly recall the classification of complex manifolds with maximal torus actions given in \cite{Ishida} for reader's convenience. 
	Let $M$ be a connected smooth manifold equipped with an effective action of a compact torus $G$. We say that the $G$-action on $M$ is \emph{maximal} if there exists a point $x \in M$ such that $\dim
	G+\dim G_x = \dim M$. 
	If the action of $G$ on $M$ is maximal, then we can think of $G$ as a maximal compact torus of the group of diffeomorphisms on $M$ (see \cite[Lemma 2.2]{Ishida}). 
	 
	Let $\mathcal{C}_1$ denote the class that consists of all triples $(M,G,y)$ satisfying the followings.
	\begin{enumerate}
	 	\item $M$ is a compact connected complex manifold.
		\item $G$ is a compact torus acting on $M$. The $G$-action on $M$ is maximal and preserves the complex structure on $M$. In particular, $G$ is a maximal compact torus of $\Aut (M)$. 
		\item $y \in M$ satisfies that $G_y = \{1\}$. 
	\end{enumerate}
	For $(M_1,G_1,y_1), (M_2,G_2,y_2) \in \mathcal{C}_1$, the hom-set $\Hom _{\mathcal{C}_1} ((M_1,G_1,y_1), (M_2,G_2,y_2))$ is defined to be the set of pairs $(f, \alpha)$ satisfying the followings.
	\begin{enumerate}
	 	\item $\alpha \co G_1 \to G_2$ is a smooth homomorphism. 
		\item $f$ is an $\alpha$-equivariant holomorphic map. Namely, for all $x \in M_1$ and all $g\in G_1$, we have $f(g\cdot x) = \alpha(g)\cdot f(x)$. 
		\item $f(y_1) = y_2$. 
	\end{enumerate}
	Then the class $\mathcal{C}_1$ and the hom-sets $\Hom_{\mathcal{C}_1}(-,-)$ form a category. 
	 
	As a combinatorial counterpart of $\mathcal{C}_1$, we consider the followings. Let $\mathcal{C}_2$ denote the class that consists of triples $(\Delta, \h, G)$ satisfying the followings.
	\begin{enumerate}
	 	\item $G$ is a compact torus. 
		\item Let $\g$ be the Lie algebra of $G$ and $\exp_G \co \g \to G$ the exponential map. We think of $\ker \exp_G\subset \g$ a lattice of $\g$. Then, $\Delta$ is a nonsingular fan in $\g$ with respect to the lattice $\ker \exp_G$. 
		\item As before, we denote by $\g^\C = \g \otimes_\R \C$ and decompose $\g^\C$ as $\g^\C = (\g \otimes 1) \oplus (\g \otimes \sqrt{-1})$. Let $p \co (\g \otimes 1) \oplus (\g \otimes \sqrt{-1}) \to \g$ be the first projection. Then the restriction $p|_{\h} \co \h \to \g$ of $p$ to $\h$ is injective.
		\item Let $q \co \g \to \g/p(\h)$ be the quotient map. Then 
		\begin{equation*}
			q(\Delta) := \{ q(\sigma) \subset \g/p(\h) \mid \sigma \in \Delta\}
		\end{equation*}
		is a complete fan and the map $\Delta \to q(\Delta)$ given by $\sigma \mapsto q(\sigma)$ is bijective. 
	 \end{enumerate}
	For $(\Delta_1,\h_1,G_1), (\Delta_2,\h_2,G_2) \in \mathcal{C}_2$, the hom-set $\Hom _{\mathcal{C}_2} ((\Delta_1,\h_1,G_1), (\Delta_2,\h_2,G_2))$ is defined to be the set of all smooth homomorphism $\alpha \co G_1 \to G_2$ satisfying the followings.
	\begin{enumerate}
	  	\item Let $\g_1$ and $\g_2$ denote the Lie algebras of $G_1$ and $G_2$, respectively. The differential $d\alpha \co \g_1 \to \g_2$ induces a morphism of fans $\Delta_1$ and $\Delta_2$. Namely, for any $\sigma_1 \in \Delta_1$, there exists $\sigma_2 \in \Delta_2$ such that $d\alpha(\sigma_1) \subset \sigma_2$. 
		\item The linear map $d\alpha^\C := d\alpha \otimes \id_\C \co \g_1 ^\C \to \g_2^\C$ satisfies that $d\alpha^\C (\h_1) \subset \h_2$. 
	\end{enumerate}
	Then the class $\mathcal{C}_2$ and the hom-sets $\Hom_{\mathcal{C}_2}(-,-)$ form a category. 
	
	To $(M,G,y) \in \mathcal{C}_1$, we can assign $(\Delta, \h, G)=:\F_1(M,G,y) \in \mathcal{C}_2$ as follows. A \emph{characteristic submanifold} of $M$ is a complex codimension $1$ connected component of the fixed point set by the action restricted to a $1$-dimensional subtorus of $G$. If there is no characteristic submanifold (in this case, the action of $G$ on $M$ is free and simply transitive), then $\Delta$ is the fan that consists of the $0$-dimensional cone $\{0\}$ in $\g$ only. Let $M_1,\dots, M_k$ be characteristic submanifolds of $M$. Let $G_i$ denote the $1$-dimensional subtorus of $G$ that fixes $M_i$ pointwise. Since the codimension of $M_i$ is $1$, the normal bundle $TM|_{M_i}/TM_i$ is a complex line bundle on $M_i$. For any $x \in M_i$, the normal space $(TM|_{M_i}/TM_i)_x$ is a representation of $G_i$. Since $G_i$ fixes $M_i$ pointwise and the action is effective, we have that $(TM|_{M_i}/TM_i)_x$ is a faithful $G_i$-representation. We denote by $\mu_i \co G_i \to S^1$ the character of the representation $(TM|_{M_i}/TM_i)_x$ of $G_i$. Since $\mu_i$ is a faithful representation, there exists the inverse $\lambda_i \co S^1 \to G_i$ of $\mu_i$. Since each $G_i$ is a subtorus of $G$, we obtain an element $\lambda_i \in \Hom (S^1,G)$ for each $i$. 
	We identify $\Hom(S^1,G)$ with $\ker \exp_G$ as follows. First, we have a smooth isomorphism $\psi \co \R/\Z \to S^1$ given by $[t] \to e^{2\pi\sqrt{-1}t}$ for $[t] \in \R/\Z$.  
	For $\gamma \in \ker \exp_G$, we have a smooth homomorphism $\lambda_\gamma \co \R/\Z \to G$ given by $[t] \mapsto \exp_G (t\gamma)$. The composition $\lambda_\gamma \circ\psi^{-1}$ is an element in $\Hom (S^1,G)$ and we have an isomorphism $\ker \exp_G \to \Hom (S^1,G)$ given by $\gamma \mapsto \lambda_\gamma \circ\psi^{-1}$. Through this isomorphism, we identify $\Hom (S^1,G)$ with $\ker \exp_G$. As a result, we have elements $\lambda_1,\dots, \lambda_k \in \g$. The fan $\Delta$ is given by 
	\begin{equation*}
		\Delta := \left\{\pos (\lambda_i \mid i \in I) \mid \bigcap_{i \in I} M_i \neq \emptyset, I \subset \{1,\dots, k\} \right\}, 
	\end{equation*}
	where $\pos(\lambda_i \mid i \in I)$ denotes the cone spanned by $\lambda_i$, $i \in I$. $\h$ is given as \eqref{eq:h}. Then one can see that $(\Delta, \h, G) = \F_1(M,G,y) \in \mathcal{C}_2$. For $(M_1,G_1,y_1), (M_2,G_2,y_2) \in \mathcal{C}_1$ and a morphism $(f,\alpha) \in \Hom_{\mathcal{C}_1}( (M_1,G_1,y_1), (M_2,G_2,y_2))$, one can see that $\alpha =:\F_1(f,\alpha) \in \Hom_{\mathcal{C}_2}(\F_1(M_1,G_1,y_1), \F_2(M_2,G_2,y_2))$. As a result, we have a covariant functor $\F_1 \co \mathcal{C}_1 \to \mathcal{C}_2$. 
	
	Conversely, we construct an object $(M,G,y) =: \F_2(\Delta,\h,G)$ in $\mathcal{C}_1$ from $(\Delta,\h,G) \in \mathcal{C}_2$ as follows. Let $X(\Delta)$ denote the toric variety associated with $\Delta$ and $G^\C$ the complexification of the compact torus $G$. Let $\exp_{G^\C} \co \g^\C \to G^\C$ denote the exponential map. The quotient $M$ of $X(\Delta)$ by $H := \exp_{G^\C}(\h) \subset G^\C$ is a compact complex manifold and the action of $G$ on $X(\Delta)$ descends to the maximal action on $X(\Delta)/H$. Let $\pi \co X(\Delta)\to M$ be the quotient map. Then we put $y:= \pi(1_{G^\C})$, where $1_{G^\C}$ is the unit of $G^\C \subset X(\Delta)$. Let $(\Delta_1,\h_1,G_1), (\Delta_2,\h_2,G_2) \in \mathcal{C}_2$. For a morphism $\alpha \in \Hom ((\Delta_1,\h_1,G_1), (\Delta_2,\h_2,G_2))$, since $d\alpha$ is a morphism of fans $\Delta_1$ and $\Delta_2$, we have an $\alpha$-equivariant holomorphic map (toric morphism) $\varphi_\alpha \co X(\Delta_1) \to X(\Delta_2)$ such that $\varphi_{\alpha}(1_{G_1^\C}) = 1_{G_2^\C}$. $\varphi_\alpha$ descends to an $\alpha$-equivariant holomorphic map $f_\alpha \co X(\Delta_1)/H_1 \to X(\Delta_2)/H_2$. Namely, $(f_\alpha, \alpha) =: \F_2(\alpha) \in \Hom (\F_2(\Delta_1,\h_1,G_1),\F_2(\Delta_2,\h_2,G_2))$. As a result, we have a covariant functor $\F_2 \co \mathcal{C}_2 \to \mathcal{C}_1$. 
	
	It has been shown that the functors $\F_1$ and $\F_2$ are (weak) inverses to each other. In particular, $\mathcal{C}_1$ and $\mathcal{C}_2$ are equivalent as categories. 
	
\section{Transversally equivalence}\Label{sec:ete}
	Let $(M_1,F_1)$ and $(M_2,F_2)$ be smooth manifolds with foliations $F_1$ on $M_1$ and $F_2$ on $M_2$. We say that $(M_1,F_1)$ and $(M_2,F_2)$ are \emph{transversally equivalent} if there exist a foliated manifold $(M_0,F_0)$ and a surjective submersion $f _i \co M_0 \to M_i$ for $i=1,2$ such that 
	\begin{enumerate}
		\item $f_i^{-1}(x_i)$ is connected for all $x_i \in M_i$ and
		\item the preimage of every leaf of $F_i$ by $f_i$ is a leaf of $F_0$
	\end{enumerate}
	(see \cite[Definition 2.1]{Molino} for detail). 
	The notion of transversally equivalence is an equivalence relation on foliated manifolds. We restrict our attention to the category $\mathcal{C}_1$ and canonical foliations. 
	\begin{defn}\Label{defn:equivalent}
		Let $(M_1,G_1,y_1), (M_2,G_2,y_2) \in \mathcal{C}_1$. We say that $(M_1,G_1,y_1)$ and $(M_2,G_2,y_2)$ are \emph{equivalent} if there exist $(M_0,G_0,y_0) \in \mathcal{C}_1$ and morphisms $(f_i, \alpha_i) \in \Hom_{\mathcal{C}_1}((M_0,G_0,y_0), (M_i,G_i,y_i))$ for $i=1,2$ satisfying the followings.
		\begin{enumerate}
			\item $\ker \alpha_i$ is connected. 
			\item $f_i \co M_0 \to M_i$ is a principal $\ker \alpha_i$-bundle. 
		\end{enumerate}
	\end{defn}
	We remark that $(M_i,G_i,y_i)$ and $(M_0,G_0,y_0)$ are also equivalent in Definition \ref{defn:equivalent}. Definition \ref{defn:equivalent} determines an equivalence relation on $\mathcal{C}_1$. 
	
	\begin{prop}\Label{prop:equivalence}
		Let $(M,G,y), (M_0,G_0,y_0) \in \mathcal{C}_1$. Let $F$ and $F_0$ be the canonical foliation on $M$ and $M_0$, respectively. Let a morphism 
		\begin{equation*}
			(f,\alpha) \in \Hom_{\mathcal{C}_1}((M_0,G_0,y_0),(M,G,y))
		\end{equation*} satisfy the followings.
		\begin{enumerate}
			\item $\ker \alpha$ is connected. 
			\item $f \co M_0 \to M$ is a principal $\ker \alpha$-bundle. 
		\end{enumerate}
		Then, 
		\begin{enumerate}
			\item $f^{-1}(x)$ is connected for all $x \in M$ and
		\item the preimage of every leaf of $F$ by $f$ is a leaf of $F_0$.
		\end{enumerate}
	\end{prop}
	Before the proof of Proposition \ref{prop:equivalence}, we recall the following: 
	\begin{theo}[{\cite[Theorem 11.1]{Ishida}}]\Label{theo:principal}
		Let $(M,G,y), (M_0,G_0,y_0) \in \mathcal{C}_1$. Let $(\Delta,\h,G) := \F_1(M,G,y)$ and $(\Delta_0,\h_0, G_0) := \F_1(M_0,G_0,y_0)$. Let 
		\begin{equation*}
			(f,\alpha) \in \Hom_{\mathcal{C}_1}((M,G,y), (M_0,G_0,y_0)).
		\end{equation*}
		Let $\Delta^{(1)}$ and ${\Delta}_0^{(1)}$ denote the set of $1$-cones in $\Delta$ and $\Delta_0$, respectively. Then, $f \co M \to M_0$ is a principal $\ker \alpha$-bundle if and only if $\alpha$ is surjective and $d\alpha \co \g \to \g_0$ induces a one-to-one correspondence from the primitive generators of $1$-cones in $\Delta$ to the primitive generators of $1$-cones in $\Delta'$. 
	\end{theo}
	
	\begin{proof}[proof of Proposition \ref{prop:equivalence}]
		Since $\ker \alpha$ is connected and $f \co M_0 \to M$ is a principal $\ker \alpha$-bundle, we have that $f^{-1}(x)$ is connected for all $x \in M$. 
	
		Let $p_0 \co \g_0^\C \to \g_0$ and $p \co \g^\C \to \g$ be the projections. Let $q_0 \co \g_0 \to \g_0/p_0(\h_0)$ and $q \co \g \to \g/p(\h)$ be the quotient maps. 
		Since $d\alpha^\C (\h_0) \subset \h$ and $p\circ d\alpha^\C = d\alpha  \circ p_0$, we have $d\alpha(p_0(\h_0)) \subset p(\h)$. 
		Thus $d\alpha \co \g_0 \to \g$ induces a linear map $\overline{d\alpha} \co \g_0/p_0(\h_0) \to \g/p(\h)$.
		We show that $\overline{d\alpha}$ is an isomorphism.
		Since $\alpha$ is surjective, so is $d\alpha$. Since $d\alpha$ is surjective, so is $\overline{d\alpha}$.
		Let $\overline{v} \in \g_0/p_0(\h_0)$ be a nonzero element. Since $q_0(\Delta_0)$ is complete, we have that there exists $\sigma \in \Delta_0$ such that $\overline{v} \in q_0(\sigma)$. Let $v \in \sigma$ such that $q_0(v) = \overline{v}$. 
		Let $\lambda_1,\dots, \lambda_k$ be the primitive generators of $\sigma$. Since $\sigma$ is nonsingular, $\lambda_1,\dots, \lambda_k$ are linearly independent. 
		Therefore there uniquely exist $c_1,\dots, c_k \in \R$ such that $v = \sum_j c_j\lambda_j$ and one of $c_1,\dots, c_k$ is nonzero.
		By Theorem \ref{theo:principal}, each $d\alpha(\lambda_j)$ is a primitive generator of $1$-cone in $\Delta$.
		Since $d\alpha$ induces a morphism of fans $\Delta_0$ and $\Delta$, $d\alpha(\sigma)$ should be a cone in $\Delta$ and   $d\alpha(\lambda_1),\dots, d\alpha(\lambda_k)$ are primitive generators of $d\alpha(\sigma)$. 
		Therefore $d\alpha(v) = \sum_jc_jd\alpha(\lambda_j)$ is nonzero, because $d\alpha(\lambda_1),\dots, d\alpha(\lambda_k)$ are linearly independent and one of $c_1,\dots, c_k$ is nonzero. 
		Thus $d\alpha(v) \in d\alpha(\sigma) \setminus \{0\}$. Since $q$ gives a one-to-one correspondence between $\Delta$ and $q(\Delta)$, it follows from $d\alpha(v) \in d\alpha(\sigma) \setminus \{0\}$  that $q\circ d\alpha (v) \neq 0$. It turns out that $\overline{d\alpha}(\overline{v}) \neq 0$. Therefore $\overline{d\alpha}$ is injective.
		
		We  show that $d\alpha^{-1}(p(\h)) = p_0(\h_0)$. Since $d\alpha(p_0(\h_0)) \subset p(\h)$, we have $d\alpha^{-1}(p(\h)) \supset p_0(\h_0)$. To show the opposite inclusion, let $v \in d\alpha^{-1}(p(\h))$. Then $d\alpha(v)\in p(\h)$. Thus $q\circ d\alpha(v) = 0$. On the other hand, we have $q\circ d\alpha(v) = \overline{d\alpha}\circ q_0(v)$. Therefore $\overline{d\alpha}\circ q_0(v)=0$. Since $\overline{d\alpha}$ is injective, we have $q_0(v)=0$. This yields that $v \in p_0(\h_0)$. Therefore $d\alpha^{-1}(p(\h)) = p_0(\h_0)$. 
		
		Let $L$ be a leaf of $F$. Since each leaf of $F$ is an orbit with respect to the action of $G$ restricted to $\exp_{G}(p(\h)) =: H'$, we have $L = H' \cdot x$ for some $x \in L$.  Since $f \co M_0 \to M$ is a principal $\ker \alpha$-bundle, there exits $x_0 \in M_0$ such that $f(x_0) = x$. The leaf through $x_0$ is the orbit of the action of $G_0$ restricted to $\exp_{G_0}(p_0(\h_0)) =: H_0'$. We show that $f^{-1}(L) = H_0'\cdot x_0$. Let $x_0' \in H_0' \cdot x_0$. Then there exists $h_0' \in H_0'$ such that $x_0' =h_0' \cdot x_0$. Since $f$ is $\alpha$-equivariant, we have  $f(x_0') = \alpha(h_0') \cdot f(x_0) =  \alpha(h_0') \cdot x$. Since $d\alpha^{-1}(p(\h)) = p_0(\h_0)$, we have $\alpha(h_0') \in H'$ and hence $f(x_0') \in L$. Thus we have the inclusion $f^{-1}(L) \supset H_0'\cdot x_0$. To show the opposite inclusion, let $x_0'' \in f^{-1}(L)$. Then $f(x_0'') \in L$. Thus there exists $h' \in H'$ such that $x = h' \cdot f(x_0'')$. Since $\ker \alpha$ is connected and $d\alpha^{-1}(p(\h)) = p_0(\h_0)$, there exists $\widetilde{h_0'} \in H_0'$ such that $\alpha(\widetilde{h_0'}) = h'$. Then we have $f(\widetilde{h_0'}\cdot x_0'') = h \cdot f(x_0'') = x = f(x_0)$. Since $f$ is a principal $\ker \alpha$-bundle, there exists $k \in \ker \alpha$ such that $x_0 = k\cdot (\widetilde{h_0'}\cdot x_0'')$. Since $\ker \alpha$ is connected and $d\alpha^{-1}(p(\h)) = p_0(\h_0)$, we have $\ker\alpha \subset H_0'$. Therefore $k\cdot \widetilde{h_0'} \in H_0'$. Since $x_0 = k\cdot (\widetilde{h_0'}\cdot x_0'') = (k\cdot \widetilde{h_0'}) \cdot x_0''$, we have  $x_0'' = (k\cdot \widetilde{h_0'})^{-1} \cdot x_0 \in H_0' \cdot x_0$. Therefore $f^{-1}(L) =H_0'\cdot x_0$. It turns out that the preimage of every leaf of $F$ by $f$ is a leaf of $F_0$. The theorem is proved. 
	\end{proof}
	\begin{theo}\Label{theo:equivalence}
		Let $(M_1,G_1,y_1), (M_2,G_2,y_2) \in \mathcal{C}_1$. Let $F_1$ and $F_2$ be the canonical foliation on $M_1$ and $M_2$, respectively. If $(M_1,G_1,y_1)$ and $(M_2,G_2,y_2)$ are equivalent, then $(M_1,F_1)$ and $(M_2,F_2)$ are transversally equivalent. 
	\end{theo}
	\begin{proof}
		It follows from Proposition \ref{prop:equivalence} immediately. 
	\end{proof}
\section{The fundamental theorems}\Label{sec:fund}
	\begin{defn}
		A \emph{marked fan} is a quadruple $(\widetilde{V}, \widetilde{\Gamma}, \widetilde{\Delta}, \widetilde{\lambda})$ satisfying the followings.
		\begin{enumerate}
			\item $\widetilde{V}$ is a finite dimensional $\R$-vector space. 
			\item $\widetilde{\Gamma}$ is a finitely generated subgroup of $\widetilde{V}$ that spans $\widetilde{V}$ linearly. 
			\item $\widetilde{\Delta}$ is a fan and each cone in $\widetilde\Delta$ is generated by finite elements in $\widetilde{\Gamma}$. 
			\item Let $\widetilde{\Delta}^{(1)}$ denote the set of all $1$-cones in $\widetilde{\Delta}$. $\widetilde{\lambda}$ is a function $\widetilde{\lambda} \co \widetilde{\Delta}^{(1)} \to \widetilde{\Gamma}$ such that $\widetilde{\lambda} (\rho)$ is a generator of $\rho$ for $1$-cone $\rho \in \widetilde{\Delta}$. 
		\end{enumerate}
		Moreover if the fan $\widetilde{\Delta}$ is simplicial, we say that the marked fan $(\widetilde{V}, \widetilde{\Gamma}, \widetilde{\Delta}, \widetilde{\lambda})$ is simplicial. If $\widetilde{\Delta}$ is complete, we say that the marked fan $(\widetilde{V}, \widetilde{\Gamma}, \widetilde{\Delta}, \widetilde{\lambda})$ is complete. We denote by $\widetilde{\mathcal{C}}_2$ the class that consists of complete simplicial marked fans. 
	\end{defn}
	\begin{defn}
		We say that marked fans $(\widetilde{V}_1,\widetilde{\Gamma}_1,\widetilde{\Delta}_1,\widetilde{\lambda}_1)$ and $(\widetilde{V}_2,\widetilde{\Gamma}_2,\widetilde{\Delta}_2,\widetilde{\lambda}_2)$ are \emph{isomorphic} if there exists a linear isomorphism $\varphi \co \widetilde{V}_1\to \widetilde{V}_2$ that satisfies the followings.
		\begin{enumerate}
			\item $\varphi (\widetilde{\Gamma}_1) = \widetilde{\Gamma}_2$.
			\item $\varphi$ induces an isomorphism of fans $\widetilde{\Delta}_1$ and $\widetilde{\Delta}_2$. Namely, $\varphi(\widetilde{\sigma}_1) \in \widetilde{\Delta}_2$ for all $\widetilde{\sigma}_1 \in \widetilde{\Delta}_1$ and the map $\widetilde{\Delta}_1 \to \widetilde{\Delta}_2$ given by $\widetilde{\sigma}_1 \mapsto \varphi(\widetilde{\sigma}_1)$ is bijective. We denote the bijection $\widetilde{\Delta}_1\to \widetilde{\Delta}_2$ by the same symbol $\varphi$.
			\item $\widetilde{\lambda}_2 \circ \varphi|_{\widetilde{\Delta}_1^{(1)}} = \varphi \circ \widetilde{\lambda}_1$. 
		\end{enumerate}
	\end{defn}
	To each $(M,G,y) \in \mathcal{C}_1$, we can assign a complete simplicial marked fan $(\widetilde{V}, \widetilde{\Gamma}, \widetilde{\Delta}, \widetilde{\lambda}) \in \widetilde{\mathcal{C}}_2$ as follows: Let $(\Delta,\h,G) := \F_1(M,G,y)$. As before, we denote by $p \co \g^\C \to \g$ the projection and $q\co \g \to \g/p(\h)$ the quotient map. For each $1$-cone $\rho \in \Delta^{(1)}$, we denote by $\lambda(\rho) \in \ker \exp_G$ the primitive generator of $\rho$. Under this notation, 
	\begin{enumerate}
		\item $\widetilde{V} := \g/p(\h)$. 
		\item $\widetilde{\Gamma} := q(\ker \exp_G)$. 
		\item $\widetilde{\Delta} := q(\Delta)$. 
		\item $\widetilde{\lambda}(q(\rho)) := q(\lambda(\rho))$ for $\rho \in \Delta^{(1)}$. 
	\end{enumerate}
	We denote by $\widetilde{\mathcal{F}}_1 \co \mathcal{C}_1 \to \widetilde{\mathcal{C}}_2$ the assignment above. 
	Theorems \ref{theo:fundamental} and \ref{theo:ess.surjective} below tell us that $\widetilde{\mathcal{F}}_1$ gives an interpretation between $\mathcal{C}_1$ and $\widetilde{\mathcal{C}}_2$. 
	\begin{theo}\Label{theo:fundamental}
		Let $(M_1,G_1,y_1),(M_2,G_2,y_2) \in \mathcal{C}_1$. Then, $(M_1,G_1,y_1)$ and $(M_2,G_2,y_2)$ are equivalent if and only if $\widetilde{\mathcal{F}}_1(M_1,G_1,y_1)$ and $\widetilde{\mathcal{F}}_1(M_2,G_2,y_2)$ are isomorphic. 
	\end{theo}
	\begin{theo}\Label{theo:ess.surjective}
		$\widetilde{\mathcal{F}}_1$ is essentially surjective. Namely, for any $(\widetilde{V}, \widetilde{\Gamma}, \widetilde{\Delta}, \widetilde{\lambda}) \in \widetilde{\mathcal{C}}_2$, there exists $(M,G,y) \in \mathcal{C}_1$ such that $(\widetilde{V}, \widetilde{\Gamma}, \widetilde{\Delta}, \widetilde{\lambda})$ and  $\widetilde{\mathcal{F}}_1(M,G,y)$ are isomorphic. 
	\end{theo}
	\begin{proof}[Proof of Theorem \ref{theo:fundamental}]
		Suppose that $(M_1,G_1,y_1),(M_2,G_2,y_2) \in \mathcal{C}_1$ are equivalent. Then there exist $(M_0,G_0,y_0) \in \mathcal{C}_1$ and $(f_i,\alpha_i) \in \Hom _{\mathcal{C}_1}((M_0,G_0,y_0),(M_i,G_i,y_i))$ for $i=1,2$ such that 
		\begin{enumerate}
			\item $\ker \alpha_i$ is connected and 
			\item $f_i \co M_0 \to M_i$ is a principal $\ker \alpha_i$-bundle. 
		\end{enumerate}
		We show that $\widetilde{\F}_1(M_0,G_0,y_0)$ and $\widetilde{\F}_1(M_i,G_i,y_i)$ are isomorphic. 
		
		For $j=0,1,2$, let $(\Delta_j,\h_j,G_j) := \F_1 (M_j,G_j,y_j)$. Let $p_j \co \g_j^\C \to \g_j$ be the projection, $q_j \co \g_j \to \g_j/p_j(\h_j)$ the quotient map and $\exp_{G_j} \co \g_j \to G_j$ the exponential map. We have that $\overline{d\alpha_i} \co \g_0/p_0(\h_0) \to \g_i/p_i(\h_i)$  is an isomorphism (see the proof of Proposition \ref{prop:equivalence}). By Theorem \ref{theo:principal}, we have that $\alpha_i$ is surjective. Since $\ker \alpha_i$ is connected and $\alpha_i$ is surjective, we have $d\alpha_i(\ker \exp_{G_0}) = \ker \exp_{G_i}$. Since $q_i \circ d\alpha_i = \overline{d\alpha_i}\circ q_0$, we have $\overline{d\alpha_i} (q_0(\ker \exp_{G_0})) = q_i(\ker \exp_{G_i})$. Let $\sigma_0 \in \Delta_0$. Since $d\alpha_i$ induces a morphism of fans $\Delta_0$ and $\Delta_i$ and one-to-one correspondence between primitive generators by Theorem \ref{theo:principal}, we have $d\alpha_i(\sigma_0) \in \Delta_i$ and $d\alpha_i$ induces a bijection $\Delta_0 \to \Delta_i$ via $\sigma_0 \mapsto d\alpha_i(\sigma_0)$. Since $q_j \co \g_j \to \g_j/p_j(\h_j)$ induces a bijection $\Delta_j \to q_j(\Delta_j)$ via $\sigma_j \mapsto q_j(\sigma_j)$, we have $\overline{d\alpha_i} \co \g_0/p_0(\h_0) \to \g_i/p_i(\h_i)$ induces a bijection $q_0(\Delta_0) \to q_i(\Delta_i)$. For $\rho_j \in \Delta_j^{(1)}$, we denote by $\lambda_j(\rho_j)$ the primitive generator of $\rho_j$. Since $d\alpha_i$ induces a one-to-one correspondence between primitive generators, we have 
		\begin{equation}\Label{eq:lambda}
			d\alpha_i (\lambda_0(\rho_0)) = \lambda_i(d\alpha_i(\rho_0))
		\end{equation} 
		for $\rho_0 \in \Delta_0^{(1)}$. Applying $q_i$ to the left hand side of \eqref{eq:lambda}, we have 
			$q_i(d\alpha_i(\lambda_0(\rho_0))) = \overline{d\alpha_i}(q_0(\lambda_0(\rho_0)))$.
		Thus we have 
		\begin{equation}\Label{eq:qlambda}
			\overline{d\alpha_i}(q_0(\lambda_0(\rho_0))) = q_i(\lambda_i(d\alpha_i(\rho_0))).
		\end{equation}
		Let $\widetilde{\lambda}_j \co q_j(\Delta_j)^{(1)} \to q_j(\ker \exp_{G_j})$ be the map given by $\widetilde{\lambda}_j(q_j(\rho_j)) = q_j(\lambda_j(\rho_j))$ for $\rho_j \in \Delta_j^{(1)}$. By \eqref{eq:qlambda}, we have 
		\begin{equation*}
			\overline{d\alpha_i} \circ \widetilde{\lambda}_0(q_0(\rho_0)) = \widetilde{\lambda}_i (q_i(d\alpha_i(\rho_0))) = \widetilde{\lambda}_i (\overline{d\alpha_i}(q_0(\rho_0))),  
		\end{equation*}
		that is, $\overline{d\alpha_i} \circ \widetilde{\lambda}_0 = \widetilde{\lambda}_i \circ \overline{d\alpha_i}$. Therefore $\widetilde{\F}_1(M_0,G_0,y_0)$ and $\widetilde{\F}_1(M_i,G_i,y_i)$ are isomorphic. 
		
		Conversely, suppose that $\widetilde{\F}_1(M_1,G_1,y_1)$ and $\widetilde{\F}_1(M_2,G_2,y_2)$ are isomorphic. Then there exists a linear isomorphism $\varphi \co \g_1/p_1(\h_1) \to \g_2/p_2(\h_2)$ satisfying the followings.
		\begin{enumerate}
			\item $\varphi(q_1(\ker \exp_{G_1})) = q_2(\ker \exp_{G_2})$. 
			\item $\varphi$ induces an isomorphism of fans $q_1(\Delta_1) \to q_2(\Delta_2)$. 
			\item $\widetilde{\lambda}_2 \circ \varphi|_{q_1(\Delta_1)^{(1)}} = \varphi \circ \widetilde{\lambda}_1$. 
		\end{enumerate}
		We construct $(\Delta_0, \h_0, G_0) \in \mathcal{C}_2$ and apply Theorem \ref{theo:principal} in order to show that $(M_1,G_1,y_1)$ and $(M_2,G_2,y_2)$ are equivalent. 
		Define 
		\begin{equation*}
			\Gamma_0 := \{ (\gamma_1,\gamma_2) \in \ker \exp_{G_1} \times \ker \exp_{G_2} \mid \varphi(q_1(\gamma_1)) = q_2(\gamma_2)\}
		\end{equation*}
		and denote by $\g_0$ the linear hull of $\Gamma_0$ in $\g_1 \times \g_2$. Let $\exp_{G_1 \times G_2} \co \g_1 \times \g_2 \to G_1 \times G_2$ be the exponential map. Then $G_0:= \exp_{G_1\times G_2} (\g_0)$ is a subtorus of $G_1 \times G_2$ and $\Gamma_0$ coincides with the kernel of the exponential map $\exp_{G_0} \co \g_0 \to G_0$. For each $\sigma_1 \in \Delta_1$, there uniquely exists $\sigma_2 \in \Delta_2$ such that $\varphi(q_1(\sigma_1)) = q_2(\sigma_2)$. We denote such $\sigma_2$ by $\Phi(\sigma_1)$ for $\sigma _1 \in \Delta_1$. 
		Then, for each $\rho_1 \in \Delta_1^{(1)}$, the element $\lambda_0(\rho_1) := (\lambda_1(\rho_1), \lambda_2(\Phi (\rho_1)))$ is a primitive element in $\Gamma_0$. Suppose that a cone $\sigma_1\in \Delta_1$ is the Minkowsky sum $\sigma_1 = \rho_{1,1} + \dots +\rho_{1,k}$ of $1$-cones $\rho_{1,1}, \dots, \rho_{1,k} \in \Delta_1^{(1)}$. Then we denote by $\Psi(\sigma_1)$ the cone in $\g_0$ spanned by $\lambda_0(\rho_{1,1}), \dots, \lambda_0(\rho_{1,k})$. Under this notation, we have a nonsingular fan $\Delta_0 := \{ \Psi({\sigma_1}) \subset \g_0 \mid \sigma_1 \in \Delta_1\}$. Let $\alpha_i \co G_0 \to G_i$ be the projection $G_1 \times G_2 \to G_i$ restricted to $G_0 \subset G_1 \times G_2$ for $i=1,2$. Then $\alpha_1 \co G_0 \to G_1$ and  $\alpha_2 \co G_0 \to G_2$ induce morphisms $d\alpha_1 \co \Delta_0 \to \Delta_1$ and $d\alpha_2 \co \Delta_0 \to \Delta_2$ of fan. In addition, these morphisms of fan induce  bijections between cones in fans. Moreover, $d\alpha_i \co \mathfrak{g}_0 \to \mathfrak{g}_i$ induces a one to one correspondence from the primitive generators of $1$-cones in $\Delta_0$ and the primitive generators of $1$-cones in $\Delta_i$. 
		
		Define
	\begin{equation*}
		\mathfrak{h}_0 := \{ (v_1,v_2) \in \mathfrak{h}_1 \times \mathfrak{h}_2 \mid (p_1(v_1),p_2(v_2)) \in \mathfrak{g}_0\}. 
	\end{equation*}
	Then $\mathfrak{h}_0$ is a $\C$-subspace of $\mathfrak{g}_0^\C \subset \g_1^\C \times \g_2^\C$. Moreover, the restriction $p_0|_{\mathfrak{h}_0}$ of the projection $p_0 \co \mathfrak{g}_0^\C \to \mathfrak{g}_0$ is injective because $p_1|_{\mathfrak{h}_1}$ and $p_2|_{\mathfrak{h}_2}$ both are injective. 
	Since $q_i$ and $d\alpha_i$ both are surjective, we have that $q_i\circ d\alpha_i \co \mathfrak{g}_0 \to \g_i/p_i(\h_i)$ is surjective for $i=1,2$. We claim that 
	\begin{equation}\Label{eq:2equalities}
		\ker q_1\circ d\alpha_1=  \ker q_2\circ d\alpha_2 = p_0(\mathfrak{h}_0).
	\end{equation} 
	Since $q_2\circ d\alpha_2 = \varphi \circ q_1 \circ d\alpha_1$ and $\varphi$ is an isomorphism, we have that the first equality of \eqref{eq:2equalities} holds. For the second equality of \eqref{eq:2equalities}, let $(\gamma_1,\gamma_2) \in \ker q_2\circ d\alpha_2$. Then we have $q_2(\gamma_2) = 0$ and $\gamma_2 \in p_2(\mathfrak{h}_2)$. Since $q_2\circ d\alpha_2 = \varphi \circ q_1 \circ d\alpha_1$, we have $f\circ q_1(\gamma_1) = 0$. Since $\varphi$ is an isomorphism, we have $q_1(\gamma_1)= 0$. Thus we have $\gamma_1 \in p_1(\h_1)$. Therefore $(\gamma_1, \gamma_2) \in p_0(\h_0)$. Therefore $\ker q_2 \circ d\alpha_2 \subset p_0(\h_0)$. Conversely, let $(\gamma_1,\gamma_2) \in p_0(\h_0)$. Then $q_2\circ d\alpha_2(\gamma_1,\gamma_2) = q_2(\gamma_2)$. Since $\gamma_2 \in p_2(\h_2)$, we have $q_2(\gamma_2)=0$. Therefore $(\gamma_1,\gamma_2) \in \ker q_2\circ d\alpha_2$, showing the opposite inclusion $\ker q_2 \circ d\alpha_2 \supset p_0(\h_0)$. 
	
	Let $q_0 \co \g_0 \to \g_0/p_0(\h_0)$ be the quotient map. Since $d\alpha_1 \co \g_0 \to \g_1$ satisfies that $\ker q_1\circ d\alpha_1 = p_0(\mathfrak{h}_0)$, $d\alpha_1$ induces an isomorphism $\overline{d\alpha_1} \co \g_0/p_0(\h_0) \to \g_1/p_1(\h_1)$ such that $\overline{d\alpha_1} \circ q_0 = q_1\circ d\alpha_1$. Since the maps $\Delta_0 \to \Delta_1$ given by $\sigma_0 \mapsto d\alpha_1(\sigma_0)$ and $\Delta_1 \to q_1(\Delta_1)$ given by $\sigma_1 \mapsto q_1(\sigma_1)$ are bijective, we have that the composition $\Delta_0 \to q_1(\Delta_1)$ given by $\sigma_0 \mapsto q_1 \circ d\alpha_1(\sigma_0)$ is bijective. Since $q_0 = \overline{d\alpha_1}^{-1}\circ q_1 \circ d\alpha_1$ and $\overline{d\alpha_1}^{-1}$ is an isomorphism, we have that 
	\begin{equation*}
		q_0(\Delta_0) = \{q_0(\sigma_0) \mid \sigma_0 \in \Delta_0\}
	\end{equation*}
	is a complete fan in $\g_0/p_0(\h_0)$ and the map $\Delta_0 \to q_0(\Delta_0)$ given by $\sigma _0 \mapsto q_0(\sigma_0)$ is bijective. Therefore we have $(\Delta_0,\h_0, G_0) \in \mathcal{C}_2$. 

	Since $\h_0 \subset \h_1 \times \h_2 \subset \g_1^\C \times \g_2^\C$ and $d\alpha_i^\C$ is nothing but the restriction of the projection $\g_1 ^\C \times \g_2^\C \to \g_i^\C$ to $\g_0^\C$, we have  $d\alpha_i^\C(\h_0) \subset \h_i$. Thus we have $\alpha_i \in \Hom _{\mathcal{C}_2}((\Delta_0,\h_0,G_0), (\Delta_i,\h_i,G_i))$. Since $\alpha_i$ is surjective and $d\alpha_i \co \g_0 \to \g_i$ induces a one-to-one correspondence from the primitive generator of $1$-cones in $\Delta_0$ to the primitive generators of $1$-cones in $\Delta_i$, there exist $(M_0,G_0,y_0) \in \mathcal{C}_1$ and $\alpha_i$-equivariant holomorphic map $f_i \co M_0 \to M_1$ such that $f_i(y_0) = y_i$ and $f_i$ is a $\ker \alpha_i$-principal bundle by Theorem \ref{theo:principal} and by the equivalence of categories $\mathcal{C}_1$ and $\mathcal{C}_2$. 
	
	It remains to show that $\ker \alpha_i$ is connected for $i=1,2$. Since $\alpha_i \co G_0 \to G_i$ is surjective, we have that $\ker \alpha_i$ is connected if and only if $d\alpha_i(\ker \exp_{G_0}) = \ker \exp_{G_i}$. Remark that $\ker \exp_{G_0} = \Gamma_0$. Let $\gamma_1 \in \ker \exp_{G_1}$. Since $\varphi(q_1(\ker \exp_{G_1})) = q_2(\ker \exp_{G_2})$, there exists $\gamma_2 \in \ker \exp_{G_2}$ such that $q_2(\gamma_2) = \varphi \circ q_1(\gamma_1)$. Then we have $(\gamma_1,\gamma_2) \in \Gamma_0$, 
	showing that $d\alpha_i (\ker \exp_{G_0}) = \ker \exp_{G_1}$. The same argument works for that $d\alpha_i (\ker \exp_{G_0}) = \ker \exp_{G_2}$, and we have that $\ker \alpha_i$ is connected. The theorem is proved.  
	\end{proof}
	\begin{proof}[Proof of Theorem \ref{theo:ess.surjective}]
		Let $(\widetilde{V}, \widetilde{\Gamma}, \widetilde{\Delta}, \widetilde{\lambda}) \in \widetilde{\mathcal{C}}_2$. Let $\widetilde{\rho}_1,\dots, \widetilde{\rho}_m$ be all $1$-dimensional cones in $\widetilde{\Delta}$. 
		Let $K$ be the abstract simplicial complex given by
		\begin{equation*}
			K := \{ \{i_1,\dots, i_k\} \subset \{1,\dots, m\} \mid \rho_{i_1} + \dots +\rho_{i_k} \in \widetilde{\Delta}\}.
		\end{equation*} 
		For $j=1,\dots, m$, put $\widetilde{\gamma}_j := \widetilde{\lambda}(\rho_j)$. Suppose that $\dim \widetilde{V} = N_1$. 
		Since $\widetilde{\Gamma}$ is finitely generated, there exist a positive integer $N_2$ and $\widetilde{\gamma}_{m+1}, \dots, \widetilde{\gamma}_{N_2} \in \widetilde{\Gamma}$ such that $\widetilde{\gamma}_1,\dots, \widetilde{\gamma}_m, \widetilde{\gamma}_{m+1},\dots, \widetilde{\gamma}_{N_2}$ generates $\widetilde{\Gamma}$, $m \leq N_2$ and $N_2 - N_1$ is nonnegative  even integer. 
		For $j=1,\dots, N_2$, let $e_j$ denote the $j$-th standard basis vector of $\R^{N_2}$. 
		The collection of cones 
		\begin{equation*}	
			\Delta := \{ \R_{\geq 0}e_{i_1} + \dots +\R_{\geq 0}e_{i_k} \mid \{i_1,\dots, i_k\} \in K\}
		\end{equation*}
		is a nonsingular fan with respect to the lattice $\Z^{N_2}$. 
		
		Let $\psi \co \R^{N_2} \to \widetilde{V}$ be the linear map given by $\psi (e_j) = \widetilde{\gamma}_j$ for $j=1,\dots, N_2$. Then the map $\Delta \to \widetilde{\Delta}$ given by $\sigma \mapsto \psi(\sigma)$ is bijective. 
		Since $\widetilde{\gamma}_1,\dots, \widetilde{\gamma}_m$ generates $\widetilde{\Gamma}$, we have that $\psi$ is surjective.  
		Put $N_3 := (N_2-N_1)/2 =  \dim \ker \psi / 2$. Let $b_1,\dots, b_{N_3}, b'_1,\dots, b'_{N_3}$ be basis vectors of $\ker \psi$. Let $\h$ be the $\C$-subspace of $(\R^{N_2})^\C =  \C^{N_2}$ spanned by $b_1+\sqrt{-1}b_1',\dots, b_{N_3}+\sqrt{-1}b_{N_3}'$. Let $p \co (\R^{N_2})^\C = \C^{N_2} \to \R^{N_2}$ be the projection. We claim that the restriction $p|_{\h}$ is injective. Suppose that 
		\begin{equation}\Label{eq:b}
			p\left(\sum_{j=1}^{N_3}\alpha_j(b_j+\sqrt{-1}b_j')\right) =0, \quad \alpha_j \in \C.
		\end{equation}
		Let $\alpha_j = a_j + \sqrt{-1}a_j'$, $a_j,a_j' \in \R$. Then \eqref{eq:b} yields that $\sum_{j=1}^{N_3} a_jb_j-a_j'b_j' = 0$. Since $b_1,\dots, b_{N_3}, b_1',\dots, b_{N_3}$ are basis vectors of $\ker \psi$, we have $\alpha_j = 0$ for all $j$. Therefore $p|_{\h}$ is injective. 
		
		Since $p(b_j + \sqrt{-1}b_j') = b_j$ and $p(-\sqrt{-1}(b_j + \sqrt{-1}b_j')) = b_j'$, we have $p(\h) \subset \ker \psi$. On the other hand, $p(v + \sqrt{-1}v')$ is a linear combination of $b_1,\dots, b_{N_3}, b_1',\dots, b_{N_3}'$ for $v+\sqrt{-1}v' \in \h$. Therefore we have $p(\h) = \ker \psi$. Let $q \co \R^{N_2} \to \R^{N_2}/p(\h)$. Then $\psi$ induces the linear isomorphism $\overline{\psi} \co \R^{N_2}/p(\h) \to \widetilde{V}$ such that $\overline{\psi} \circ q= \psi$. Therefore $q (\Delta) = \overline{\psi}^{-1}(\widetilde{\Delta})$ is a complete fan and the map $\Delta \to q(\Delta)$ given by $\sigma \mapsto q(\sigma)$ is bijective. 
		Thus we have $(\Delta, \h, \R^{N_2}/\Z^{N_2}) \in \mathcal{C}_2$. $\F_2(\Delta, \h, \R^{N_2}/\Z^{N_2})\in \mathcal{C}_1$ is what we wanted, proving the theorem. 
	\end{proof}
\section{Preliminaries 2}\Label{sec:preliminaries2}
\subsection{Structures around minimal orbits}\Label{subsec:minimal}
	Let $M$ be a connected manifold equipped with an effective action of $G$. 
	Since the action of $G$ on $M$ is effective, we have $\dim G_x + \dim G \leq \dim M$
	for any $x$. This inequality is equivalent to the inequality
	\begin{equation}\Label{eq:minimal}
		\dim G\cdot x \geq 2\dim G - \dim M. 
	\end{equation}
	We say that a $G$-orbit $G \cdot x$ is \emph{minimal} if the equality of \eqref{eq:minimal} holds. Then, the effective action of $G$ is maximal if and only if there exists a minimal orbit $G\cdot x$. 
	\begin{lemm}[{\cite[Lemma 2.3]{Ishida}}]\Label{lemm:minimal}
		Let $M$ be a connected manifold equipped with a maximal action of a compact torus $G$. Let $G\cdot x$ be a minimal orbit. Then the followings hold.
		\begin{enumerate}
			\item The isotropy subgroup $G_x$ of $G$ at $x$ is connected. 
			\item $G\cdot x$ is a connected component of the fixed point set of the action of $G$ restricted to $G_x$ on $M$. 
			\item Each minimal orbit is isolated. In particular, there are finitely many minimal orbits if $M$ is compact. 
		\end{enumerate}
	\end{lemm}
	
	Until the end of this subsection, let $(M,G,y) \in \mathcal{C}_1$ and let $G\cdot x$ be a minimal orbit. By Lemma \ref{lemm:minimal} (2), we have that $G\cdot x$ is a connected component of the fixed point set of the action of $G$ restricted to $G_x$ on $M$. This together with the assumption that the $G$-action on $M$ preserves the complex structure implies that $G\cdot x$ is a holomorphic submanifold of $M$. 
	Let $G^\C$ be the complexification of $G$. Since $M$ is compact and the action of $G$ preserves the complex structure $J$ on $M$, we have the holomorphic $G^\C$-action on $M$. The global stabilizers form a holomorphic closed subgroup 
	\begin{equation*}
		H := \{ g \in G^\C \mid g\cdot x' = x' \text{ for all } x' \in M\}
	\end{equation*} 
	of $G^\C$. The Lie algebra of $H$ is nothing but $\h$ given by \eqref{eq:h}. By definition, we have an effective holomorphic action of the quotient group $G^\C/H$ on $M$. 
	We denote by $G^M$ the quotient group $G^\C/H$. 	
	\begin{lemm}[{\cite[Lemma 4.6 and Equation (4.4)]{Ishida}}]\Label{lemm:dimension}
		The followings hold.
		\begin{enumerate}
			\item There exists an open dense $G^\C$-orbit in $M$. In particular, we have $\dim H =2\dim G-dim M = \dim G \cdot x$. 
			\item We have the decomposition $\g^\C = \g \oplus \sqrt{-1}\g_x \oplus \h$. 
		\end{enumerate}
	\end{lemm}
	Since $G\cdot x$ is a holomorphic submanifold of $M$, we have that the vector field $JX_v$ tangents to $G\cdot x$ for all $v\in \g$, where $X_v$ denotes the fundamental vector field generated by $v \in \g$. Thus $G\cdot x$ is invariant under the action of $G^M$. Since $G$ acts on $G\cdot x$ transitively, $G^M$ also acts on $G\cdot x$ transitively. Therefore we have that $G\cdot x$ is biholomorphic to the homogeneous space $G^M/(G^M)_x$. 
	\begin{lemm}[{\cite[Lemma 4.8]{Ishida}}]\Label{lemm:gx}
		Let $\g_x$ be the Lie algebra of $G_x$ and $(\g^M)_x$ the Lie algebra of the isotropy subgroup $(G^M)_x$ of $G^M$. The natural map $(\g_x)^\C \hookrightarrow \g^\C/\h$ induces an isomorphism $(G_x)^\C \to (G^M)_x$.
	\end{lemm}
	The normal space $T_xM/T_x(G\cdot x)$ is a faithful representation space of $(G_x)^\C$. Thus there exist $\alpha_1,\dots, \alpha_k \in \Hom ((G_x)^\C,\C^*)$ such that 
	\begin{equation*}
		T_xM/T_x(G\cdot x) \cong \C_{\alpha_1} \oplus \dots \oplus \C_{\alpha_k}
	\end{equation*}
	as $(G_x)^\C$-representations, where $\C_\alpha$ denotes the complex $1$-dimensional representation of the weight $\alpha \in \Hom ((G_x)^\C,\C^*)$. Since $\dim T_xM/T_x(G\cdot x) = \dim (G_x)^\C$ and $T_xM/T_x(G\cdot x)$ is faithful, we have that $\alpha_1,\dots, \alpha_k$ form a basis of $\Hom((G_x)^\C,\C^*)$. 
	\begin{prop}[{\cite[Proposition 4.9]{Ishida}}]\Label{prop:slice}
		There uniquely exists a minimal $G^M$-invariant open neighborhood $N(G\cdot x) \subset M$ of $G\cdot x$ that is $G^M$-equivariantly biholomorphic to 
		\begin{equation*}
			G^M \times _{(G_x)^\C} \bigoplus_{i=1}^k \C_{\alpha_i}. 
		\end{equation*}
	\end{prop}
	
\subsection{Transverse K\"ahler structures}\Label{subsec:transverse}
	Let $M$ be a smooth manifold. A \emph{presymplectic form} is a closed  $2$-form on $M$. Let $F$ be a smooth foliation on $M$. A presymplectic form $\omega$ is said to be \emph{transverse symplectic with respect to $F$} if the kernel of $\omega$ coincides with the subbundle $TF$ of $TM$ that consists of vectors tangent to leaves of $F$. Suppose that a torus $G$ acts on $M$ smoothly and a transverse symplectic form $\omega$ with respect to $F$ is $G$-invariant. Then, the Cartan formula tells us that the $1$-form $i_{X_v}\omega$ is a closed $1$-form, where $X_v$ denotes the fundamental vector field generated by an element $v$ of the Lie algebra $\mathfrak{g}$ of $G$ and $i_{X_v}\omega$ denotes the interior product. A \emph{moment map} $\Phi \co M \to \mathfrak{g}^*$ is a map satisfying that $dh_v = -i_{X_v}\omega$, where $h_v \co M \to \R$ is a function given by $h_v(x) = \langle \Phi (x), v \rangle$ for all $x \in M$ and $v \in \g$. 
	
	Suppose that $M$ is a complex manifold and $F$ is a holomorphic foliation on $M$. Let $J$ be the complex structure on $M$. A \emph{transverse K\"ahler form} $\omega$ on $M$ with respect to $F$ is a closed $2$ -form that satisfies the followings.
	\begin{enumerate}
		\item $\omega(JX,JY) = \omega (X, Y)$ for all $X, Y \in T_xM$. 
		\item $\omega(X, JX) \geq 0$ for all $X \in T_xM$. The equality holds if and only if $X \in T_xF$. 
	\end{enumerate}
	By definition, a transverse K\"ahler form $\omega$ with respect to $F$ is a transverse symplectic form with respect to $F$. The canonical foliation relates the existence of a transverse K\"ahler form that has a moment map, see {\cite[Proposition 4.2]{Ishida2}}. 
	
	Until the end of this subsection, let $(M,G,y) \in \mathcal{C}_1$ and let $F$ be the canonical foliation on $M$. Let $(\Delta, \h, G) := \mathcal{F}_1(M,G,y) \in \mathcal{C}_2$. As before, let $p \co \g^\C \to \g$ be the projection and $q \co \g \to \g/p(\h)$ the quotient map. 
	\begin{theo}[{\cite[Theorems 2.7, 4.3 and 5.8 and Proposition 4.1]{Ishida2}}]\Label{theo:transverse} 
		The followings hold.
		\begin{enumerate}
			\item Let $\omega$ be a $G$-invariant transverse K\"ahler form with respect to $F$. Then there exists a moment map $\Phi \co M \to \g^*$. Moreover, there exist $\alpha \in \g^*$ and a smooth map $\widetilde{\Phi} \co M \to (\g/p(\h))^*$ such that $\Phi = q^* \circ \widetilde{\Phi} +\alpha$.  
			\item Let $\widetilde{\Phi}$ be as \rm{(1)}. Let $Z_1,\dots, Z_l$ be the connected components of the set of common critical points of $h_v$ for all $v \in \g$. Then $\widetilde{\Phi}(Z_j)$ is a point $c_j$ for $j =1,\dots, l$ and $\widetilde{\Phi}(M)$ is a convex hull of $c_1,\dots, c_N$. The fan $q(\Delta)$ is an inner normal fan of the polytope $\widetilde{\Phi}(M)$. 
			\item If $q(\Delta)$ is polytopal, then there exists a $G$-invariant transverse K\"ahler form with respect to $F$. 
			\item For a transverse K\"ahler form $\omega$ with respect to $F$, the average $\int_{g \in G}g^*\omega dg$ is a $G$-invariant transverse K\"ahler form with respect to $F$. 
		\end{enumerate}
		In particular, there exists a transverse K\"ahler form with respect to $F$ if and only if $q(\Delta)$ is polytopal. 
	\end{theo}
	In Section \ref{sec:basicDolbeault}, we will see how we construct a transverse K\"ahler form with respect to $F$ from a polytopal fan $q(\Delta)$. 
	
	Let $\omega$ be a $G$-invariant transverse K\"ahler form with respect to $F$ and $\widetilde{\Phi}$ as above. We shall describe the set $Z := Z_1\sqcup \dots \sqcup Z_N$ of common critical points of $h_v$ for all $v \in \g$. Since $dh_v = -i_{X_v}\omega$ and the kernel of $\omega$ coincides with $TF$, we have that $x \in Z$ if and only if $(X_v)_x \in T_xF$ for all $v\in \g$. Namely, we have  
	\begin{equation*}
		Z = \{ x \in M \mid (X_v)_x \in T_xF \text{ for all } v \in \g\}.
	\end{equation*}
	Therefore $Z$ does not depend on $\omega$ and $\widetilde{\Phi}$. 
	\begin{prop}\Label{prop:crit}
		$Z$ is the union of all minimal orbits. 
		Each connected component of $Z$ is a minimal orbit. Conversely, each minimal orbit is a connected component of $Z$. 
	\end{prop}
	\begin{proof}
		Let $x \in Z$ and $v \in \g$. Since $(X_v)_x \in T_xF$, there uniquely exists $v' \in p(\h)$ such that $(X_v)_x = (X_{v'})_x$. Then we have $(X_{v-v'})_x = (X_v)_x-(X_{v'})_x = 0$. Therefore $v-v' \in \g_x$. Thus we have the decomposition $\g = \g_x \oplus p(\h)$. Therefore we have $\dim G\cdot x = \dim \g - \dim \g_x = \dim p(\h)$. Since $p|_{\h}$ is injective, we have $\dim G\cdot x = \dim H = 2\dim G - \dim M$  by Lemma \ref{lemm:dimension} (1). Namely, $G\cdot x$ is a minimal orbit. Therefore $Z$ is contained in the union of minimal orbits. Conversely, if $G \cdot x$ is a minimal orbit, then $ \dim G - \dim G_x = \dim G\cdot x = 2\dim G-\dim M$. By Lemma \ref{lemm:dimension} (1), we have $2\dim G- \dim M = \dim H$. Thus we have $\dim G - \dim G_x = \dim H$. Since $p|_\h$ is injective and $\dim G - \dim G_x = \dim H$, we have $\dim p(\h) +\dim \g_x = \dim \g$. This together with $\g_x \cap p(\h)= \{0\}$ yields that $\g = \g_x \oplus p(\h)$. Therefore we have $(X_v)_x \in T_xF$ for any $v$. This shows that $x \in Z$. Therefore $Z$ is the union of all minimal orbits. 
		
	By Lemma \ref{lemm:minimal} (3), each minimal orbit is isolated. Therefore each connected component of $Z$ is a minimal orbit and vice versa. The proposition is proved. 
	\end{proof}
	\begin{lemm}[{\cite[Lemma 2.2]{Ishida2}}]\Label{lemm:Bott-Morse}
		Let $v \in \g$. Let $x \in M$ be a critical point of $h_v$. Let $\alpha_1,\dots, \alpha_k \in \Hom (G_x,S^1)$ be the nonzero weights of the $G_x$-representation $T_xM$.
		The followings hold.
		\begin{enumerate}
			\item $h_v$ is nondegenerate. Namely, each component of the critical set is a submanifold and the second derivative is a nondegenerate quadric form in the transverse direction.
			\item There exist $v_x \in \g_x$ and $v' \in p(\h)$ such that $v = v_x + v'$. The index of $h_v$ at $x$ is twice as many as the number of $\alpha_i$ such that $\langle d\alpha_i,v_x \rangle < 0$. 
		\end{enumerate}
	\end{lemm}
	\begin{lemm}\Label{lemm:minimum}
		Let $G\cdot x$ be a minimal orbit. Then there exists $v \in \g$ such that $G\cdot x$ is a critical submanifold and $h_v$ attains the minimum on $G\cdot x$. 
	\end{lemm}
	\begin{proof}
		Let $G\cdot x$ be a minimal orbit.
		Let $\alpha_1,\dots, \alpha_k \in \Hom(G_x,S^1)$ be the nonzero weights of $T_xM$. It follows from Lemma \ref{lemm:minimal} (2) that the codimension of $G\cdot x$ is equal to $2k$. Thus we have that $\alpha_1,\dots, \alpha_k$ form a $\Z$-basis of $\Hom(G_x,S^1)$; otherwise, the action of $G$ is not effective. We identify $\Hom (S^1,G_x)$ with $\ker \exp_{G_x} \subset \g_x$ as before. Let $v_1,\dots, v_k \in \g_x$ be the dual basis of $\alpha_1,\dots, \alpha_k$. Put $v:= v_1+\dots + v_k$. Since $\langle d\alpha_i,v\rangle =1 >0$ for all $i$, we have that $G\cdot x$ is a critical submanifold and the index of $h_v$ on $G\cdot x$ is $0$. Therefore $h_v$ attains a local minimum on $G\cdot x$. Put $c := h_v (x)$. By Lemma \ref{lemm:Bott-Morse}, all indices of $h_v$ are even. Thus we have that $h_v^{-1}(c)$ is connected. Therefore $h_v$ attains the global minimum on $G\cdot x$, as required. 
	\end{proof}
	\begin{prop}\Label{prop:vertex}
		Let $Z_1,\dots, Z_N$ be minimal orbits. Each image $\widetilde{\Phi}(Z_j)$ of $Z_j$ by $\widetilde{\Phi} \co M \to (\g/p(\h))^*$ is a vertex of $\widetilde{\Phi}(M)$ and $\widetilde{\Phi}(Z_i) \neq \widetilde{\Phi}(Z_j)$ if $i \neq j$. 
	\end{prop}
	\begin{proof}
		Let $Z_j$ be a minimal orbit. 
		By Proposition \ref{prop:crit}, we have that $Z_j$ is a connected component of the set of common critical points of $h_v$. By Theorem \ref{theo:transverse} (2), we have that $\widetilde{\Phi}(Z_j)$ is a point $c_j$. By Lemma \ref{lemm:minimum}, there exists $v \in \g$ such that $Z_j$ is a critical submanifold and $h_v$ attains the minimum on $Z_j$. Let $c \in \R$ be the minimum value of $h_v$. Since 
		\begin{equation}\Label{eq:hv}
			h_v(x) = \langle \Phi(x), v\rangle = \langle q^*\circ \widetilde{\Phi}(x)+\alpha, v\rangle =\langle \widetilde{\Phi}(x), q(v)\rangle +\langle \alpha, v\rangle
		\end{equation}
		and $h_v$ attains the minimum $c$ on $Z_j$, we have that the half space 
		\begin{equation*}
			\{ y \in (\g/p(\h))^* \mid \langle y, q(v)\rangle +\langle \alpha, v\rangle \geq c\} \supset \widetilde{\Phi}(M)
		\end{equation*} 
		and 
		\begin{equation*}
			\{ y \in (\g/p(\h))^* \mid \langle y, q(v)\rangle +\langle \alpha, v\rangle = c\} \cap \widetilde{\Phi}(M) = \widetilde{\Phi}(Z_j).
		\end{equation*}
		Therefore $\widetilde{\Phi}(Z_j)$ is a vertex of $\widetilde{\Phi}(M)$.
		
		Let $Z_i$ be another minimal orbit. Suppose that $\widetilde{\Phi}(Z_i) = \widetilde{\Phi}(Z_j)$. Then, by \eqref{eq:hv}, we have $h_v(Z_i) = h_v(Z_j)$. Since $h_v^{-1}(c) = Z_j$, we have $Z_i = Z_j$. It turns out that $\widetilde{\Phi}(Z_i) \neq \widetilde{\Phi}(Z_j)$ if $i \neq j$. The proposition is proved. 
	\end{proof}
	\begin{lemm}\Label{lemm:generic}
		There exists $v \in \g$ that satisfies the followings.
		\begin{enumerate}
			\item The set of critical points of $h_v$ coincides with $Z  = Z_1 \sqcup \dots \sqcup Z_l$. 
			\item $h_v(Z_i) \neq h_v(Z_j)$ if $i \neq j$. 
		\end{enumerate}
	\end{lemm}
	\begin{proof}
		By Proposition \ref{prop:vertex}, the subset  
		\begin{equation*}
			H_{i,j} := \{ v \in \g \mid \langle \widetilde{\Phi}(Z_i)-\widetilde{\Phi}(Z_j),v\rangle = 0\}
		\end{equation*}
		is a hyperplane for $1\leq i<j\leq l$. By Proposition \ref{prop:crit}, we have that $\g_x \oplus p(\h) = \g$ if and only if $x \in Z$. Since $M$ is compact, the set $\{\g_x \mid x \in M\}$ is finite. Thus the set
		\begin{equation*}
			A:= \g \setminus \left( \bigcup_{1\leq i<j\leq N}H_{i,j} \cup \bigcup_{x \notin Z} (\g_x \oplus p(\h))\right)
		\end{equation*}
		is not empty. Any element $v \in A$ satisfies the conditions (1) and (2), proving the lemma. 
	\end{proof}

\section{Basic cohomology and equivariant cohomology}\Label{sec:cohomology}
	Let $M$ be a smooth manifold equipped with an effective action of a compact torus $G$. Let a subspace $\h' \subset \g$ generate a smooth foliation $F$ on $M$. Namely, $H' := \exp_G (\h')$ acts on $M$ local freely and each leaf of $F$ is an $H'$-orbit. Since $G$ is commutative, $TF$ is a $G$-invariant subbundle of $TM$. 
	
	We denote by $\Omega^*(M)$ the DGA (differential graded algebra) that consists of all differential forms on $M$. Since $G$ acts on $M$ smoothly, we have that $\Omega^*(M)$ is a $G$-representation space. We denote by $\Omega^*(M)^G$ the sub-DGA of all $G$-invariant differential forms. We denote by $H^*(M)$ the cohomology $H^*(\Omega^*(M))$ of $\Omega^*(M)$. 
	
	Let $\alpha \in \Omega^*(M)$. We say that $\alpha$ is \emph{basic} if $\alpha$ satisfies that $i_{X_v}\alpha = 0$ and $L_{X_v}\alpha =0$ for all $v \in \h'$, where $L_{X_v}\alpha$ denotes the Lie derivative of $\alpha$ by $X_v$. We denote by $\Omega_B^*(M)$ the set of all basic forms. By definition, $\Omega_B^*(M)$ becomes a sub-DGA of $\Omega^*(M)$. Namely, the differential of a basic form is basic. Since the $G$-action commutes with the $H'$-action, $\Omega_B^*(M)$ is a sub-representation of $\Omega^*(M)$. We denote by $\Omega^*_B(M)^G$ the sub-DGA of all $G$-invariant basic forms on $M$. We denote by $H_B^*(M)$ the cohomology $H^*(\Omega_B^*(M))$ of $\Omega_B^*(M)$ and call it \emph{the basic cohomology of $M$ for $F$}. 
	\begin{lemm}\Label{lemm:cohomology}
		Let $M, G, \h', F$ be as above. Then the followings hold.
		\begin{enumerate}
			\item $H^*(M)$ is isomorphic to $H^*(\Omega^*(M)^G)$. 
			\item $H^*_B(M)$ is isomorphic to $H^*(\Omega^*_B(M)^G)$. 
		\end{enumerate}
	\end{lemm}
	\begin{proof}
		In case when $\h' = \{0\}$, any differential form is basic. Therefore it suffices to show Part (2). 
		
		First, we show that the inclusion $\Omega^*_B(M)^G \hookrightarrow \Omega^*_B(M)$ induces an injective homomorphism $H^*(\Omega^*_B(M)^G) \to H^*_B(M)^G$. Let $I \co \Omega^*_B(M) \to \Omega^*_B(M)^G$ be the linear map given by 
		\begin{equation*}
			I(\alpha) := \int_{g \in G} g^*\alpha dg,\quad \alpha \in \Omega^*_B(M),
		\end{equation*}
		where $dg$ denotes the normalized Haar measure on $G$. Then the composition of the inclusion $\Omega^*_B(M)^G \hookrightarrow \Omega^*_B(M)$ with $I$ is the identity. Thus, the composition induces an cohomology isomorphism $H^*(\Omega^*_B(M)^G) \to H^*(\Omega^*_B(M)^G)$. Therefore the inclusion $\Omega^*_B(M)^G \hookrightarrow \Omega^*_B(M)$ induces an injective homomorphism $H^*(\Omega^*_B(M)^G) \to H^*_B(M)$. 
		
		We show that the induced homomorphism $H^*(\Omega^*_B(M)^G) \to H^*_B(M)$ is surjective. Let $[\alpha] \in H^*_B(M)$ be a cohomology class represented by a closed basic form $\alpha \in \Omega^*_B(M)$. Let $\gamma_1,\dots, \gamma_n$ be basis vectors of $\ker \exp_G$. Define 
		\begin{equation*}
			D:= \left\{ v = \sum_{i=1}^n a_i \gamma_i\mid 0\leq a_i <1\right\}.
		\end{equation*} 
		Then the exponential map restricted to $D$ gives a bijection $\exp_G|_D \co D \to G$. For $v \in D$ and $t \in \R$, we define $g_t := \exp_G(tv) \in G$ and 
		\begin{equation*}
			\theta_{g_1} := \int_0^1 g_t^*(i_{X_v}\alpha) dt \in \Omega^*_B(M).
		\end{equation*}
		We claim that $d\theta_{g_1} = g_1^*\alpha - \alpha$. Since
		\begin{equation*}
			\begin{split}
				\lim_{h \to 0} \frac{g^*_{t+h}\alpha - g^*_t\alpha}{h} &= L_{X_v}g^*_t\alpha\\
				&= i_{X_v}dg^*_t\alpha + di_{X_v}g^*_t\alpha \\ 
				&= di_{X_v}g^*_t\alpha \\
				&= dg^*_ti_{X_v}\alpha, 
			\end{split}
		\end{equation*}
		we have 
		\begin{equation*}
			\begin{split}
				d\theta_{g_1} &= d\int_0^1 g_t^*i_{X_v}\alpha dt\\
				&=\int_0^1 dg_t^*i_{X_v}\alpha dt\\
				&= g^*_1\alpha - g^*_0 \alpha \\
				&= g^*_1 \alpha - \alpha.
			\end{split}
		\end{equation*}
		Therefore 
		\begin{equation*}
			\begin{split}
				\int_{g_1 \in G}g^*_1\alpha-\alpha dg &= \int_{g_1 \in G}d \theta_{g_1}dg \\
				&= d\int_{g_1 \in G}\theta_{g_1} dg. 
			\end{split}
		\end{equation*}
		On the other hand, 
		\begin{equation*}
			\int_{g_1 \in G}g^*_1\alpha-\alpha dg = I(\alpha) - \alpha.
		\end{equation*}
		Therefore $I(\alpha)$ and $\alpha$ determine the same cohomology class in $H^*_B(M)$. This together with $I(\alpha) \in \Omega^*_B(M)^G$ yields that the induced homomorphism $H^*(\Omega_B^*(M)^G)\to H^*_B(M)$ is surjective, proving the lemma.
	\end{proof}
	The \emph{Cartan model} of the equivariant de Rham complex is given by
	\begin{equation*}
			\Omega_G^*(M) := (S^*(\g^*) \otimes \Omega^*(M))^G
	\end{equation*}
	where $S^*(\g^*)$ denotes the symmetric tensor algebra of $\g^*$. The degree of elements in $\g^*$ is $2$ and the degree of an element in $\Omega^*(M)$ is as usual. An element in $\Omega_G^*(M)$ can be think of a $G$-equivariant $\Omega^*(M)$-valued polynomial function on $\g$. Since $G$ is abelian, we have that $S^*(\g^*)$ is a trivial $G$-representation. Thus we have 
	\begin{equation*}
		\begin{split}
			\Omega_G^*(M) &:= (S^*(\g^*) \otimes \Omega^*(M))^G\\
			&=S^*(\g^*) \otimes \Omega^*(M)^G.
		\end{split}
	\end{equation*}
	Namely, an element in $\Omega_G^*(M)$ is a $\Omega^*(M)^G$-valued polynomial function on $\g$. The differential $d_G \co \Omega_G^*(M) \to \Omega_G^*(M)$ is given by 
	\begin{equation*}
		(d_G\alpha )(v) := d(\alpha(v)) - i_{X_v}\alpha(v)
	\end{equation*}
	for $\alpha \in \Omega_G^*(M)$ and $v \in \g$. The cohomology of the DGA $(\Omega_G^*(M), d_G)$ is called \emph{the equivariant cohomology of $M$} and denoted by $H_G^*(M)$. $H_G^*(M)$ is not only an $\R$-algebra but also an $S^*(\g^*)$-algebra. 
	
	\begin{lemm}[{\cite[Lemma 4.4 and Proposition 4.6]{Ishida-Kasuya}}]\Label{lemm:connection}
		Suppose that $M$ is paracompact. The followings hold.
		\begin{enumerate}
			\item There exists a $\h'$-valued $G$-invariant $1$-form $\theta$ on $M$ such that $i_{X_v}\theta = v$ for all $v \in \h'$.
			\item Let $v_1,\dots, v_k$ be basis vectors of $\h'$. Write 
			\begin{equation*}
				\theta = \sum_{i=1}^k \theta_i \otimes v_i, \quad \theta_i \in \Omega^1(M)^G.
			\end{equation*}
			Let $W$ be the subspace of $\Omega^1(M)^G$ spanned by $\theta_1,\dots, \theta_k$. Then we have the decomposition 
			\begin{equation*}
				\Omega^*(M)^G = \Omega_B^*(M)^G \otimes \bigwedge W. 
			\end{equation*}
		\end{enumerate}
	\end{lemm}
	Suppose that $M$ is paracompact and let $W$ be as in Lemma \ref{lemm:connection}. Then we have the decomposition 
	\begin{equation}\Label{eq:Weil}
		\Omega^*_G(M)= S^*(\g^*) \otimes \Omega_B^*(M)^G \otimes \bigwedge W.   
	\end{equation}
	Since $\h' \subset \g'$, we have that the inclusion induces the surjective homomorphism $S^*(\g^*) \to S^*(\h'^*)$. 
	We think of elements in $S^*(\h'^*) \otimes \Omega_B^*(M)^G \otimes \bigwedge W$ as $\Omega_B^*(M)^G \otimes \bigwedge W$-valued polynomial functions on $\h'$. For short, we denote $S^*(\h'^*) \otimes \Omega_B^*(M)^G \otimes \bigwedge W$ by $\Omega_{\h'}^*(M)$. 
	Let 
	\begin{equation*}
		d_{\h'} \co \Omega_{\h'}^*(M) \to \Omega_{\h'}^*(M)
	\end{equation*}
	be the linear map given by 
	\begin{equation*}
		d_{\h'}\alpha (u') = d(\alpha(u')) - i_{X_{u'}}(\alpha(u')), \quad \alpha \in \Omega_{\h'}(M), u' \in \h'.
	\end{equation*}
	Then $d_\h'^2 = 0$. Let $x_1,\dots, x_k$ denote the dual basis vectors of the basis vectors $v_1,\dots, v_k$ of $\h'$. Then $d_{\h'}$ is represented as 
	\begin{equation*}
		d_{\h'}\beta = d\beta - \sum_{i=1}^k x_i \otimes i_{X_{v_i}}\beta
	\end{equation*}
	for $\beta \in \Omega^*_B(M)^G$ and $d_{\h'}$ is a homomorphism of $S^*(\h'^*)$-algebra. 
	We define an $\R$-algebra homomorphism 
	\begin{equation*}
		f \co \Omega_{\h'}^*(M) \to \Omega_B^*(M)^G
	\end{equation*}
	given by 
	\begin{equation*}
		f(x_i \otimes 1\otimes 1) = d\theta_i, \quad f(1\otimes \beta \otimes 1) = \beta, \quad f(1\otimes 1\otimes w_i) = 0
	\end{equation*}
	for $i=1,\dots, k$ and $\beta \in \Omega_B^*(M)^G$. 
	\begin{lemm}\Label{lemm:Cartan}
		The map $f \co \Omega_{\h'}^*(M) \to  \Omega_B^*(M)^G$ above is a DGA homomorphism.
	\end{lemm}
	\begin{proof}
		We need to show that $d \circ f = f \circ d_{\h'}$. Let $n_1,\dots n_k$ be nonnegative integers and  $\beta \in \Omega_B^*(M)^G$. Then 
		\begin{equation*}
			x_1^{n_1} \dots x_k^{n_k} \otimes \beta \otimes \theta_{j_1}\wedge \dots \wedge \theta_{j_m} \in S^*(\h'^*) \otimes \Omega_B^*(M)^G \otimes \bigwedge  W. 
		\end{equation*}
		By definition of $f$, we have 
		\begin{equation*}
			\begin{split}
			d\circ f(x_1^{n_1} \dots x_k^{n_k} \otimes \beta \otimes \theta_{j_1}\wedge \dots \wedge \theta_{j_m}) &= d(\beta \wedge d\theta_{1}^{n_1}\wedge \dots \wedge d\theta_{k}^{n_k})\\
			&= d\beta \wedge d\theta_{1}^{n_1}\wedge \dots \wedge d\theta_{k}^{n_k}.
			\end{split}
		\end{equation*}
		On the other hand, 
		\begin{equation}\Label{eq:dh'}
			\begin{split}
				& d_{\h'}(x_1^{n_1} \dots x_k^{n_k} \otimes \beta \otimes \theta_{j_1}\wedge \dots \wedge \theta_{j_m})\\ = & x_1^{n_1} \dots x_k^{n_k}\otimes d\beta \otimes \theta_{j_1}\wedge \dots \wedge \theta_{j_m}\\
				&+ \sum_{i=1}^m x_1^{n_1} \dots x_k^{n_k}\otimes (-1)^{\deg \beta + i-1}\beta \wedge d\theta_{j_i} \otimes \theta_{j_1}\wedge \dots \wedge \widehat\theta_{j_i} \wedge \dots \wedge \theta_{j_m}\\
				&- \sum_{i=1}^m x_{j_i}\cdot  x_1^{n_1} \dots x_k^{n_k}\otimes (-1)^{\deg \beta + i-1}\beta \otimes \theta_{j_1}\wedge \dots \wedge \widehat\theta_{j_i} \wedge \dots \wedge \theta_{j_m}.
			\end{split}
		\end{equation}
		Therefore we have 
		\begin{equation*}
			f\circ d_{\h'} (x_1^{n_1} \dots x_k^{n_k} \otimes \beta \otimes \theta_{j_1}\wedge \dots \wedge \theta_{j_m}) = d\beta \wedge d\theta_{1}^{n_1}\wedge \dots \wedge d\theta_{k}^{n_k}.
		\end{equation*}
		These computations show $d \circ f = f \circ d_{\h'}$, as required. 
	\end{proof}
	$f$ is the \emph{Cartan operator}. Namely, the following holds: 
	\begin{theo}[{\cite[Chapter 5]{Guillemin-Sternberg}}]\Label{theo:Cartan}
		$f$ induces an isomorphism 
		\begin{equation*}H ^* (\Omega^*_{\h'}(M),d_{\h'}) \to H^*(\Omega_B^*(M)^G, d_B) \cong H^*_B(M).
		\end{equation*} 
	\end{theo}

	As a conclusion, we have a DGA homomorphism $\Omega_G^*(M) \to \Omega_B^*(M)^G$ given by $\alpha \mapsto f(\alpha|_{\h'})$. We denote by $\forB \co H^*_G(M) \to H^*_B(M)$ the induced homomorphism and call it \emph{the basic forgetful map}. 
	
	\begin{rema}\Label{rema:eqbasiccoho}
		One can construct the basic forgetful map $\forB \co H^*_G(M) \to H^*_B(M)$ via the \emph{equivariant basic cohomology} introduced in \cite{Goertsches-Toeben}. To prevent confusing, by $\Omega^*(M,F)$ we mean the basic complex $\Omega^*_B(M)$ and by $H^*(M,F)$ we mean the basic cohomology $H^*_B(M)$ for a moment. Let $\g'$ be a complement of $\h'$ in $\g$. Then we have a transverse action of $\g'$ on the foliated manifold $(M,F)$. The cohomology $H^*_{\g'}(M,F)$ of the DGA $\Omega^*_{\g'}(M,F) = (S^*(\g'^*) \otimes \Omega^*(M,F))^{\g'}$ is the $\g'$-equivariant $F$-basic cohomology. One can see that $H^*_G(M)$ is isomorphic to $H^*_{\g'}(M,F)$ (see \cite[Example 4.3]{Goertsches-Toeben}). The basic forgetful map is the composition of the isomorphism $H^*_G(M) \cong H^*_{\g'}(M,F)$ with the natural map $H^*_{\g'}(M,F) \to H^*(M,F)$. We note that equivariant basic cohomology has been constructed for not only abelian but also arbitrary transverse actions. See \cite{Goertsches-Toeben} for details.
	\end{rema}

\subsection{Local computations}
	Let $G$ be a compact torus and $G'$ a subtorus of $G$. Let $\h'$ be a complement of $\g'$ in $\g$, that is, $\h'$ is a linear subspace of $\g$ such that $\g' \oplus \h' = \g$. Let $Y$ be a smooth manifold equipped with an action of $G'$. We define the right $G'$-action on $G \times Y$ by 
	\begin{equation*}
		(g,y) \cdot g' := (gg', g'^{-1}\cdot y)
	\end{equation*}
	for $(g,y) \in G \times Y$ and $g' \in G$. We define the left $G$-action on $G \times Y$ by 
	\begin{equation*}
		\widetilde{g}\cdot (g,y) :=(\widetilde{g}g,y)
	\end{equation*}
	for $(g,y) \in G\times Y$ and $\widetilde{g} \in G$. 
	Then the left $G$-action on $G\times Y$ descends to the left $G$-action on the quotient manifold $G \times_{G'} Y$. Since $\g = \g' \oplus \h'$, we have a foliation $F$ on $G \times_{G'}Y$ whose leaves are generated by $\h'$. We denote by $\Omega_B^*(G \times_{G'}Y)$ the basic complex with respect to $F$.
	
	We denote 
		\begin{itemize}
			\item by $X_v^L$ the fundamental vector field on $G \times Y$ generated by $v \in \g$ with respect to the left $G$-action, 
		 	\item by $X_{v'}^R$ the fundamental vector field on $G \times Y$ generated by $v' \in \g'$ with respect to the right $G'$-action, 
			\item by $X^G_v$ the fundamental vector field on $G$ generated by $v \in \g$ and 
			\item by $X^Y_{v'}$ the fundamental vector field on $Y$ generated by $v' \in \g$. 
		\end{itemize}
		Let $\omega \in \Omega^*(G \times Y)$. We say that $\omega$ is \emph{right $\g'$-basic} if $i_{X_{v'}^R}\omega = 0$ and $L_{X_{v'}^R}\omega = 0$ for all $v' \in \g'$. We say that $\omega$ is \emph{left $\h'$-basic} if $i_{X_{u'}^L}\omega = 0$ and $L_{X_{u'}^L}\omega = 0$ for all $u' \in \h'$.
		Let $p_G \co G \times Y \to G$ and $p_Y \co G\times Y \to Y$ be the projections. Since $p_G$ and $p_Y$ both have global sections, the pull-back maps $p_G^*$ and $p_Y^*$ are injective. For a differential form $\alpha \in \Omega^*(G)$ and $\beta \in \Omega^*(Y)$, we have $p_G^*(\alpha) \wedge p_Y^*(\beta) \in \Omega^*(G \times Y)$ and such elements generates $ \Omega^*(G \times Y)$. Let $\alpha \in \Omega^*(G)$ and $\beta \in \Omega^*(Y)$. Then we have 
		\begin{equation*}
			i_{X_v^L}(p_G^*(\alpha) \wedge p_Y^*(\beta)) = p_G^*(i_{X_v^G}\alpha ) \wedge p_Y^*(\beta) 
		\end{equation*}
		for $v \in \g$ and 
		\begin{equation*}
			i_{X_{v'}^R}(p_G^*(\alpha) \wedge p_Y^*(\beta)) = p_G^*(i_{X_{v'}^G}\alpha ) \wedge p_Y^*(\beta) + (-1)^{\deg \alpha} p_G^*(\alpha) \wedge p_Y^*(i_{-X^Y_{v'}}\beta)
		\end{equation*}
		for $v' \in \g'$. 
		
	\begin{lemm}\Label{lemm:OmegaY}
		Let $G, G', \h', Y$ be as above. Then there exists an isomorphism $\eta \co \Omega_B^*(G \times_{G'}Y)^G \to \Omega^*(Y)^{G'}$ such that $\eta \circ i_{X_v} = i_{X_{v'}^Y}\circ \eta $ for $v = v' +u'$, $v' \in \g'$, $u' \in \h'$. 
	\end{lemm}
	\begin{proof}
		Let $\omega' \in \Omega_B^*(G \times_{G'}Y)^G$. Let $\pi \co G \times Y \to G\times_{G'} Y$ be the quotient map. Then the pull-back $\pi^*(\omega') \in \Omega^* (G\times Y)$ is $G$-invariant, right $\g'$-basic and left $\h'$-basic form. Conversely, such a differential form on $G \times Y$ descends to a differential form in $\Omega_B^*(G \times_{G'}Y)^G$. These correspondence gives an isomorphism between $\Omega_B^*(G \times_{G'}Y)^G$ and 
		\begin{equation*}
			A^* := \left\{ \omega \in \Omega^*(G \times Y)^G \mid \text{$\omega$ is left $\h'$-basic and right $\g'$-basic} \right\}.
		\end{equation*}
		Let $\omega \in A^*$. Suppose that 
		\begin{equation*}
			\omega = \sum_{i=1}^k p_G^*(\alpha_i) \wedge p_Y^*(\beta_i), \quad \alpha_i \in \Omega^*(G), \beta_i \in \Omega^*(Y). 
		\end{equation*}
		Since $\omega$ is $G$-invariant, by averaging with the $G$-action, we may assume that all $\alpha_i \in \Omega^*(G)$ is $G$-invariant. Since $\Omega^*(G)^G \cong  \bigwedge \g^*$ and the pull-back maps are injective, we have that $A^*$ is isomorphic to 
		\begin{equation*}
			\{ \omega \in \bigwedge \g^* \otimes \Omega^*(Y) \mid \text{$\omega$ is left $\h'$-basic and right $\g'$-basic}\}. 
		\end{equation*}
		But being left $\h'$-basic implies that this DGA is isomorphic to 
		\begin{equation*}
			\{ \omega \in \bigwedge (\g/\h')^* \otimes \Omega^*(Y) \mid \text{$\omega$ is right $\g'$-basic}\}.
		\end{equation*}
		Since $\g' \oplus \h' = \g$, we have 
		\begin{equation*}
			\{ \omega \in \bigwedge (\g/\h')^* \otimes \Omega^*(Y) \mid \text{$\omega$ is right $\g'$-basic}\} \cong \Omega^*(Y)^{G'}. 
		\end{equation*}
		Thus we obtain the isomorphism $\eta \co \Omega^*_B(G\times_{G'}Y)^G \to \Omega^*(Y)^{G'}$. It follows from the construction of $\eta$ that $\eta \circ i_{X_v} = i_{X_{v'}^Y} \circ \eta$ for $v = v' +u'$, $v' \in \g'$, $u' \in \h'$. The lemma is proved.
	\end{proof}
	\begin{lemm}\Label{lemm:local}
		Let $G, G', \h', Y$ be as above. We think of $S^*((\g/\h')^*)$ as a subalgebra of $S^*(\g^*)$ via the dual of the quotient map $\g \to \g/\h'$. We identify $\g'^*$ with $(\g/\h')^*$ via the decomposition $\g = \g' \oplus \h'$. Then the followings hold.
		\begin{enumerate}
			\item $H^*_B(G\times_{G'}Y)$ is isomorphic to $H^*(Y)$ as $\R$-algebras. 
			\item $H^*_G(G\times_{G'}Y)$ is isomorphic to $H^*_{G'}(Y)$ as $S^*((\g/\h')^*)$-algebras. 
		\end{enumerate}
	\end{lemm}
	\begin{proof}
		Part (1) follows from Lemmas \ref{lemm:cohomology} and  \ref{lemm:OmegaY} immediately. 
		
		We show Part (2). Let $\pi \co G \times_{G'} Y \to G/G'$ be the projection induced by the first projection $G \times Y \to G$. $\pi$ is a $G$-equivariant map. Since $H'$ acts on $G/G'$ transitively and local freely, there exists a $\h'$-valued $G$-invariant $1$-form $\theta$ on $G/G'$ such that $i_{X_v'}\theta = v$ for all $v \in \h'$, where $X_v'$ denotes the fundamental vector field on $G/G'$ generated by $v$. Since $\pi$ is $G$-equivariant, we have that the pull-back $\pi^*\theta$ is a $\h'$-valued $G$-invariant $1$-form on $G \times_{G'} Y$ such that $i_{X_v}\pi^*\theta = v$. Thus we have a Cartan operator $f \co S^*(\g') \otimes \Omega^*_{\h'}(G \times _{G'}Y) \to S^*(\g') \otimes \Omega^*_B(G \times _{G'}Y)$ that induces a cohomology isomorphism. By Lemma \ref{lemm:OmegaY} and facts $\Omega^*_{G}(G \times _{G'}Y) = S^*(\g'^*)\otimes \Omega^*_{\h'}(G \times _{G'}Y)$ and $\Omega^*_{G'}(Y) = S(\g'^*)\otimes \Omega^*(Y)^{G'}$, we have that $H^*_G(G\times_{G'}Y)$ is isomorphic to $H^*_{G'}(Y)$ as $S^*((\g/\h')^*)$-algebras, as required. 
	\end{proof}
\subsection{Global computations}
	Until the end of this subsection, let $(M,G,y) \in \mathcal{C}_1$. Let $F$ be the canonical foliation on $M$. Let $(\Delta, \h, G) = \F_1(M,G,y) \in \mathcal{C}_2$. We denote by $p \co \g^\C \to \g$ the projection, $\h' := p(\h)$ and $q \co \g \to \g/\h'$ the quotient map. We assume that $q(\Delta)$ is polytopal. In particular, $M$ admits a transverse K\"ahler form with respect to $F$.

	By Theorem \ref{theo:transverse}, there exist a $G$-invariant transverse K\"ahler form $\omega$ and a smooth function $h_v \co M \to \R$ such that $dh_v = -i_{X_v}\omega$ for any $v \in \g$. 
	By Proposition \ref{prop:crit} and Lemma \ref{lemm:generic}, there exists $v \in \g$ such that the set of critical points of $h_v$ coincides with the union $Z$ of minimal orbits of $M$. 
	Let $Z_1,\dots, Z_l$ be the minimal orbits of $M$. By renumbering them if necessary, we may assume that 
	\begin{equation*}
		h_v(Z_1) < h_v(Z_2) < \dots < h_v(Z_l).
	\end{equation*}
	Throughout this subsection, we fix such $v \in \g$. 
	
	By Proposition \ref{prop:slice}, there uniquely exists a $G^M$-invariant open neighborhood $U_j$ of $Z_j$ that is $G^M$-equivariantly biholomorphic to 
	\begin{equation*}
		G^M \times _{G_j^\C}\bigoplus_{i=1}^k \C_{\alpha_{j,i}}
	\end{equation*}
	where $G_j$ denotes the isotropy subgroup of $G$ at a point $x$ in $Z_j$ and $\alpha_{j,1},\dots, \alpha_{j,k} \in \Hom(G_j^\C,\C^*)$ are the weights of the $G_j^\C$-representation $T_xM/T_xZ_j$. We define 
	\begin{equation*}
		U_j^* := 
		\begin{cases}
			\emptyset & \text{for $j=1$}, \\
			\bigcup_{i=1}^{j-1} U_i \cap U_j & \text{for $j=2,\dots, l$}.
		\end{cases}
	\end{equation*}
	
	Let $J$ be the complex structure on $M$. We denote by $\gamma_t$ the partial flow of $JX_v$. For $x \in M$, consider the set 
	\begin{equation*}
		W_x := \bigcap _{t_0 \in \R} \overline{\{ \gamma_t(x) \mid t \geq t_0\}}.
	\end{equation*} 
	\begin{lemm}\Label{lemm:stable}
		For any $x \in M$, the set $W_x$ is a nonempty subset of a minimal orbit in $M$. 
	\end{lemm}
	\begin{proof}
		Let $(t_i)_{i=1,\dots}$ be a monotonic increasing sequence of real numbers diverges to $\infty$. Since $M$ is compact, the sequence $(\gamma_{t_i}(x))_{i=1,\dots}$ has a subsequence that converges to a point in $M$. This shows that $W_x$ is not empty. 
		
		Let $f \co M \to \R$ be a function defined by $f = \omega (X_v, JX_v)$. Since 
		\begin{equation*}
			L_{JX_v}h_v = i_{JX_v}dh_v = i_{JX_v} (-i_{X_v}\omega) = \omega(-X_v, JX_v) \leq 0, 
		\end{equation*} 
		the function $\R \to \R$ given by $t \mapsto h_v(\gamma_t(x))$ is monotone decreasing. Since $M$ is compact, we have that $h_v$ is bounded. Therefore we have 
		\begin{equation*}
			\lim_{t \to \infty} \frac{d}{dt} h_v(\gamma_t(x)) = 0.
		\end{equation*}
		Suppose that the sequence $(\gamma_{t_i}(x))_{i=1,\dots}$ converges to a point $x_0 \in M$. Since $f$ is continuous, we have $\lim_{i \to \infty} f(\gamma_{t_i}(x)) = f(x_0)$. 
		On the other hand, we have 
		\begin{equation*}
			\lim_{i \to \infty} f(\gamma_{t_i}(x)) = \lim_{t \to \infty} f(\gamma_t(x)) = \lim_{t \to \infty} \frac{d}{dt} h_v(\gamma_t(x)) = 0.
		\end{equation*}
		Therefore we have $f(x_0) = 0$. Since $f(x') = 0$ if and only if $(X_v)_{x'} \in T_{x'}F$ for $x' \in M$, we have $x_0 \in Z_j$ for some $j$. On the other hand, $\lim_{t \to \infty} h_v(\gamma_t(x)) = h_v(x_0) = h_v(Z_j)$. 
		Since $\lim_{t \to \infty} h_v(\gamma_t(x))$ does not depend on the sequence $(\gamma_{t_i}(x))_{i=1,\dots}$ and $h_v(Z_i) \neq h_v(Z_j)$ if $i\neq j$, we have $W_x \subset Z_j$. 
		The lemma is proved. 
	\end{proof}
	\begin{lemm}\Label{lemm:stratification}
		Let $x \in M$. For $j=1,\dots, l$, the followings hold.
		\begin{enumerate}
			\item If $W_x \subset Z_j$, then $x \in U_j$. 
			\item If $x \in U_j$, then $W_x \subset \bigcup_{i \leq j}Z_i$.
		\end{enumerate}
	\end{lemm}
	\begin{proof}
		Suppose that $W_x \subset Z_j$. Since $U_j$ is a neighborhood of $Z_j$, we have that there exists $t \in \R$ such that $\gamma_t(x) \in U_j$. Since $U_j$ is $G^M$-invariant, we have $x \in U_j$, proving (1). 
		
		Suppose that $j=l$.  Since $Z = Z_1\sqcup \dots \sqcup Z_l$, we have $W_x \subset \bigcup_{i\leq j}Z_i$ by Lemma \ref{lemm:stable}. 
		For $j=1,\dots, l-1$, we define
		\begin{equation*}
			V_j := \{ x' \in U_j \mid |h_v(x') - h_v(Z_j)| < h_v(Z_{j+1}) -h_v(Z_j) \}.
		\end{equation*}
		$V_j$ is an open neighborhood of $Z_j$. Since $V_j \subset U_j$ and $U_j$ is the minimal $G^M$-invariant neighborhood of $Z_j$, we have 
		\begin{equation*}
			U_j = \bigcup_{g \in G^M} g \cdot V_j.
		\end{equation*}
		Suppose that $x \in U_j$. Then we have that there exists $g \in G^M$ such that $g\cdot x \in V_j$. Thus we have $|h_v(g\cdot x) - h_v(Z_j)| < h_v(Z_{j+1})-h_v(Z_j)$. In particular, we have $h_v(g\cdot x) < h_v(Z_{j+1})$. Since $\gamma_t$ is $G^M$-equivariant, we have 
			$g \cdot W_x = W_{g\cdot x}$. Since each minimal orbit is $G^M$-invariant, we have that $W_x$ and $g \cdot W_{x}$ is contained in the same minimal orbit. 
		Therefore we have 
		\begin{equation*}
			h_v(W_x) = \lim_{t \to \infty} h_v(\gamma_t(x)) = \lim_{t\to \infty}h_v(\gamma_t(g\cdot x)) \leq  h_v(g\cdot x) < h_v(Z_{j+1})
		\end{equation*}
		because $t \mapsto h_v(\gamma_t(x))$ is monotone decreasing. It turns out that $W_x \subset \bigcup_{i\leq j}Z_i$, proving (2). 
	\end{proof}
	\begin{lemm}\Label{lemm:Uj*}
		Let $v_j \in \g_j$ and $v' \in \h'$ such that $v = v_j + v'$. Then, 
		$U_j^*$ is $G^M$-equivariantly biholomorphic to $G_M \times_{G_j^\C} Y_j$, where 
		\begin{equation*}
			\begin{split}
				Y_j &= \bigoplus_{\langle d\alpha_{j,i},v_j\rangle \leq 0} \C_{\alpha_j,i} \times \left( \left(\bigoplus_{\langle d\alpha_{j,i},v_j\rangle >0}\C_{\alpha_{j,i}}\right) \setminus \{0\}\right) \subset \bigoplus_{i=1}^k\C_{\alpha_{j,i}}.
			\end{split}
		\end{equation*}
	\end{lemm}
	\begin{proof}
		Let $x \in U_j$. By Lemma \ref{lemm:stratification} (2), we have $W_x \subset \bigcup_{i\leq j}Z_i$. If $W_x \subset \bigcup_{i<j}Z_i$, then $x \in \bigcup_{i<j}U_i$ by Lemma \ref{lemm:stratification} (1). Therefore we have 
		\begin{equation*}
			U_j^* = U_j \setminus \{ x \in U_j \mid W_x \subset Z_j\}.
		\end{equation*}
		Let $\phi \co U_j \to  G^M \times_{G_j^\C} \bigoplus_{i=1}^k \C_{\alpha_{j,i}}$ be a $G^M$-equivariant biholomoprhism. Suppose that $[g,w_1,\dots, w_k] = \phi (x)$. Then we have 
		\begin{equation*}
			\begin{split}
				\phi(\gamma_t(x)) &= [\exp_{G^M}(tJv) g, w_1,\dots, w_k] \\
				&=  [\exp_{G^M}(tJv') g, e^{t\langle d\alpha_{j,1},v_j\rangle} w_1,\dots, e^{t\langle d\alpha_{j,k},v_j\rangle}w_k]. 
			\end{split}
		\end{equation*}
		Since $\phi(Z_j)$ is represented as $\{ [g,w_1,\dots, w_k] \mid w_1=\dots =w_k=0\}$, we have that $W_x \subset Z_j$ if and only if $e^{t\langle d\alpha_{j,1},v_j\rangle} w_1,\dots, e^{t\langle d\alpha_{j,k},v_j\rangle}w_k$ converge to $0$ as $t$ reaches $\infty$. It turns out that $W_x \subset Z_j$ if and only if $w_i = 0$ for $\langle d\alpha_{j,i},v_j\rangle >0$, proving the lemma. 
	\end{proof}
	Now we are in a position to compute the basic betti numbers. A similar result can be found in \cite[Theorem 3.1]{Battaglia-Zaffran} for certain LVMB manifolds. 
	\begin{prop}\Label{prop:oddvanish}
		The followings hold.
		\begin{enumerate}
			\item $H^{\text{odd}}_B\left(\bigcup_{i\leq j}U_i\right) = 0$ for all $j=1,\dots, l$. 
			\item $\dim H^{2d}_B(M)$ coincides with the number of critical submanifolds whose co-index is $2d$, where co-index is defined to be the index of $-h_v$. 
		\end{enumerate}
	\end{prop}
	\begin{rema}
		\begin{enumerate}
			\item By the Poincar\'e duality of the basic cohomology and Proposition \ref{prop:oddvanish} (2), we have that the number of critical submanifolds whose co-index is $2d$ coincides with the number of critical submanifolds whose index is $2d$.
			\item The basic cohomologies of critical submanifolds $Z_1,\dots, Z_l$ of $h_v$ are trivial. Proposition \ref{prop:oddvanish} means that the basic Morse-Bott function $h_v$ is perfect. 
			\item In case when closed leaves of $F$ are $Z_1,\dots, Z_l$ only, we can apply \cite[Theorem 6.4]{Goertsches-Toeben} to $(M,F)$ and we obtain Proposition \ref{prop:oddvanish} (2) immediately. However, not every $M$ satisfies this condition. 
			\item Proposition \ref{prop:oddvanish} follows from the Morse inequality for the basic cohomology shown in \cite{Alvarez} immediately. We will give an elementary proof of Proposition \ref{prop:oddvanish} below by using Mayer-Vietoris. 
		\end{enumerate}
	\end{rema}
	\begin{proof}[Proof of Proposition \ref{prop:oddvanish}]
		By Mayer-Vietoris, we have a long exact sequence 
		\begin{equation*}
			\dots \to H_B^{q}\left(\bigcup_{i\leq j}U_i\right) \to H^{q}_B\left(\bigcup_{i\leq j-1}U_i\right)\oplus H^{q}_B\left(U_j\right) \to H^{q} _B(U_j^*) \to \cdots .
		\end{equation*}
		By Lemma \ref{lemm:local}, we have $H_B^q(U_j) = 0$ and 
		\begin{equation*}
			H^{q-1}_B(U_j^*)\to H^q_B\left(\bigcup_{i\leq j}U_i\right)  \to H^q_B\left(\bigcup_{i< j}U_i\right) \to H^{q} _B(U_j^*)
		\end{equation*}
		is exact for $q \geq 1$. By Lemmas \ref{lemm:local} and \ref{lemm:Uj*}, we have 
		\begin{equation*}
			H_B^q(U_j^*) \cong \begin{cases}
				\R & \text{if $q=0$ or the co-index of  $Z_j$ is $q+1$},\\
				0 & \text{otherwise}.\\
			\end{cases}
		\end{equation*}
		Thus we have $H^q_B\left(\bigcup_{i\leq j}U_i\right) \cong H^q_B\left(\bigcup_{i\leq j-1}U_i\right)$ unless $q$ and $q+1$ are not the co-index on $Z_j$. If $q$ is the co-index of $Z_j$, then we have an exact sequence 
		\begin{equation}\Label{eq:exact}
				  H^{q-1} \left( \bigcup_{i\leq j-1} U_i\right) \to \R \to  H^q \left( \bigcup_{i\leq j} U_i\right) \to H^q \left( \bigcup_{i\leq j-1} U_i\right) \to 0. 
		\end{equation}
		We show that $H^{\text{odd}}\left( \bigcup_{i\leq j} U_i\right) = 0$ by induction on $j$. If $j = 1$, then $H^{2d-1}(U_1)=0$ for any $d \in \mathbb{N}$. Suppose that $H^{\text{odd}}\left( \bigcup_{i\leq j-1} U_i\right) = 0$. Since $H^{2d-1}_B\left(\bigcup_{i\leq j}U_i\right) \cong H^{2d-1}_B\left(\bigcup_{i\leq j-1}U_i\right)$ unless $2d-1$ and $2d$ are not the co-index on $Z_j$ and all co-indices are even, we have $H^{2d-1}\left( \bigcup_{i\leq j}U_i\right) =0$ for $d \in \mathbb{N}$ except the case when $2d$ is the co-index on $Z_j$. If $2d$ is the co-index on $Z_j$, by \eqref{eq:exact} and induction hypothesis we have $H^{2d-1}\left( \bigcup_{i\leq j}U_i\right) = 0$, showing (1).  
		
		By \eqref{eq:exact}, we have a short exact sequence 
		\begin{equation*}
			0 \to \R \to H^{2d}\left(\bigcup_{i\leq j}U_i\right) \to H^{2d}\left(\bigcup_{i\leq j-1}U_i\right) \to 0.
		\end{equation*}
		if $2d$ is the co-index of $Z_j$. Therefore we  have that $\dim H^{2d}_B(M)$ coincides with the number of critical submanifolds whose co-index is $2d$, showing (2). 
	\end{proof}
	Let $\g' \subset \g$ be a complement of $\h'$ in $\g$. 
	Then the decomposition $\g = \g'\oplus \h'$ induces a decomposition $S^*(\g^*) = S^*(\g'^*) \otimes S^*(\h'^*)$.
	This allows us to think of elements in $\Omega^*_G(M)$ as $\Omega^*_{\h'}(M)$-valued polynomial functions on 
	$\g'$. Namely, we have the decomposition $\Omega^*_G(M) = S^*(\g'^*) \otimes \Omega^*_{\h'}(M)$. 
	This gives us to define the following double complex structure. 
	For $p, q \in \Z$, we set
	\begin{equation*}
		C^{p,q} := S^p(\g'^*) \otimes \Omega^{q-p}_{\h'}(M)
	\end{equation*}
	so that $\Omega^*_G(M) = \bigoplus_{p,q\in\Z} C^{p,q}$.
	Define $d \co C^{p,q} \to C^{p,q+1}$ by $d := 1\otimes d_{\h'}$ and $\delta\co C^{p,q} \to C^{p+1,q}$ by 
	\begin{equation*}
		((\delta \alpha )(v'))(u')= -i_{X_{v'+u'}}((\alpha(v'))(u'))
	\end{equation*}
	for $\alpha \in \Omega^*_G(M)$ and $v' \in \g'$ and $u'\in \h'$. Here, $\alpha$ is regarded as an $\Omega^*_{\h'}(M)$-valued polynomial functions on $\g'$. 
	Then we have $d_G = d+\delta$ and $d^2=\delta^2= d\delta+\delta d =0$. 
	\begin{prop}
		The $E_1$ term of the spectral sequence of the double complex $(\bigoplus C^{p,q}, d, \delta)$ is 
		\begin{equation*}
			E_1 = S^*(\g'^*)\otimes H^*(\Omega^*_{\h'}(M),d_{\h'}). 
		\end{equation*}
		More explicitly, 
		\begin{equation*}
			E_1^{p,q} = S^p(\g'^*)\otimes H^{q-p}(\Omega^*_{\h'}(M),d_{\h'}).
		\end{equation*}
	\end{prop}
	\begin{proof}
		This follows from the definition of $d$ immediately. 
	\end{proof}
	\begin{prop}
		The spectral sequence of the double complex $(\bigoplus C^{p,q}, d, \delta)$ collapses at the $E_1$ term. 
	\end{prop}
	\begin{proof}
		By Theorem \ref{theo:Cartan}, we have that $H^*(\Omega^*_{\h'}(M),d_{\h'})$ is isomorphic to $H^*_B(M)$. By Proposition \ref{prop:oddvanish}, we have $H^{\text{odd}}_B(M)=0$. It turns out that $E^{p,q}_1 =0$ when $p+q$ is odd. Therefore the spectral sequence of $(\bigoplus C^{p,q}, d, \delta)$ collapses at the $E_1$ term. 
	\end{proof}
	As a conclusion, we have the following:
	\begin{theo}\Label{theo:eqformal}
		$H^*_G(M)$ is isomorphic to $S^*(\g'^*) \otimes H^*(\Omega^*_{\h'}(M),d_{\h'})$ as $S^*(\g'^*)$-modules. 
	\end{theo}
	\begin{rema}\Label{rema:eqformal}
		The assertion of Theorem \ref{theo:eqformal} is equivalent to that the transverse action of $\g'$ on $(M,F)$ is \emph{equivariantly formal} in the sense of \cite{Goertsches-Toeben}. 
	\end{rema}
	Theorem \ref{theo:eqformal} yields that the DGA homomorphism $\Omega^*_G(M) \to \Omega^*_{\h'}(M)$ given by $\alpha \mapsto \alpha(0)$ induces a surjective homomorphism $H_G^*(M) \to H^*(\Omega^*_{\h'}(M),d_{\h'})$ and its kernel is the ideal generated by elements in $S^2(\g'^*) = \g'^*$. Since the Cartan operator $f \co \Omega^*_{\h'}(M) \to \Omega^*_B(M)$ given as Lemma \ref{lemm:Cartan} above induces an isomorphism $H^*(\Omega^*_{\h'}(M),d_{\h'}) \to H^*_B(M)$, we have the following. 
	\begin{theo}\Label{theo:basicforgetful}
		We think of elements in $\Omega_G^*(M)$ as $\Omega^*(M)^G$-valued polynomial functions on $\g$. Then, the map $\Omega^*_G(M) \to \Omega^*_B(M)$ given by 
		\begin{equation*}
			\alpha \mapsto f(\alpha|_{\h'}) \quad \text{for $\alpha \in \Omega^*_G(M)$}
		\end{equation*}
		induces a surjective homomorphism $\forB \co H^*_G(M) \to H^*_B(M)$. The kernel of $\forB$ is generated by elements in $q^*((\g/\h')^*) \subset \g^* = S^2(\g^*)$.
	\end{theo}
\subsection{Localization to minimal orbits}
	We use the same notations as the last subsection. 
	\begin{prop}\Label{prop:localization}
		Let $\iota_j \co \bigcup_{i\leq j}Z_i \to \bigcup_{i\leq j}U_i$ be the inclusion. Then the induced homomorphism 
		\begin{equation*}
			\iota_j^* \co H^*_G\left(\bigcup_{i\leq j}U_i\right) \to H^*_G\left(\bigcup_{i\leq j}Z_i\right) = \bigoplus_{i \leq j} H^*_G(Z_i)
		\end{equation*}
		is injective. 
	\end{prop}
	\begin{proof}
		It is obvious that $H^0_G\left(\bigcup_{i\leq j}U_i\right) \to H^0_G\left(\bigcup_{i\leq j}Z_i\right)$ is injective. Let $q>0$. Induction on $j$.  By Proposition \ref{prop:slice} and Lemma \ref{lemm:local}, we have $H^q_G(U_1) \cong H^q_G(Z_1)$. Suppose that the proposition holds for $j-1$. 
		By Mayer-Vietoris sequences of equivariant cohomologies, we have a commutative diagram
		\begin{equation*}
			\xymatrix{
				H^{q-1}_G(U_j^*) \ar[r] \ar[d] & H^q_G(\bigcup_{i\leq j}U_i) \ar[r] \ar[d] ^{\iota_{j}^*}& H^q_G(\bigcup_{i\leq j-1}U_i) \oplus H^q_G(U_j)\ar[r] \ar[d] & H^q_G(U_j^*) \ar[d]\\
				H^{q-1}_G(\emptyset ) \ar[r] & H^q_G(\bigcup_{i\leq j}Z_i) \ar[r] & H^q_G(\bigcup_{i\leq j-1}Z_j) \oplus H^q_G(Z_j)\ar[r] & H^q_G(\emptyset)
			}
		\end{equation*}
		whose horizontal lines are exact. Since $q>0$, we have that $H^{q-1}_G(\emptyset)$ and $H^q_G(\emptyset)$ vanish. Thus it is enough to show that $H_G^{\text{odd}}(\bigcup_{i\leq j}U_i) = 0$ and $H_G^{\text{odd}}(U_j^*) = 0$. 
		By Lemmas \ref{lemm:local}, \ref{lemm:Uj*} and \cite[Corollary 2.18]{Buchstaber-Panov}, we have $H_G^{\text{odd}}(U_j^*) = 0$. We show that $H_G^{\text{odd}}(\bigcup_{i\leq j}U_i) = 0$ by induction on $j$. 
		Since $H^*_G(U_1) \cong H^*_G(Z_1) \cong S^*((\g/\h')^*)$, we have $H_G^{\text{odd}}(U_1) = 0$. Suppose that $H^{2d+1}_G(\bigcup_{i\leq j-1} U_i) = 0$ for all nonnegative integer $d$. By Mayer-Vietoris, we have an exact sequence 
		\begin{equation}\Label{eq:oddeq}
			H^{2d}_G\left(\bigcup_{i\leq j-1}U_i\right) \oplus H^{2d}_G(U_j) \to H^{2d}_G(U_j^*) \to H^{2d+1}_G\left(\bigcup_{i\leq j}U_i\right) \to 0. 
		\end{equation}
	Since the inclusion $U_j^* \to U_j$ induces a surjective map $H^*_G(U_j) \to H^*_G(U_j^*)$, we have that the first arrow $H^{2d}_G\left(\bigcup_{i\leq j-1}U_i\right) \oplus H^{2d}_G(U_j) \to H^{2d}_G(U_j^*)$ in \eqref{eq:oddeq} is surjective. Therefore we have that the second arrow $H^{2d}_G(U_j^*) \to H^{2d+1}_G\left(\bigcup_{i\leq j}U_i\right)$ in \eqref{eq:oddeq} is $0$. Therefore $H^{2d+1}_G\left(\bigcup_{i\leq j}U_i\right) =0$, as required. 
	\end{proof}
	\begin{coro}\Label{coro:localization}
		Let $\iota \co Z \to M$ be the inclusion. Then the induced homomorphism $\iota^* \co H^*_G(M) \to H^*_G(Z)$ is injective. 
	\end{coro}
	\begin{proof}
		It follows from Proposition \ref{prop:localization} immediately. 
	\end{proof}
	We shall see the degree $2$ part of $\iota^*$ for later use. Since 
	\begin{equation*}
		\Omega^2_G(M) = S^0(\g^*) \otimes \Omega^2(M)^G \oplus S^2(\g^*)\otimes \Omega^0(M)^G
	\end{equation*}
	and $S^2(\g^*) = \g^*$, we have that every element in $\Omega^2_G(M)$ can be written as $1\otimes \beta + \Psi$
	for some $\beta \in \Omega^2(M)^G$ and $G$-invariant smooth map $\Psi \co M \to \g^*$. Let $Z_j$ be a connected component of $Z$ and $\kappa_j \co Z_j \to M$ the inclusion. 
	\begin{lemm}\Label{lemm:localizationminimal}
		Let $1\otimes \beta + \Psi \in \Omega^2_G(M)$. If $\beta$ is basic with respect to $F$, then $\Psi |_{Z_j}$ is constant and 
		\begin{equation*}
			\kappa_j^*(1\otimes \beta + \Psi ) = \Psi(Z_j) \otimes 1. 
		\end{equation*}
	\end{lemm}
	\begin{proof}
		Since $Z_j$ is a connected component of $Z$, we have that $Z_j$ is a minimal orbit by Proposition \ref{prop:crit}. Thus $\beta|_{Z_j} = 0$ and $\Psi|_{Z_j}$ is constant. 
	\end{proof}
\section{The Danilov-Jurkiewicz type formula}\Label{sec:DJ}
	Let $(M,G,y) \in \mathcal{C}_1$. Let $F$ be the canonical foliation of $M$. Let $(\Delta, \h, G) = \F_1(M,G,y) \in \mathcal{C}_2$. As before, we denote by $p \co \g^\C \to \g$ the projection and by $q \co \g \to \g/p(\h)$ the quotient map. Let $(\widetilde{V}, \widetilde{\Gamma}, \widetilde{\Delta}, \widetilde{\lambda}) = \widetilde{\F}_1(M,G,y) \in \widetilde{\mathcal{C}_2}$. The purpose of this subsection is to describe the basic cohomology $H^*_B(M)$ explicitly in terms of the corresponding marked fan $(\widetilde{V}, \widetilde{\Gamma}, \widetilde{\Delta}, \widetilde{\lambda})$. 
	
	Suppose that $\Delta$ has $m$ $1$-cones $\rho_1,\dots, \rho_m$. Let $K$ be the abstract simplicial complex given by 
	\begin{equation*}
		K := \{ \{i_1,\dots, i_k\} \subset \{1,\dots, m\} \mid \rho_{i_1}+\dots +\rho_{i_k} \in \Delta\}.
	\end{equation*}
	$K$ is the underlying simplicial complex of $\Delta$. Let $\lambda(\rho_i)$ be the primitive generator of $\rho_i$ for $i=1,\dots, m$. We take a positive integer $N \in \mathbb{N}$ and a linear map $\varphi \co \R^N \to \g$ such that $\varphi(e_i) = \lambda (\rho_i)$ for $i=1,\dots, m$, where $e_i$ denotes the standard basis vector of $\R^N$, $\varphi(\Z^N) = \ker \exp_G$ and $\dim \ker \varphi$ is even. These assumptions imply that $\varphi \co \R^N \to \g$ induces a surjective homomorphism $\alpha \co \R^N/\Z^N \to G$ whose kernel is an even dimensional subtorus. Put $\g_0 = \R^N$ and $G_0 = \R^N/\Z^N$. 
	The collection of cones 
	\begin{equation*}
		{\Delta}_0 := \{ \R_{\geq 0}e_{i_1} + \dots + \R_{\geq 0}e_{i_k} \mid \{i_1,\dots, i_k\} \in K\}
	\end{equation*}
	is a nonsingular fan in $\g_0$ with respect to the lattice $\ker \exp_{G_0} = \Z^N$. $K$ is also the underlying simplicial complex of $\Delta_0$. 
	Let $v_1,\dots, v_l$ be $\C$-basis vectors of $\h$. Then, $p(v_1),\dots, p(v_l), p(-\sqrt{-1}v_1), \dots, p(-\sqrt{-1}v_l)$ are $\R$-basis vectors of $p(\h)$. Since $\varphi$ is surjective, there exist $b_j, b_j' \in \g_0$ such that $\varphi(b_j) = p(v_j)$, $\varphi(b_j') = p(-\sqrt{-1}v_j)$ for $j=1,\dots, l$. Let $c_1,\dots, c_{l'}, c_1', \dots, c_{l'}'$ be $\R$-basis vectors of $\ker \varphi$. Let ${\h}_0$ be the $\C$-subspace of $\C^N$ spanned by $b_1+\sqrt{-1}b_1',\dots, b_l+\sqrt{-1}b_l'$, $c_1+\sqrt{-1}c_1'$, \dots, $c_{l'}+\sqrt{-1}c_{l'}'$. Then by definition, we have $({\Delta}_0, {\h}_0, G_0) \in \mathcal{C}_2$
	and $\alpha \in \Hom_{\mathcal{C}_2}(({\Delta}_0, {\h}_0, G_0), (\Delta, \h,G))$. Let $(M_0,G_0,y_0) = \F_2(\Delta_0, \h_0, G_0)$. 
	\begin{lemm}\Label{lemm:M0M}
		There exists $f \co M_0 \to M$ satisfying the followings.
		\begin{enumerate}
		\setcounter{enumi}{-1}
			\item $(f,\alpha) \in \Hom_{\mathcal{C}_1}((M_0,G_0,y_0),(M,G,y))$. 
			\item $\ker \alpha$ is connected.
			\item $f \co M_0 \to M$ is a principal $\ker \alpha$-bundle. 
		\end{enumerate}
		In particular, $(M,G,y)$ and $(M_0,G_0,y_0)$ are equivalent, see Definition \ref{defn:equivalent}. 
	\end{lemm}
	\begin{proof}
		It follows from Theorem \ref{theo:principal} and the definition of $(\Delta_0,\h_0,G_0)$ immediately. 
	\end{proof}
	$M_0$ is the quotient space $X(\Delta_0)/H_0$ of the toric variety $X(\Delta_0)$ by the action of $H_0 = \exp_{G^\C_0} (\h_0)$. By definition of $\Delta_0$, the toric variety $X(\Delta_0)$ is the complement of the coordinate subspace arrangement 
	\begin{equation*}
		X(\Delta_0) = \C^N \setminus \bigcup_{I \notin K}L_I, 
	\end{equation*}
	where 
	\begin{equation*}
		L_ I = \{ (z_1,\dots, z_N) \in \C^N \mid z_i = 0 \text{ for all $i \in I$}\}. 
	\end{equation*}
	The action of $G_0 = (S^1)^N$ on $\C^N$ is given by coordinatewise multiplication. 
	Let $x_1,\dots, x_N$ be the dual basis vectors of $e_1,\dots, e_N \in \g_0$. Then 
	\begin{equation*}
		H^*_{G_0}(\C^N) \cong S^*(\g_0^*) = \R[x_1,\dots, x_N]. 
	\end{equation*}
	By \cite[Corollary 2.18]{Buchstaber-Panov} and \cite[Theorem 4.8 and Remark 4.10]{Davis-Januskiewicz}, we have that 
	the inclusion $X(\Delta_0) \to \C^N$ induces a surjective homomorphism $\R[x_1,\dots, x_N] \to H^*_{G_0} (X(\Delta_0))$ and 
	the kernel $\mathcal{I}$ is the Stanley-Reisner ideal $\mathcal{I} = \langle x_{i_1}\dots x_{i_k} \mid \{i_1,\dots, i_k\} \notin K \rangle$. 
	In particular, $H^*_{G_0} (X(\Delta_0)) \cong \R[x_1,\dots, x_N]/\mathcal{I}$. 
	
	Since $H_0$ is contractible and the action of $H_0$ on $X(\Delta_0)$ is free, we have $H^*_{G_0}(M_0) \cong H^*_{G_0} (X(\Delta_0))$. To describe $H^*_B(M_0)$, we shall see the image of $(\g_0/p_0(\h_0))^*$ by the dual of the quotient map $q_0 \co \g_0\to \g_0/p_0(\h_0)$, where $p_0 \co \g_0^\C \to \g_0$ denotes the projection. 
	Let $\widetilde{x} \in (\g_0/p_0(\h_0))^*$. Then 
	\begin{equation}\Label{eq:q0y}
		q_0^*(\widetilde{x})  = \sum_{i=1}^N \langle q_0^*(\widetilde{x}),e_i\rangle x_i = \sum_{i=1}^N \langle \widetilde{x}, q_0(e_i)\rangle x_i.
	\end{equation}
	\begin{prop}\Label{prop:DJ}
		Let $(M_0,G_0,y_0) \in \mathcal{C}_1$ be as above. Let $(\widetilde{V}_0, \widetilde{\Gamma}_0, \widetilde{\Delta}_0, \widetilde{\lambda}_0) = \widetilde{\F}_1(M_0,G_0,y_0) \in \widetilde{\mathcal{C}}_2$ be the corresponding marked fan. Let $\widetilde{\rho}_1, \dots, \widetilde{\rho}_m$ be $1$-cones of $\widetilde{\Delta}_0$. Suppose that $M_0$ is transverse K\"ahler with respect to the canonical foliation $F_0$ on $M_0$. Then we have an isomorphism 
		\begin{equation*}
			H^*_B(M_0) \cong \R[x_1,\dots, x_m]/\mathcal{I}' + \mathcal{J}', 
		\end{equation*}
		where the degree of $x_i$ is $2$ for $i=1,\dots, m$, $\mathcal{I}'$ is the Stanley-Reisner ideal of the underlying simplicial complex of $\widetilde{\Delta}_0$, that is, 
		\begin{equation*}
			\mathcal{I}' = \langle x_{i_1}\dots x_{i_k} \mid \widetilde{\rho}_{i_1} + \dots +\widetilde{\rho}_{i_k} \notin \widetilde{\Delta}_0\rangle 
		\end{equation*}
		and 	
		\begin{equation*}
			\mathcal{J}' = \left\langle \sum_{i=1}^m \left\langle \widetilde{x}, \widetilde{\lambda}_0(\widetilde{\rho}_i) \right\rangle x_i \mid \widetilde{x} \in \widetilde{V}_0^*\right\rangle .
		\end{equation*}
	\end{prop}
	\begin{proof}
		By Theorem \ref{theo:basicforgetful}, there exists a surjective homomorphism 
		\begin{equation*}
			\forB \co H^*_G(M_0) \cong \R[x_1,\dots, x_N]/\mathcal{I}  \to H^*_B(M_0)
		\end{equation*} whose kernel is generated by elements in the image of $(\g_0/p_0(\h_0))^*$ by the dual of the quotient map $q_0 \co \g_0 \to \g_0/p_0(\h_0)$, where $\mathcal{I}$ denotes the Stanley-Reisner ideal. By \eqref{eq:q0y}, we have that $\ker \forB$ is the image of the ideal 
		\begin{equation*}
			\mathcal{J} = \left\langle \sum_{i=1}^N \left\langle \widetilde{x}, q_0(e_i) \right\rangle x_i \mid \widetilde{x} \in \widetilde{V}_0^*\right\rangle.
		\end{equation*}
		
		By renumbering, we may assume that $\widetilde{\rho}_i = q_0(\R_{\geq 0}e_i)$ for $i=1,\dots, m$ without loss of generality. Then, the underlying simplicial complex $K$ of $\Delta$ is also the one of $\widetilde{\Delta}_0$. More precisely, $\{i_1,\dots, i_k\} \in K$ if and only if $1\leq i_j\leq m$ for all $j$ and $\widetilde\rho_{i_1}+ \dots + \widetilde{\rho}_{i_k} \in \widetilde{\Delta}_0$. Since singletons $\{m+1\},\dots, \{N\} \notin K$, we have $x_{m+1},\dots, x_{N} \in \mathcal{I}$. The images of $\mathcal{I}$ and $\mathcal{J}$ by the quotient map
		\begin{equation*}
			\R[x_1,\dots, x_N] \to \R[x_1,\dots, x_N]/\langle x_{m+1},\dots, x_{N}\rangle = \R[x_1,\dots, x_m]
		\end{equation*}
		are nothing but $\mathcal{I}'$ and $\mathcal{J}'$, respectively. The proposition is proved. 
	\end{proof}
	The following formula is well known as the theorem of Danilov and Jurkiewicz in case of complete nonsingular toric varieties. 
	\begin{theo}\Label{theo:DJ}
		Let $(M,G,y) \in \mathcal{C}_1$. Let $(\widetilde{V},\widetilde{\Gamma}, \widetilde{\Delta}, \widetilde{\lambda}) =\widetilde{F}_1(M,G,y) \in \widetilde{C}_2$. Let $\widetilde{\rho}_1,\dots, \widetilde{\rho}_m$ be $1$-cones of $\widetilde{\Delta}$. Suppose that $M$ is transverse K\"ahler with respect to the canonical foliation $F$ on $M$. Then we have an isomorphism 
		\begin{equation*}
			H^*_B(M) \cong \R[x_1,\dots, x_m]/\mathcal{I}'+\mathcal{J}',
		\end{equation*}
		where $\mathcal{I}'$ is the Stanley-Reisner ideal of the underlying simplicial complex of $\widetilde{\Delta}$, that is, 
		\begin{equation*}
			\mathcal{I}' = \langle x_{i_1}\dots x_{i_k} \mid \widetilde{\rho}_{i_1} + \dots +\widetilde{\rho}_{i_k} \notin \widetilde{\Delta}\rangle 
		\end{equation*}
		and 	
		\begin{equation*}
			\mathcal{J}' = \left\langle \sum_{i=1}^m \left\langle \widetilde{x}, \widetilde{\lambda}(\widetilde{\rho}_i) \right\rangle x_i \mid \widetilde{x} \in \widetilde{V}^*\right\rangle .
		\end{equation*}
	\end{theo}
	\begin{proof}
		Let $(M_0,G_0,y_0) \in \mathcal{C}_1$ be as above. By Lemma \ref{lemm:M0M}, we have that $(M,G,y)$ and $(M_0,G_0,y_0)$ are equivalent. Let $F$ and $F_0$ be the canonical foliations on $M$ and $M_0$, respectively. By Theorem \ref{theo:equivalence}, we have that $(M,F)$ and $(M_0,F_0)$ are transversally equivalent. Thus we have an isomorphism $H^*_B(M_0) \to H^*_B(M)$. 
		
		It follows from Theorem \ref{theo:fundamental} that $\widetilde{F}_1(M,G,y)$ and $\widetilde{F}_1(M_0,G_0,y_0)$ are isomorphic. This together with Proposition \ref{prop:DJ} yields that $H^*_B(M_0)$ and  $H^*_B(M)$ have the same description. 
	\end{proof}
	\begin{coro}\Label{coro:degree2}
		Let $(M,G,y) \in \mathcal{C}_1$. Let $F$ be the canonical foliation on $M$. Suppose that $M$ is transverse K\"ahler with respect to $F$. Then, the basic cohomology algebra $H^*_B(M)$ is generated by degree $2$ elements. 
	\end{coro}
\section{Basic Dolbeault cohomology}\Label{sec:basicDolbeault}
	Let $(M,G,y) \in \mathcal{C}_1$. Let $F$ be the canonical foliation on $M$. The canonical foliation $F$ on a compact connected complex manifold $M$ is homologically orientable (see \cite[Proposition 4.8]{Ishida-Kasuya}). Suppose that $M$ is transverse K\"ahler with respect to $F$. Then we have the Hodge decomposition 
	\begin{equation*}
		H^r_B(M)\otimes \C = \bigoplus_{p+q=r}H^{p,q}_B(M)
	\end{equation*}
	(see \cite{EKA} for detail). The purpose of this section is to show the following.
	\begin{theo}\Label{theo:Dolbeault}
		Let $(M,G,y) \in \mathcal{C}_1$. Let $F$ be the canonical foliation on $M$. Suppose that $M$ is transverse K\"ahler with respect to $F$. Then,  
		\begin{equation*}
			H^{p,q}_B(M) = \begin{cases}
				0 & \text{if $p \neq q$}, \\
				H^{2p}_B(M)\otimes \C & \text{if $p=q$}.
			\end{cases}
		\end{equation*}
	\end{theo}
	\begin{rema}\Label{rema:basichodge}
		\begin{enumerate}
			\item Theorem \ref{theo:Dolbeault} can be proved by applying \cite[Theorem 7.5]{Goertsches-Nozawa-Toeben} with a small modification. We will show Theorem \ref{theo:Dolbeault} by a totally different argument. 
			We will show that $H^*_B(M)$ is generated by transverse K\"ahler forms with respect to $F$. Remark that a transverse K\"ahler form is a closed positive $(1,1)$-form and basic with respect to $F$. 
			\item For toric Sasaki manifolds, a result similar to Theorem \ref{theo:Dolbeault} has been shown, see \cite[Corollary 8.4]{Goertsches-Nozawa-Toeben}. 
		\end{enumerate}
	\end{rema}
	
	We use the same notation as Section \ref{sec:DJ}. Let $(M_0,G_0,y_0) \in \mathcal{C}_1$ and $(\Delta_0,\h_0,G_0) \in \mathcal{C}_2$ be as in Section \ref{sec:DJ}. We assume that $M_0$ admits a transverse K\"ahler form with respect to the canonical foliation $F_0$ of $M_0$. Let $(\widetilde{V}_0, \widetilde{\Gamma}_0, \widetilde{\Delta}_0, \widetilde{\lambda}_0) = \widetilde{F}_1(M_0,G_0,y_0) \in \widetilde{\mathcal{C}}_2$. As before, we denote by $p_0 \co \g_0^\C \to \g_0$ the projection and by $q_0 \co \g_0 \to \g_0/p_0(\h_0)$ the quotient map.
$\Delta_0$ has $m$ $1$-cones $\rho_1 = \R_{\geq 0}e_1,\dots, \rho_m = \R_{\geq 0}e_m$, where $e_1,\dots, e_m, e_{m+1},\dots, e_N$ denote the standard basis vectors of $\R^N = \g_0$. Thus $\widetilde{\Delta}_0 = q_0(\Delta_0)$ has $m$ $1$-cones $\widetilde{\rho}_1 = q_0(\rho_1),\dots, \widetilde{\rho}_m = q_0(\rho_m)$. Then we have $\widetilde{\lambda_0}(\widetilde{\rho_i}) = q_0(e_i)$. For short, we put $\mu_i := q_0(e_i)$ for $i=1,\dots, N$. For $b \in \R$, we define 
	\begin{equation*}
	H_{i,b} := \{ \widetilde{x} \in \widetilde{V}_0^* \mid \langle \widetilde{x}, \mu_i\rangle \geq b\}.
	\end{equation*}
	If $\mu_i \neq 0$, then $H_{i,b}$ is a half space whose inner normal vector is $\mu_i$. If $\mu_i=0$, then 
	\begin{equation*}
		H_{i,b} =
		\begin{cases}
			\widetilde{V}_0^* & \text{if $b \leq 0$,}\\
			\emptyset & \text{if $b >0$}.
		\end{cases}
	\end{equation*}

	It follows from Theorem \ref{theo:transverse} that $\widetilde{\Delta}_0$ is polytopal. Let $P \subset \widetilde{V}_0 = \g_0/p_0(\h_0)$ be an inner normal polytope of $\Delta_0$. Then there uniquely exist $b_1,\dots, b_m \in \R$ such that $P = \bigcap_{i=1}^m H_{i,b_i}$. 
	
	Let $\epsilon >0$. We take $b_{m+1},\dots, b_N \in \R$ so that $P \subset H_{j,b_j+\epsilon} \subset H_{j,b_j}$ for $j=m+1,\dots, N$. Then we have an embedding $\Psi \co P \to \g_0^* =  (\R^N)^*$ given by 
	\begin{equation}\Label{eq:I}
		\Psi (\widetilde{x}) = \sum_{i=1}^N (\langle \widetilde{x}, \mu_i\rangle -b_i)x_i\quad \text{for $\widetilde{x} \in P$}, 
	\end{equation}
	where $x_1,\dots, x_N$ denote the dual basis vectors of $e_1,\dots, e_N$. In other words, $\Psi = q_0^* - \sum_{i=1}^Nb_ix_i$. By definition of $b_1,\dots, b_N$, the coefficients $\langle \widetilde{x}, \mu_i\rangle -b_i$ are nonnegative. Let $\h'_0 := p_0(\h_0)$, $H'_0 := \exp_{G_0}(\h'_0)$ and $r_0 \co \h'_0 \to \g_0$ the inclusion. Then we have that $r^*_0 \circ \Psi \co \widetilde{V}^* \to {\h'_0}^*$ is constant and $r^*_0 \circ \Psi(\widetilde{x}) = -\sum_{i=1}^Nb_i r^*_0(x_i)$. 	
	Let $\omega_{\text{st}}$ be the K\"ahler form on $\C^N$ given by
	\begin{equation*}
		\omega_\text{st} = -\frac{\sqrt{-1}}{2\pi}\sum_{i=1}^N dz_i \wedge d\overline{z}_i. 
	\end{equation*}
	Then, a moment map $\Phi_{\text{st}} \co \C^N \to \g_0^* = (\R^N)^*$ is given by 
	\begin{equation*}
		\Phi_{\text{st}}(z) =  \sum_{i=1}^N |z_i|^2 x_i \quad \text{for $z=(z_1,\dots, z_N) \in \C^N$}.
	\end{equation*}
	A moment map $\Phi_{st}^{H'_0}\co \C^N \to {\h' _0}^*$ with respect to the action restricted to $H'_0$ is given by composing $\Phi_{\text{st}}$ with the surjective map $r_0^* \co g^*_0 \to \h'^*_0$ induced by the inclusion $r_0 \co \h'_0 \to \g_0$. 
	
	For short, we denote 
	\begin{equation*}
		\mathcal{Z}_{b_1,\dots, b_N} := (\Phi^{H'_0}_{\text{st}})^{-1}\left(r^*_0\left( -\sum_{i=1}^N b_ix_i\right)\right). 
	\end{equation*}
	\begin{lemm}\Label{lemm:regular}
		The followings hold.
		\begin{enumerate}
			\item $\Phi_{\text{st}}(\mathcal{Z}_{b_1,\dots, b_N}) = \Psi(P)$. 
			\item $\mathcal{Z}_{b_1,\dots, b_N}$  is compact. 
			\item The element $r_0^*(-\sum_{i=1}^N b_ix_i)$ is a regular value of $\Phi^{H'_0}_{\text{st}}$.
		\end{enumerate}
		In particular, $\mathcal{Z}_{b_1,\dots, b_N}$ is a compact smooth manifold equipped with an action of $G_0$. 
	\end{lemm}
	\begin{proof}
		Let $(z_1,\dots, z_N) \in \C^N$ be a point in $\mathcal{Z}_{b_1,\dots, b_N}$.  
		Let $e_1',\dots, e_k'$ be basis vectors of $\h'_0$. Since 
		\begin{equation*}
			\langle \Phi_{\text{st}}^{H'}(z_1,\dots, z_N), e_j'\rangle = \langle \Phi_{\text{st}}(z_1,\dots, z_N), r_0(e_j')\rangle , 
		\end{equation*}
		 we have 
		\begin{equation}\Label{eq:zi}
			\sum_{i=1}^N (|z_i|^2+b_i)\langle x_i,e_j'\rangle =0
		\end{equation}
		for all $j =1,\dots, k$. It follows from \eqref{eq:zi} that the element $\sum_{i=1}^N(|z_i|^2+b_i)x_i \in \ker r_0^*$. Thus there uniquely exists $\widetilde{x} \in \widetilde{V}_0^*$ such that 
		\begin{equation*}
			q^*_0(\widetilde{x}) = \sum_{i=1}^N(|z_i|^2+b_i)x_i. 
		\end{equation*}
		By pairing with $e_i$, we have 
		\begin{equation}\Label{eq:xmu}
			\langle \widetilde{x}, \mu_i\rangle = |z_i|^2 + b_i
		\end{equation} for all $i=1,\dots, N$. Thus we have $\langle \widetilde{x}, \mu_i\rangle \geq b_i$ for all $i$. It turns out that $\widetilde{x} \in P$. 
		Conversely, if $\widetilde{x} \in P$, then we have that 
		\begin{equation*}
			(\sqrt{\langle \widetilde{x},\mu_1\rangle -b_1},\dots, \sqrt{\langle \widetilde{x},\mu_N\rangle -b_N}) \in \mathcal{Z}_{b_1,\dots, b_N}. 
		\end{equation*}
		Thus we have 
		$\Phi_{\text{st}}(\mathcal{Z}_{b_1,\dots, b_N}) = \Psi(P)$, where $\Psi \co P \to (\R^N)^*$ is the embedding given by \eqref{eq:I}. Thus $\mathcal{Z}_{b_1,\dots, b_N}$ is a bounded subset of $\C^N$. Since $\mathcal{Z}_{b_1,\dots, b_N}$ is closed, we have that $\mathcal{Z}_{b_1,\dots, b_N}$ is compact, proving Part (1) and Part (2). 
		
		Now we show Part (3). Let $\xi_i$ and $\eta_i$ be the real and imaginary part of $z_i$, respectively. By \eqref{eq:zi}, it is enough to show that the matrix 
		\begin{equation}\Label{eq:Jacobian}
			\begin{pmatrix}
				2\xi_1\langle x_1,e_1'\rangle & 2\eta_1\langle x_1,e_1'\rangle & \cdots & 2\xi_N\langle x_N,e_1'\rangle & 2\eta_N\langle x_N,e_1'\rangle \\
				\vdots & \vdots & & \vdots & \vdots \\
				2\xi_1\langle x_1,e_k'\rangle & 2\eta_1\langle x_1,e_k'\rangle & \cdots & 2\xi_N\langle x_N,e_k'\rangle & 2\eta_N\langle x_N,e_k'\rangle
			\end{pmatrix}
		\end{equation}
		has rank $k$ for some basis vectors $e_1',\dots, e_k'$.
		Since $P$ is a normal polytope of the simplicial fan $\widetilde{\Delta}_0$, we have that $P$ is a simple polytope of dimension $N-k$. We denote by $F_i$ the facet of $P$ given by
		\begin{equation*}
			F_i := P \cap \{ \widetilde{x}\in \widetilde{V}_0^* \mid \langle \widetilde{x},\mu_i\rangle= b_i \}
		\end{equation*}
		for $i=1,\dots, m$. Define
		\begin{equation*}
			I_{\widetilde{x}} := \{ i \in \{1,\dots, m\} \mid \widetilde{x} \in F_i \}.
		\end{equation*}
		Since $P$ is simple, we have $|I_{\widetilde{x}}| \leq N-k$. By renumbering, we may assume that $I_{\widetilde{x}} = \{1,\dots, N-k'\}$ for some $k' \geq k$ and $F_1 \cap \dots \cap F_{N-k'} \cap F_{N-k'+1}  \dots \cap F_{N-k} \neq \emptyset$ without loss of generality. 
		 Then we have $|z_{j}|^2> 0$ for $j=N-k'+1,\dots, N$ by \eqref{eq:xmu}. In particular, either $\xi_{j}$ or $\eta_{j}$ are nonzero for $j=N-k+1,\dots, N$. Thus, if the matrix 
		\begin{equation}\Label{eq:matrix}
			\begin{pmatrix}
				\langle x_{N-k+1},e_1'\rangle & \hdots & \langle x_{N},e_1'\rangle \\
				\vdots & \ddots & \vdots \\
				\langle x_{N-k+1},e_k'\rangle & \hdots & \langle x_{N},e_k'\rangle
			\end{pmatrix}
		\end{equation}
		is full-rank for some basis vectors $e_1',\dots, e_k'$, then the matrix \eqref{eq:Jacobian} has rank $k$. Since $F_1 \cap \dots \cap F_{N-k'} \cap F_{N-k'+1}  \dots \cap F_{N-k} \neq \emptyset$, we have that $\mu_1=q_0(e_1),\dots, \mu_{N-k}=q_0(e_{N-k})$ are basis vectors of $\widetilde{V}_0$. Let $\widetilde{x}_1,\dots, \widetilde{x}_{N-k} \in \widetilde{V}_0^*$ be the dual basis vectors of $\mu_1,\dots, \mu_{N-k}$. Then 
		\begin{equation}\Label{eq:q0wx}
			q_0^*(\widetilde{x}_i) = \sum_{j=1}^N \langle q_0^*(\widetilde{x}_i), e_j\rangle x_j = x_i + \sum_{j=N-k+1}^N \langle q_0^*(\widetilde{x}_i), e_j\rangle x_j
		\end{equation}
		for $i=1,\dots, N-k$. We claim that $r_0^*(x_{N-k+1}),\dots, r_0^*(x_N)$ are linearly independent. Suppose that 
		\begin{equation*}
			\sum_{j=N-k+1}^N c_jr^*_0(x_j) = 0, \quad c_j\in\R.
		\end{equation*}
		Then $\sum_{j=N-k+1}^N c_jx_j \in \ker r^*_0$. Therefore $\sum_{j=N-k+1}^N c_jx_j$ is a linear combination of $q_0^*(\widetilde{x}_1),\dots, q_0^*(\widetilde{x}_{N-k})$. By \eqref{eq:q0wx}, we have $\sum_{j=N-k+1}^N c_jx_j = 0$. Thus we have $c_j = 0$ for all $j=N-k+1,\dots, N$. Thus $r_0^*(x_{N-k+1}),\dots, r_0^*(x_N)$ are linearly independent. By dimensional reason, we have that $r_0^*(x_{N-k+1}),\dots, r_0^*(x_N)$ form a basis of ${\h'_0}^*$. If $e_1', \dots, e_k'$ are the dual basis vectors of $r_0^*(x_{N-k+1}),\dots, r_0^*(x_N)$, then the matrix \eqref{eq:matrix} is the identity matrix, proving Part (3). The lemma is proved. 
	\end{proof}
	\begin{lemm}\Label{lemm:transverseZ}
		Let $z \in \mathcal{Z}_{b_1,\dots, b_N}$. The followings hold.
		\begin{enumerate}
			\item The $H_0$-orbit $H_0 \cdot z \subset \C^N$ through $z$ intersect with $\mathcal{Z}_{b_1,\dots, b_N}$ at $z$ transversely. 
			\item $\mathcal{Z}_{b_1,\dots, b_N} \cap H_0 \cdot z = \{z\}$. 
		\end{enumerate}
	\end{lemm}
	\begin{proof}
		For $v \in \g_0$, we denote by $h_v \co \C^N \to \R$ the function given by $h_v(z) = \langle \Phi_{\text{st}}(z), v\rangle$ for $z \in \C^N$. 
		Let $z \in \mathcal{Z}_{b_1,\dots, b_N}$. By definition of $\mathcal{Z}_{b_1,\dots, b_N}$, we have 
		\begin{equation*}
			\begin{split}
				T_z\mathcal{Z}_{b_1,\dots, b_N} &= \{ Y \in T_z\C^N \mid dh_v(Y) = 0 \text{ for all $v \in \h_0'$}\} \\
				&= \{ Y \in T_z\C^N \mid \omega_{\text{st}}(X_v, Y)=0 \text{ for all $v \in \h_0'$}\}. 
			\end{split}
		\end{equation*}
		Let $u+\sqrt{-1}v \in \h_0$, where $u, v \in \h'_0$. Then, we have 
		\begin{equation*}
			\omega_{\text{st}}(X_v, X_{u+\sqrt{-1}v}) = \omega_{\text{st}}(X_v, JX_v) \geq 0
		\end{equation*}
		and the equality holds if and only if $v=0$. Since the restriction of $p_0 \co \g_0^\C \to \g_0$ to $\h_0$ is injective, we have $\omega_{\text{st}}(X_v, X_{u+\sqrt{-1}v}) >0$ unless $u+\sqrt{-1}v = 0$. Thus we have $T_z\mathcal{Z}_{b_1,\dots, b_N} \cap T_z(H_0\cdot z) = \{0\}$, proving Part (1). 
		
		Suppose that $g \cdot z \in \mathcal{Z}_{b_1,\dots, b_N}$ for some $g \in H_0$. Then there exists $u+\sqrt{-1}v \in \h_0$ such that $\exp_{G_0^\C} (u+\sqrt{-1}v) = g$. Let $\gamma_t$ denote the partial flow of $X_{u+\sqrt{-1}v} = X_u+JX_v$. Since 
		\begin{equation*}
			\begin{split}
				L_{X_u+JX_v}h_v &= i_{X_u+JX_v}dh_v\\
				&= i_{X_u+JX_v}(-i_{X_v}\omega_{\text{st}})\\
				&= \omega_{\text{st}}(-X_v, X_u+JX_v)\\
				&= -\omega_{\text{st}}(X_v,JX_v) \leq 0,
			\end{split}
		\end{equation*}
		we have that the function $\R \to \R$ given by $t \mapsto h_v(\gamma_t(v))$ is strictly monotone decreasing unless $v=0$. Thus we have 
		\begin{equation*}
			\langle \Phi_{\text{st}}^{H_0'}(z), v\rangle = h_v(z) \geq h_v(\gamma_1(z)) = h_v(g\cdot z) = \langle \Phi_{\text{st}}^{H_0'}(g\cdot z), v\rangle
		\end{equation*}
		and the equality holds if and only if $v=0$. Since $g\cdot z \in \mathcal{Z}_{b_1,\dots, b_N}$, we have $\langle \Phi_{\text{st}}^{H_0'}(z), v\rangle = \langle \Phi_{\text{st}}^{H_0'}(g\cdot z), v\rangle$. Thus we have $v=0$. This together with the injectivity of the restriction of $p_0 \co \g_0^\C \to \g_0$ to $\h_0$ implies that $u+\sqrt{-1}v=0$ and hence $g$ is the unit of $H_0$. This shows that $\mathcal{Z}_{b_1,\dots, b_N} \cap H_0 \cdot z = \{z\}$, proving Part (2). The lemma is proved.
	\end{proof}
	\begin{lemm}\Label{lemm:localdiffeo}
		The followings hold.
		\begin{enumerate}
			\item $\mathcal{Z}_{b_1,\dots, b_N} \subset X(\Delta_0)$. 
			\item Let $\varphi \co \mathcal{Z}_{b_1,\dots, b_N} \to M_0 = X(\Delta_0)/H_0$ be the smooth map induced by the inclusion $\mathcal{Z}_{b_1,\dots, b_N} \hookrightarrow X(\Delta_0)$. Then $\varphi$ is injective and a local diffeomorphism.
			
		\end{enumerate}
	\end{lemm}
	\begin{proof}
		Let $z =(z_1,\dots, z_N) \in \mathcal{Z}_{b_1,\dots, b_N}$. By Lemma \ref{lemm:regular} (1), we have that there exists $\widetilde{x} \in P$ such that $\langle \widetilde{x},\mu_i\rangle = |z_i|^2 +b_i$ for $i=1,\dots, N$. Put
		\begin{equation*}
			I_z = \{ i \in \{1,\dots, N\} \mid z_i = 0\}. 
		\end{equation*}
		Then $i \in I_z$ if and only if $\langle \widetilde{x},\mu_i\rangle -b_i=0$. It turns out that $\widetilde{x} \in \bigcap_{i \in I_z} F_i$. Let $K$ be the underlying simplicial complex of $\Delta_0$. Since $P$ is a normal polytope of the simplicial fan $\widetilde{\Delta}_0$, we have that $I \in K$ if and only if $\bigcap_{i\in I}F_i \neq \emptyset$. Thus we have $I_z \in K$. Since 
		\begin{equation*}
			X(\Delta_0) = \C^N \setminus \bigcup_{I \notin K} L_I,
		\end{equation*}
		where
		\begin{equation*}
			L_I := \{ (z_1,\dots, z_n) \in \C^N \mid z_i=0 \text{ for all $i \in I$}\},
		\end{equation*}
		we have $z \notin \bigcup_{I \notin K}L_I$ because $I_z \in K$. This shows Part (1). 
		
		Let $z \in \mathcal{Z}_{b_1,\dots, b_N}$. We denote by $[z] \in M_0 = X(\Delta_0)/H_0$ the $H_0$-orbit through $z$. The differential $d\varphi_z \co T_z\mathcal{Z}_{b_1,\dots, b_n} \to T_{[z]}M_0$ is the composition of the inclusion $T_z\mathcal{Z}_{b_1,\dots, b_n} \to T_zX(\Delta_0)$ and the quotient map $T_zX(\Delta_0) \to T_zX(\Delta_0)/T_z(H_0\cdot z) = T_{[z]}M_0$. This together with Lemma \ref{lemm:transverseZ} (1) yields that $d\varphi_z$ is an isomorphism. By inverse function theorem, we have that $\varphi$ is a local diffeomorphism. It follows from Lemma \ref{lemm:transverseZ} (2) that $\varphi$ is injective, proving Part (2). The lemma is proved. 
	\end{proof}
	\begin{coro}
		Let $\varphi \co \mathcal{Z}_{b_1,\dots, b_N} \to M_0 = X(\Delta_0)/H_0$ be the smooth map induced by the inclusion $\mathcal{Z}_{b_1,\dots, b_N} \hookrightarrow X(\Delta_0)$. $\varphi$ is a $G_0$-equivariant diffeomorphism. 
	\end{coro}
	\begin{proof}
		Since $\mathcal{Z}_{b_1,\dots, b_N}$ is a $G_0$-invariant subset of $X(\Delta_0)$, we have that $\varphi$ is $G_0$-equivariant. By Lemma \ref{lemm:localdiffeo} (2), we have that $\varphi$ is an injective local diffeomorphism. Therefore $\varphi$ is a diffeomorphism onto the open subset $\varphi(\mathcal{Z}_{b_1,\dots, b_N}) \subset M_0$. By Lemma \ref{lemm:regular} (2), we have that $\varphi(\mathcal{Z}_{b_1,\dots, b_N})$ is closed. Since $M_0$ is connected, we have that an open and closed subset of $M_0$ is $\emptyset$ or $M_0$. Since $\varphi(\mathcal{Z}_{b_1,\dots, b_N})$ is open and closed in $M_0$, we have $\varphi(\mathcal{Z}_{b_1,\dots, b_N}) = M_0$. This shows that $\varphi$ is surjective and hence $\varphi \co \mathcal{Z}_{b_1,\dots, b_N} \to M_0$ is a $G_0$-equivariant diffeomorphism. The corollary is proved. 
	\end{proof}
	Now we are in a position to construct a transverse K\"ahler form on $M_0$ with respect to the canonical foliation $F_0$ via the diffeomorphism $\varphi \co \mathcal{Z}_{b_1,\dots, b_N} \to M_0$. We denote by $\omega_{\text{st}}|_{\mathcal{Z}_{b_1,\dots, b_N}}$ the restriction of $\omega_{\text{st}}$ to the submanifold $\mathcal{Z}_{b_1,\dots, b_N} \subset X(\Delta_0) \subset \C^N$. We denote by $\omega_{b_1,\dots, b_N}$ the pull-back $(\varphi^{-1})^*\omega_{\text{st}}|_{\mathcal{Z}_{b_1,\dots, b_N}}$ of $\omega_{\text{st}}|_{\mathcal{Z}_{b_1,\dots, b_N}}$ by $\varphi^{-1} \co M_0 \to \mathcal{Z}_{b_1,\dots, b_N}$. By definition, $\omega_{b_1,\dots, b_N}$ is a closed $2$-form on $M_0$. 
	\begin{lemm}\Label{lemm:omegab}
		$\omega_{b_1,\dots, b_N}$ is a transverse K\"ahler form on $M_0$ with respect to $F_0$. 
	\end{lemm}
	\begin{proof}
		Let $z \in \mathcal{Z}_{b_1,\dots, b_N}$. Since 
		\begin{equation*}
			T_z\mathcal{Z}_{b_1,\dots, b_N} = \{ Y \in T_zX(\Delta_0) \mid \omega_{\text{st}}(X_v, Y) = 0 \text{ for all $v \in \h_0'$}\}, 
		\end{equation*}
		the kernel of $\omega_{\text{st}}|_{\mathcal{Z}_{b_1,\dots, b_N}}$ at $z$ coincides with $T_z(H_0'\cdot z)$. Since $\varphi$ is a $G_0$-equivariant diffeomorphism, we have that $\omega_{b_1,\dots, b_N}$ is a transverse symplectic form on $M_0$ with respect to the canonical foliation $F_0$. 
		
		Let $\pi \co X(\Delta_0) \to M_0$ be the quotient map. Let $J_\text{st}$ and $J_0$ be the complex structure on $X(\Delta_0)$ and $M_0$, respectively. Since $\pi$ is holomorphic, we have $d\pi_z \circ J_\text{st} = J_0\circ d\pi_z$. Let $T_z(H_0' \cdot z)^\perp$ be the orthogonal complement of $T_z(H_0' \cdot z)$ in $T_z\mathcal{Z}_{b_1,\dots, b_N}$ with respect to the inner product $\omega_{\text{st}}(-, J_\text{st}-)$. Then we have that 
		\begin{equation*}
			T_z(H_0' \cdot z)^\perp = \{ Y \in T_zX(\Delta_0) \mid \omega_{\text{st}}(X_v,Y) = \omega_{\text{st}}(X_v,J_{\text{st}}Y) = 0\text{ for all $v \in \h_0'$} \}
		\end{equation*}
		and hence $T_z(H_0' \cdot z)^\perp$ is closed under $J_\text{st}$. Let $X^\perp,Y^\perp \in T_z(H_0' \cdot z)^\perp$ and $X,Y \in T_z(H_0' \cdot z)$. Since $d\varphi _z (T_z(H_0' \cdot z)) = T_{\varphi(z)}F_0$ and  $T_{\varphi(z)}F_0$ is closed under $J_0$, we have
		\begin{equation}\Label{eq:transsymp}
			\begin{split}
				&\omega_{b_1,\dots, b_N}(J_0(d\varphi_z(X^\perp+X)), J_0(d\varphi_z(Y^\perp+Y))) \\ =& \omega_{b_1,\dots, b_N}(J_0(d\varphi_z(X^\perp)),J_0(d\varphi_z(Y^\perp))).
			\end{split}
		\end{equation}
		 Since $d\varphi_z$ coincides with the restriction of $d\pi_z$ to $T_z\mathcal{Z}_{b_1,\dots, b_N}$, we have 
		 \begin{equation*}
		 	\begin{split}
			 	&\omega_{b_1,\dots, b_N}(J_0d\varphi_z(X^\perp),J_0d\varphi_z(Y^\perp))\\
				=& \omega_{b_1,\dots, b_N}(J_0d\pi_z(X^\perp),J_0d\pi_z(Y^\perp))\\
				=& \omega_{b_1,\dots, b_N}(d\pi_z(J_\text{st}X^\perp), d\pi_z(J_{\text{st}}Y^\perp))\\
				=& \omega_{b_1,\dots, b_N}(d\varphi_z(J_\text{st}X^\perp), d\varphi_z(J_{\text{st}}Y^\perp))\\
				=& \omega_{\text{st}}(J_\text{st}X^\perp, J_\text{st}Y^\perp)\\
				=& \omega_{\text{st}}(X^\perp, Y^\perp)\\
				=& \omega_{b_1,\dots, b_N}(d\varphi_z(X^\perp),d\varphi_z(Y^\perp)).
			\end{split}
		 \end{equation*}
		 This together with \eqref{eq:transsymp} yields that $\omega_{b_1,\dots, b_n}(J_0X_0,J_0Y_0) = \omega_{b_1,\dots, b_n}(X_0,Y_0)$ for any $X_0,Y_0 \in T_{\varphi(z)}M_0$. Namely, $\omega_{b_1,\dots, b_n}$ is of type $(1,1)$. 
		 
		 The positivity of $\omega_{b_1,\dots, b_n}$ can be shown by the same argument. More precisely, we have 
		 \begin{equation*}
		 	\begin{split}
			 	&\omega_{b_1,\dots, b_N} (d\varphi_z(X^\perp + X), J_0(d\varphi_z(X^\perp + X))) \\
				=& \omega_{b_1,\dots, b_N} (d\varphi_z(X^\perp), J_0(d\varphi_z(X^\perp))) \\
				=& \omega_{b_1,\dots, b_N} (d\varphi_z(X^\perp), J_0(d\pi_z(X^\perp)))\\
				=& \omega_{b_1,\dots, b_N} (d\varphi_z(X^\perp), d\pi_z(J_\text{st}X^\perp)) \\
				=& \omega_{b_1,\dots, b_N} (d\varphi_z(X^\perp), d\varphi_z(J_\text{st}X^\perp)) \\
				=& \omega_{\text{st}}(X^\perp, J_\text{st}X^\perp) \geq 0.
			\end{split}
		 \end{equation*}
		 Thus $\omega_{b_1,\dots, b_N} (X_0,J_0X_0) \geq 0$ for any $X_0 \in T_{\varphi(z)}M_0$. 
		 The lemma is proved.
	\end{proof}
	\begin{lemm}\Label{lemm:Phib}
		The composition $\Phi_{b_1,\dots, b_N} := \Phi_{\text{st}}\circ \varphi^{-1} \co M_0 \to \g_0^*$ is a moment map of $M_0$ with respect to the transverse symplectic form $\omega_{b_1,\dots,b_N}$. The image $\Phi_{b_1,\dots, b_N}(M_0)$ coincides with $\Psi (P)$. 
	\end{lemm}
	\begin{proof}
		Let $v \in \g_0$. We show that $dh_v = -i_{X_v}\omega_{b_1,\dots, b_N}$, where $h_v \co M_0 \to \R$ is given by $h_v(x) = \langle \Phi_{b_1,\dots, b_N}(x), v\rangle$. Since $\varphi^*h_v = \langle \varphi^*\Phi_{\text{st}}, v\rangle$, by taking differential we have $d\varphi^*h_v = \varphi^*(-i_{X_v}\omega_{\text{st}}|_{\mathcal{Z}_{b_1,\dots, b_N}})$. Since $d\varphi^*h_v = \varphi^*(dh_v)$ and $\varphi$ is a $G_0$-equivariant diffeomorphism, we have $dh_v=-i_{X_v}\omega_{b_1,\dots, b_N}$. Therefore $\Phi_{b_1,\dots ,b_N}$ is a moment map. 
		
		By definition, we have $\Phi_{b_1,\dots, b_N}(M_0) = \Phi_{\text{st}}(\mathcal{Z}_{b_1,\dots, b_N})$. This together with Lemma \ref{lemm:regular} (1) yields that $\Phi_{b_1,\dots, b_N}(M_0) = \Psi (P)$, proving the lemma. 
	\end{proof}
	By Lemma \ref{lemm:Phib} and $\Psi(P) = q_0^*(P) - \sum_{i=1}^Nb_ix_i$, we have an induced moment map $\widetilde{\Phi}_{b_1,\dots,b_N} \co M_0 \to \widetilde{V}_0^*$ that satisfies 
	\begin{equation*}
		q_0^*\circ\widetilde{\Phi}_{b_1,\dots,b_N} - \sum_{i=1}^Nb_ix_i= \Phi_{b_1,\dots, b_N}. 
	\end{equation*}
	The image $\widetilde{\Phi}_{b_1,\dots,b_N}(M_0)$ coincides with $P$ and $q_0^*\circ \widetilde{\Phi}_{b_1,\dots,b_N} \co M_0 \to \g_0^*$ is also a moment map. 
	We say that $(b_1,\dots, b_N) \in \R^N$ is \emph{admissible} if $(b_1,\dots, b_N)$ satisfies the followings.
	\begin{enumerate}
		\item $P _{b_1,\dots, b_m} := \bigcap_{i=1}^m H_{i,b_i}$ is a normal polytope of $\widetilde{\Delta}_0$. In particular, 
		\begin{equation*}
			P \cap \{ \widetilde{x} \in V^*_0 \mid \langle \widetilde{x}, \mu_i\rangle =b_i\}
		\end{equation*}
		is a facet of $P$ for $i=1,\dots, m$. 
		\item There exists $\epsilon >0$ such that $P _{b_1,\dots, b_N} \subset H_{i,b_i+\epsilon} \subset H_{i,b_i}$ for $i =m+1,\dots, N$. 
	\end{enumerate}
	The set of all admissible elements in $\R^N$ is nonempty and open in $\R^N$. 
	For each admissible element $(b_1,\dots, b_N)$, we have a $d_{G_0}$-closed class $\omega^{G_0}_{b_1,\dots, b_N} := 1\otimes \omega_{b_1,\dots, b_N}  - q^*_0\circ \Phi_{b_1,\dots, b_N} \in S^0(\g_0^*) \otimes \Omega^2_B(M)^{G_0} \oplus S^2(\g_0^*) \otimes \Omega^0_B(M)^{G_0}$. 
	
	We will see the relation between the basic cohomology class $[\omega_{b_1,\dots, b_N}] \in H^2_B(M_0)$ and the equivariant cohomology class $[\omega^{G_0}_{b_1,\dots, b_N}] \in H^2_{G_0}(M_0)$ via the localization map. We need the following lemma. 
	\begin{lemm}\Label{lemm:vertexcorrespondence}
		Let $Z_1, \dots, Z_l$ be minimal orbits of $M_0$. Let $(b_1,\dots, b_N), (b_1',\dots, b_N') \in \R^N$ be admissible. For $i=1,\dots, m$, we define facets 
		\begin{equation*}
			\begin{split}
				F_i &:= P_{b_1,\dots, b_m} \cap \{\widetilde{x} \in \widetilde{V}^*_0 \mid \langle \widetilde{x}, \mu_i\rangle = b_i\}, \\
				F_i' &:= P_{b_1',\dots, b_m'} \cap \{\widetilde{x}' \in \widetilde{V}^*_0 \mid \langle \widetilde{x}', \mu_i\rangle = b_i'\}
			\end{split}
		\end{equation*}
		of $P_{b_1,\dots, b_m}$, $P_{b_1',\dots, b_m'}$, respectively. Then, $\widetilde{\Phi}_{b_1,\dots, b_m}(Z_j) \in F_i$ if and only if $\widetilde{\Phi}_{b_1',\dots, b_m'}(Z_j) \in F_i'$. 
	\end{lemm}
	\begin{proof}
		Let $\varphi \co \mathcal{Z}_{b_1,\dots, b_N} \to M_0$ and $\varphi' \co \mathcal{Z}_{b_1',\dots, b_N'} \to M_0$ be the $G_0$-equivariant diffeomorphisms induced by the inclusions. Let $z = (z_1,\dots, z_N) \in \varphi^{-1}(Z_j)$. Put
		\begin{equation*}
			I_{z} := \{ i \in \{1,\dots, m\} \mid z_i = 0\}. 
		\end{equation*}
		Then the isotropy subgroup of $G_0$ at $z$ is 
		\begin{equation*}
			\{ (g_1,\dots, g_N) \in (S^1)^N = G_0 \mid g_i =1 \text{ for $i\notin I_z$}\}.
		\end{equation*}
		Since $(\varphi')^{-1} \circ \varphi \co \mathcal{Z}_{b_1,\dots,b_N} \to \mathcal{Z}_{b_1',\dots,b_N'}$ is a $G_0$-equivariant diffeomorphism, we have $I_z = I_{z'}$, where $z' =(\varphi')^{-1}\circ \varphi(z)$. 
		Let $\Psi \co P_{b_1,\dots, b_m} \to \g_0^*$ be the embedding given by 
		\begin{equation*}
			\Psi(\widetilde{x}) = \sum_{i=1}^N (\langle \widetilde{x},\mu_i\rangle -b_i)x_i \quad \text{for $\widetilde{x} \in P_{b_1,\dots, b_N}$}.
		\end{equation*}
		By Lemma \ref{lemm:Phib}, there exists $\widetilde{x} \in P_{b_1,\dots, b_N}$ such that 
		\begin{equation*}
			\Psi (\widetilde{x}) = \Phi_{b_1,\dots, b_N}(Z_j) = \Phi_{\text{st}}\circ \varphi^{-1}(Z_j).
		\end{equation*}
		Thus we have that $\langle \widetilde{x},\mu_i\rangle =b_i$ if and only if $i\in I_z$. Since $q^*_0\circ \widetilde{\Phi}_{b_1,\dots, b_N}(Z_j) - \sum_{i=1}^mb_ix_i = \Phi_{b_1,\dots, b_N}(Z_j) = \Psi (\widetilde{x})$, we have $q^*_0\circ \widetilde{\Phi}_{b_1,\dots, b_N}(Z_j) = \sum_{i=1}^N\langle \widetilde{x},\mu_i\rangle x_i$. It turns out that $\widetilde{\Phi}_{b_1,\dots, b_N}(Z_j) = \widetilde{x}$. Since $\langle \widetilde{x}, \mu_i\rangle =b_i$ if and only if $i\in I_z$, we have that $\widetilde{\Phi}_{b_1,\dots, b_N}(Z_j) \in F_i$ if and only if $i \in I_z$. By the same argument, we also have that $\widetilde{\Phi}_{b_1',\dots, b_N'}(Z_j) \in F_i'$ if and only if $i \in I_{z'}$. Since $I_z = I_z'$, we have that  $\widetilde{\Phi}_{b_1,\dots, b_m}(Z_j) \in F_i$ if and only if $\widetilde{\Phi}_{b_1',\dots, b_m'}(Z_j) \in F_i'$, as required. 
	\end{proof}
	
	\begin{lemm}\Label{lemm:parallel}
		Let $(b_1,\dots, b_N), (b_1',\dots, b_N') \in \R^N$ be admissible. 
		Then, $[\omega_{b_1,\dots, b_N}] = [\omega_{b_1',\dots, b_N'}]$ if and only if there exists $\widetilde{y} \in \widetilde{V}^*_0$ such that $P_{b_1,\dots, b_m} = P_{b_1',\dots, b_m'} + \widetilde{y}$. 
	\end{lemm}
	\begin{proof}
		Suppose that $[\omega_{b_1,\dots, b_N}] = [\omega_{b_1',\dots, b_N'}]$. Then we have $\forB([\omega_{b_1,\dots, b_N}^{G_0}]) = \forB([\omega_{b_1',\dots, b_N'}^{G_0}])$. By Theorem \ref{theo:basicforgetful}, we have that there exists an element $\widetilde{y} \in \widetilde{V}_0^*$ such that $[\omega_{b_1,\dots, b_N}^{G_0}] = [\omega_{b_1',\dots, b_N'}^{G_0}] + [q_0^*(\widetilde{y})]$. 
		
		Let $Z_1,\dots, Z_l$ be minimal orbits. By Lemma \ref{lemm:localizationminimal}, we have 
		\begin{equation*}
			[q^*_0\circ \widetilde{\Phi}_{b_1,\dots, b_N}(Z_j)] = [q^*_0\circ \widetilde{\Phi}_{b_1',\dots, b_N'}(Z_j)] + [q_0^*(\widetilde{y})].
		\end{equation*}
		By Lemma \ref{lemm:local} (2), we have that $H^2_{G_0}(Z_j)$ is isomorphic to $q_0^*(\widetilde{V}_0^*)$. Thus we have 
		\begin{equation*}
			\widetilde{\Phi}_{b_1,\dots, b_N}(Z_j) = \widetilde{\Phi}_{b_1',\dots, b_N'}(Z_j) + \widetilde{y}
		\end{equation*}
		for all $j$. 	By Theorem \ref{theo:transverse} (2), we have $P_{b_1,\dots, b_m} = P_{b_1',\dots, b_m'} + \widetilde{y}$. 
		
		Now we show the converse. Suppose that there exists $\widetilde{y} \in \widetilde{V}^*_0$ such that $P_{b_1,\dots, b_m} = P_{b_1',\dots, b_m'} + \widetilde{y}$. 
		By Proposition \ref{prop:vertex} and Lemma \ref{lemm:vertexcorrespondence}, we have $\widetilde{\Phi}_{b_1,\dots, b_N}(Z_j) = \widetilde{\Phi}_{b_1',\dots, b_N'}(Z_j) + \widetilde{y}$ for all $j$. By Corollary \ref{coro:localization} and Lemma \ref{lemm:localizationminimal}, we have $[\omega^{G_0}_{b_1,\dots, b_N}] = [\omega^{G_0}_{b_1',\dots, b_N'}] +[q^*_0(\widetilde{y})]$. Since $\forB([q_0^*(\widetilde{y})]) = 0$, $\forB([\omega^{G_0}_{b_1,\dots, b_N}]) = [\omega_{b_1,\dots, b_N}]$ and $\forB([\omega^{G_0}_{b_1',\dots, b_N'}]) = [\omega_{b_1',\dots, b_N'}]$, we have $[\omega_{b_1,\dots, b_N}] = [\omega_{b_1',\dots, b_N'}]$, proving the lemma. 
	\end{proof}
	\begin{lemm}\Label{lemm:paralleltranslation}
		Let $(b_1,\dots, b_m), (b_1',\dots, b_m') \in \R^m$. Assume that $P_{b_1,\dots, b_m}$ and $P_{b_1',\dots, b_m'}$ are inner normal polytopes of $\widetilde{\Delta}_0$. Let $\widetilde{y} \in \widetilde{V}^*_0$. Then, $P_{b_1,\dots, b_m} = P_{b_1',\dots, b_m'} + \widetilde{y}$ if and only if $\langle \widetilde{y}, \mu_i\rangle = b_i-b_i'$ for $i=1,\dots, m$.
	\end{lemm}
	\begin{proof}
		Let $\widetilde{x} \in V^*_0$. Put $\widetilde{x}' := \widetilde{x}-\widetilde{y}$. Then we have 
		\begin{equation*}
			\begin{split}
				\langle \widetilde{x},\mu_i\rangle \geq b_i &\Longleftrightarrow \langle \widetilde{x}' +\widetilde{y}, \mu_i\rangle \geq b_i\\
				&\Longleftrightarrow \langle \widetilde{x}',\mu_i\rangle \geq b_i-\langle \widetilde{y},\mu_i\rangle
			\end{split}
		\end{equation*}
		for $i=1,\dots , m$. Thus, if  $\langle \widetilde{y}, \mu_i\rangle = b_i-b_i'$ for all $i=1,\dots, m$, then the condition that $\widetilde{x} \in P_{b_1,\dots, b_m}$ is equivalent to the condition that $\widetilde{x}' \in P_{b_1',\dots, b_m'}$. 
		
		Suppose that $P_{b_1,\dots, b_m} = P_{b_1',\dots, b_m'} + \widetilde{y}$. As before, for $i=1,\dots, m$, we define
		\begin{equation*}
			\begin{split}
				F_i &:= P_{b_1,\dots, b_m} \cap \{ \widetilde{x} \in V^*_0\mid \langle \widetilde{x},\mu_i\rangle =b_i\},\\
				F_i' &:=P_{b_1',\dots, b_m'} \cap \{ \widetilde{x}' \in V^*_0\mid \langle \widetilde{x}',\mu_i\rangle =b_i'\}.
			\end{split}
		\end{equation*}
		Since $P_{b_1,\dots, b_m}$ and $P_{b_1',\dots, b_m'}$ are inner normal polytopes of $\widetilde{\Delta}_0$, we have that $F_i$ and $F_i'$ both are nonempty facets. Since $P_{b_1,\dots, b_m} = P_{b_1',\dots, b_m'} + \widetilde{y}$, we have $F_i = F_i' + \widetilde{y}$. Let $\widetilde{x} \in F_i$. Then we have $\widetilde{x}' := \widetilde{x}-\widetilde{y} \in F_i'$. On the other hand,
		\begin{equation*}
			\begin{split}
				\langle \widetilde{x},\mu_i\rangle = b_i &\Longleftrightarrow \langle \widetilde{x}' +\widetilde{y}, \mu_i\rangle = b_i\\
				&\Longleftrightarrow \langle \widetilde{x}',\mu_i\rangle = b_i-\langle \widetilde{y},\mu_i\rangle.
			\end{split}
		\end{equation*}
		This together with the fact that $\widetilde{x}' := \widetilde{x}-\widetilde{y} \in F_i'$ yields that $b_i - \langle \widetilde{y},\mu_i\rangle =b_i'$. Namely, we have that if $P_{b_1,\dots, b_m} = P_{b_1',\dots, b_m'} + \widetilde{y}$, then $\langle \widetilde{y}, \mu_i\rangle = b_i-b_i'$ for $i=1,\dots, m$. The lemma is proved.
	\end{proof}
	\begin{lemm}\Label{lemm:monoid}
		The followings hold. 
		\begin{enumerate}
			\item Let $(b_1,\dots, b_N), (b_1',\dots, b_N') \in \R^N$ be admissible. Then $(b_1+b_1',\dots, b_N+b_N')$ is admissible, $P_{b_1,\dots, b_m} + P_{b_1',\dots, b_m'} = P_{b_1+b_1',\dots, b_N+b_N'}$ and $[\omega_{b_1,\dots, b_N}] + [\omega_{b_1',\dots, b_N'}] = [\omega_{b_1+b_1',\dots, b_N+b_N'}]$. 
			\item Let $(b_1,\dots, b_N) \in \R^N$ be admissible and $r$ a positive real number. Then $(rb_1,\dots, rb_N)$ is admissible, $rP_{b_1,\dots, b_N} = P_{rb_1,\dots, rb_N}$ and $r[\omega_{b_1,\dots, b_N}] = [\omega_{rb_1,\dots, rb_N}]$. 
		\end{enumerate}
	\end{lemm}
	\begin{proof}
		For (1), let $(b_1,\dots, b_N), (b_1',\dots, b_N') \in \R^N$ be admissible. 
		Let $\widetilde{x} \in P_{b_1,\dots, b_m}$ and $\widetilde{x}' \in P_{b_1',\dots, b_m'}$. Since $P_{b_1,\dots, b_m} = \bigcap_{i=1}^m H_{i,b_i}$ and $P_{b_1',\dots, b_m'}= \bigcap_{i=1}^m H_{i,b_i'}$, we have $\langle \widetilde{x},\mu_i\rangle \leq b_i$ and $\langle \widetilde{x}',\mu_i\rangle \leq b_i'$ for $i=1,\dots, m$. Therefore we have $\langle \widetilde{x}+\widetilde{x}',\mu_i\rangle \leq b_i+b_i'$. This shows that $P_{b_1,\dots, b_m} + P_{b_1',\dots, b_m'} \subset P_{b_1+b_1',\dots, b_m+b_m'}$. 
		For $\mu \in \widetilde{V}_0$, we denote by $F_\mu$, $F_\mu'$, $F_\mu''$ the maximal faces of $P_{b_1,\dots, b_m}, P_{b_1',\dots, b_m'}, P_{b_1,\dots, b_m}+ P_{b_1',\dots, b_m'}$ whose inner normal vector is $\mu$. Let $\widetilde{x}''\in F_\mu''$. Then there exist $\widetilde{x} \in P_{b_1,\dots, b_m}$ and $\widetilde{x}' \in P_{b_1',\dots, b_m'}$ such that $\widetilde{x}'' = \widetilde{x} + \widetilde{x}'$. By definition of $F_\mu''$, we have $\langle \widetilde{y}'', \mu\rangle \leq \langle \widetilde{x}'',\mu\rangle$ for any $\widetilde{y}'' \in P_{b_1,\dots, b_m} + P_{b_1',\dots, b_m'}$. Suppose that $\widetilde{x} \notin F_\mu$. Then for $\widetilde{y} \in F_\mu$ we have $\langle \widetilde{x},\mu\rangle <\langle \widetilde{y},\mu\rangle$. Thus we have $\langle \widetilde{x}'',\mu\rangle < \langle \widetilde{y}+\widetilde{x}',\mu\rangle$. This contradicts to that $\langle \widetilde{y}'', \mu\rangle \leq \langle \widetilde{x}'',\mu\rangle$ for any $\widetilde{y}'' \in P_{b_1,\dots, b_m} + P_{b_1',\dots, b_m'}$. Therefore we have $\widetilde{x} \in F_\mu$. Using the same argument, we also have that $\widetilde{x}' \in F_\mu'$. As a conclusion, we have $F_\mu'' = F_\mu+F_\mu'$. Since $P_{b_1,\dots, b_m}$ and $P_{b_1',\dots, b_m'}$ are inner normal polytopes of $\widetilde{\Delta}_0$, $F_\mu$ and $F_\mu'$ have the same inner normal cone. Thus we have that $F_\mu$, $F_\mu'$ and $F_\mu''$ have the same inner normal cone. Suppose that $F_\mu''$ is a facet. Then $F_\mu$ is also a facet of $P_{b_1,\dots, b_N}$. Therefore $\mu$ coincides with one of $\mu_1,\dots, \mu_m$ up to positive scalar multiplication. Put $\mu = \mu_i$, $i=1,\dots, m$. Since $F_{\mu_i}'' = F_{\mu_i} + F_{\mu_i}'$, we have 
		\begin{equation*}
			F_{\mu_i}'' = (P_{b_1,\dots, b_m} +P_{b_1',\dots, b_m'}) \cap \{ \widetilde{x}'' \mid \langle \widetilde{x}'', \mu_i\rangle = b_i+b_i''\}. 
		\end{equation*}
		Since $F_{\mu_1}'',\dots, F_{\mu_m}''$ are all facets of $P_{b_1,\dots, b_m} +P_{b_1',\dots, b_m'}$, we have 
		\begin{equation*}
			P_{b_1,\dots, b_m} +P_{b_1',\dots, b_m'} = \bigcap_{i=1}^m H_{i,b_i+b_i'} = P_{b_1+b_1',\dots, b_m+b_m'}.
		\end{equation*}
		This together with the fact that $F_\mu''$ and $F_\mu$ have the same inner normal cone yields that $P_{b_1+b_1',\dots, b_m+b_m'}$ is an inner normal polytope of $\widetilde{\Delta}_0$. Since $(b_1,\dots, b_N)$ is admissible, there exists $\epsilon >0$ such that $\langle \widetilde{x},\mu_i\rangle \geq b_i+\epsilon$ for any $\widetilde{x} \in P_{b_1,\dots, b_m}$ and $i=m+1,\dots, N$. Since $(b_1',\dots, b_N')$ is admissible, there exists $\epsilon' >0$ such that $\langle \widetilde{x}',\mu_i\rangle \geq b_i'+\epsilon'$ for any $\widetilde{x}' \in P_{b_1',\dots, b_m'}$ and $i=m+1,\dots, N$. Thus we have 
		$\langle \widetilde{x}+\widetilde{x}',\mu_i\rangle \geq b_i+b_i' +\epsilon + \epsilon'$ for any $\widetilde{x}+\widetilde{x}' \in P_{b_1,\dots, b_m} +P_{b_1',\dots, b_m'}$ and $i=m+1,\dots, N$. This together with the fact that $P_{b_1,\dots, b_m} +P_{b_1',\dots, b_m'} = P_{b_1+b_1',\dots, b_m+b_m'}$ yields that $(b_1+b_1',\dots, b_N+b_N')$ is admissible. 
		Let $Z_1,\dots, Z_l$ be minimal orbits of $M_0$. By Lemma \ref{lemm:localizationminimal}, we have 
		\begin{equation*}
			\begin{split}
				\kappa_j^*([\omega^{G_0}_{b_1,\dots, b_N}]) &= [q_0^*\circ \widetilde{\Phi}_{b_1,\dots, b_N} (Z_j)],\\
				\kappa_j^*([\omega^{G_0}_{b_1',\dots, b_N'}]) &= [q_0^*\circ \widetilde{\Phi}_{b_1',\dots, b_N'} (Z_j)],\\
				\kappa_j^*([\omega^{G_0}_{b_1+b_1',\dots, b_N+b_N'}]) &= [q_0^*\circ \widetilde{\Phi}_{b_1+b_1',\dots, b_N+b_N'} (Z_j)].
			\end{split}
		\end{equation*}
		On the other hand, since $P_{b_1,\dots, b_m} +P_{b_1',\dots, b_m'} = P_{b_1+b_1',\dots, b_m+b_m'}$ and 
		\begin{equation*}
			\begin{split}
				\widetilde{\Phi}_{b_1,\dots, b_N} (M_0) &= P_{b_1,\dots, b_m},\\
				\widetilde{\Phi}_{b_1',\dots, b_N'} (M_0) &= P_{b_1',\dots, b_m'},\\
				\widetilde{\Phi}_{b_1+b_1',\dots, b_N+b_N'} (M_0) &= P_{b_1+b_1',\dots, b_m+b_m'},
			\end{split}
		\end{equation*}
		we have $\kappa_j^*([\omega^{G_0}_{b_1,\dots, b_N}]) + \kappa_j^*([\omega^{G_0}_{b_1',\dots, b_N'}]) = \kappa_j^*([\omega^{G_0}_{b_1+b_1',\dots, b_N+b_N'}])$ by Lemma \ref{lemm:vertexcorrespondence}. By Corollary \ref{coro:localization} we have $[\omega^{G_0}_{b_1,\dots, b_N}] + [\omega^{G_0}_{b_1',\dots, b_N'}] = [\omega^{G_0}_{b_1+b_1',\dots, b_N+b_N'}]$. Applying the basic forgetful map $\forB\co H^*_{G_0}(M_0) \to H^*_B(M_0)$ to the above, we have $[\omega_{b_1,\dots, b_N}] + [\omega_{b_1',\dots, b_N'}] = [\omega_{b_1+b_1',\dots, b_N+b_N'}]$, proving Part (1). 
		
		For (2), let $(b_1,\dots, b_N) \in \R^N$ be admissible and $r$ a real positive number. Let $i =1,\dots, N$. Since $\langle \widetilde{x},\mu_i \rangle \geq b_i$ if and only if 
		$\langle r\widetilde{x},\mu_i \rangle \geq rb_i$, we have that $(rb_1,\dots, rb_N)$ is admissible and $rP_{b_1,\dots, b_m} = P_{rb_1,\dots, rb_m}$. By Lemma \ref{lemm:localizationminimal}, we have $\kappa_j^*([\omega^{G_0}_{rb_1,\dots, rb_N}]) = [q_0^*\circ \widetilde{\Phi}_{rb_1,\dots, rb_N}(Z_j)]$. By Lemma \ref{lemm:vertexcorrespondence} and the fact that $rP_{b_1,\dots, b_m} = P_{rb_1,\dots, rb_m}$, we have $\kappa_j^*([\omega^{G_0}_{rb_1,\dots, rb_n}]) = r\kappa_j^*([\omega^{G_0}_{b_1,\dots, b_n}])$. Thus by Corollary \ref{coro:localization} we have $[\omega^{G_0}_{rb_1,\dots, rb_n}] = r[\omega^{G_0}_{b_1,\dots, b_n}]$. Applying the basic forgetful map $\forB \co H^*_{G_0} (M_0) \to H^*_B(M_0)$, we have 
		$r[\omega_{b_1,\dots, b_N}] = [\omega_{rb_1,\dots, rb_N}]$, proving Part (2). 
		The lemma is proved. 
	\end{proof}
	\begin{prop}\Label{prop:M0Dolbeault}
		Let $M_0$ be as above. Then we have 
		\begin{equation*}
			H^{p,q}_B(M_0) = \begin{cases}
				0 & \text{if $p \neq q$}, \\
				H^{2p}_B(M_0)\otimes \C & \text{if $p=q$}.
			\end{cases}
		\end{equation*}
	\end{prop}
	\begin{proof}
		By Corollary \ref{coro:degree2}, it is enough to show that $H^{1,1}_B(M_0) = H^2_B(M_0)\otimes \C$. We show that $H^2_B(M_0)$ is generated by basic $(1,1)$-forms. Let $b = (b_1,\dots, b_N) \in \R^N$ be admissible. Since the set of all admissible elements in $\R^N$ is nonempty and open, we have that for $i=1,\dots, m$ there exists $\epsilon_i>0$ such that $b+\epsilon_ie_i$ is admissible, where $e_i$ denotes the $i$th standard basis vector of $\R^N$. For short, we denote by $P_b$ the polytope $P_{b_1,\dots, b_m}$. We define the basic $(1,1)$-form $\tau_i$ on $M_0$ by 
		\begin{equation*}
			\tau_i := \frac{1}{\epsilon_i}\left(\omega_{b+\epsilon_ie_i} - \omega_b\right)
		\end{equation*}
		for $i=1,\dots, m$. We define a linear map $L \co \R^m \to H^2_B(M_0)$ by 
		\begin{equation*}
			L(a_1,\dots, a_m) := \sum_{i=1}^m a_i[\tau_i], \quad (a_1,\dots, a_m) \in \R^m. 
		\end{equation*}
		Let $(a_1,\dots, a_m) \in \ker L$. We set 
		\begin{equation*}
			I_+ := \{ i \mid a_i> 0\}, \quad I_{-} := \{i \mid a_i <0\}.
		\end{equation*}
		Then we have 
		\begin{equation*}
			\sum_{i \in I_+} a_i[\tau_i] = \sum_{i \in I_-} -a_i[\tau_i].
		\end{equation*}
		By definition of $\tau_i$, we have 
		\begin{equation}\Label{eq:ai}
			\sum_{i \in I_+} \frac{a_i}{\epsilon_i}[\omega_{b+\epsilon_ie_i}] + \sum_{i \in I_-} -\frac{a_i}{\epsilon_i}[\omega_b] = \sum_{i \in I_+} \frac{a_i}{\epsilon_i}[\omega_{b}] + \sum_{i \in I_-} -\frac{a_i}{\epsilon_i}[\omega_{b+\epsilon_ie_i}]. 
		\end{equation}
		All coefficients in \eqref{eq:ai} are nonnegative. By Lemma \ref{lemm:monoid}, we have that 
		\begin{equation*}
			b' := \sum_{i \in I_+} \frac{a_i}{\epsilon_i}(b+\epsilon_ie_i) + \sum_{i \in I_-} -\frac{a_i}{\epsilon_i}b
		\end{equation*}
		and 
		\begin{equation*}
			b'' := \sum_{i \in I_+} \frac{a_i}{\epsilon_i}b + \sum_{i \in I_-} -\frac{a_i}{\epsilon_i}(b + \epsilon_ie_i)
		\end{equation*}
		both are admissible. By Lemma \ref{lemm:monoid} and \eqref{eq:ai}, we have $[\omega_{b'}] = [\omega_{b''}]$. By Lemma \ref{lemm:parallel}, there exists $\widetilde{y} \in V^*_0$ such that $P_{b'} = P_{b''} + \widetilde{y}$. By Lemma \ref{lemm:paralleltranslation}, we have $\langle \widetilde{y},\mu_i\rangle = b_i' - b_i''$ for $i=1,\dots, m$, where $b_i' \in \R$ and $b_i''\in \R$ are the $i$th entries of $b' \in \R^N$ and $b''\in \R^N$, respectively. On the other hand, we have 
		\begin{equation*}
			b' - b'' = \sum_{i \in I_+} a_ie_i + \sum_{i \in I_-}a_ie_i = \sum_{i=1}^m a_ie_i.
		\end{equation*}
		Therefore we have $a_i = \langle \widetilde{y},\mu_i\rangle$ for $i=1,\dots, m$. Thus we have
		\begin{equation*}
			\ker L \subset \{ (a_1,\dots, a_m) \mid ^\exists \widetilde{y} \in \widetilde{V}^*_0\text{ s.t.} ^\forall i, \langle \widetilde{y},\mu_i\rangle =a_i\}. 
		\end{equation*}
		Since $\dim \widetilde{V}^*_0 = N-k$, we have $\dim \ker L \leq N-k$. Thus we have $\dim \im L \geq m-k+N$. On the other hand, it follows from Proposition \ref{prop:DJ} that $\dim H^2_B(M_0) = m-k+N$. Thus $L \co \R^m \to H^2_B(M_0)$ is surjective. Since each element of $\im L$ is represented by a difference of $2$ positive $(1,1)$-forms, we have that $H^2_B(M_0)$ is generated by $(1,1)$-forms, as required. 
	\end{proof}
	\begin{proof}[Proof of Theorem \ref{theo:Dolbeault}]
		By Lemma \ref{lemm:M0M}, there exists an $\alpha$-equivariant holomorphic map $f \co M_0 \to M$ such that 
		\begin{enumerate}
			\item $\ker \alpha$ is connected and 
			\item $f \co M_0 \to M$ is a principal $\ker \alpha$-bundle. 
		\end{enumerate}
		By Proposition \ref{prop:equivalence}, $f$ induces an isomorphism $f^* \co H^*_B(M) \to H^*_B(M_0)$ of basic cohomologies. Since $f$ is holomorphic, we have that $f^*$ preserves the basic Hodge structures. This together with Proposition \ref{prop:M0Dolbeault} yields that 
		\begin{equation*}
			H^{p,q}_B(M) = \begin{cases}
				0 & \text{if $p \neq q$}, \\
				H^{2p}_B(M)\otimes \C & \text{if $p=q$},
			\end{cases}
		\end{equation*}
		as required.
	\end{proof}
\section*{Acknowledgements}
	The author would like to thank Roman Krutovskiy and Mikiya Masuda for their valuable comments on the earlier version. He also would like to thank Oliver Goertsches, Hiraku Nozawa and Dirk T\"oben for explaining how their results relate to results in this paper.

\end{document}